\documentclass[a4paper]{amsart}

\pdfoutput=1

\usepackage[utf8]{inputenc}
\usepackage[T1]{fontenc}
\usepackage{lmodern}
\usepackage{amsthm, amssymb, amsmath, amsfonts, mathrsfs}
\usepackage{overpic}

\usepackage{microtype}

\usepackage[pagebackref,colorlinks=true,pdfpagemode=none,urlcolor=blue,
linkcolor=blue,citecolor=blue]{hyperref}

\usepackage{amsmath,amsfonts,amssymb,amsthm}
\usepackage{amsthm, amssymb, amsmath, amsfonts, mathrsfs}

\usepackage{mathrsfs}
\usepackage{MnSymbol}
\usepackage{color}
\usepackage{accents}

\usepackage{enumerate}

\usepackage{bbm}

\usepackage{cleveref}

\newtheorem{theorem}{Theorem}[section]
\newtheorem{lemma}[theorem]{Lemma}

\newtheorem{proposition}[theorem]{Proposition}

\theoremstyle{remark}

\renewenvironment{proof}[1][Proof]{ {\itshape \noindent {#1.}} }{$\Box$
\medskip}

\numberwithin{equation}{section}
\newcommand{\R}{\mathbb{R}}

\newcommand{\N}{\mathbb{N}}

\newcommand{\E}{\mathbb{E}}

\newcommand{\eps}{\varepsilon}
\def\les{\lesssim}

\newcommand{\1}{\mathbbm{1}}

\newcommand{\cP}{\mathcal{P}}

\newcommand{\dx}{\, dx}
\newcommand{\dy}{\, dy}
\newcommand{\dz}{\, dz}

\DeclareMathOperator{\supp}{supp}

\definecolor{darkmagenta}{rgb}{1,0,1}

\newcommand{\thc}{\theta_{\rm crit}}
\newcommand{\cc}{c_{\rm crit}}

\newcommand{\be}{\begin{equation}}
\newcommand{\ee}{\end{equation}}

\begin{document}

\title[Long-time behavior for a nonlocal model]{Long-time behavior for a nonlocal model from directed polymers}
\author{Yu Gu, Christopher Henderson}

\address[Yu Gu]{Department of Mathematics, University of Maryland, 
College Park, MD 20742}

\address[Christopher Henderson]{Department of Mathematics, University of Arizona, Tucson, AZ 85721}

\maketitle

\begin{abstract}
We consider the long time behavior of solutions to a nonlocal reaction diffusion equation that arises in the study of directed polymers in a random environment.  The model is characterized by convolution with a kernel $R$ and an $L^2$ inner product.  In one spatial dimension, we extend a previous result of the authors [arXiv:2002.02799], where only the case $R =\delta$ was considered; in particular, we show that solutions spread according to a $2/3$ power law consistent with the KPZ scaling conjectured for directed polymers.  In the special case when $R = \delta$, we find the exact profile of the solution in the rescaled coordinates.  We also consider the behavior in higher dimensions.  When the dimension is three or larger, we show that the long-time behavior is the same as the heat equation in the sense that the solution converges to a standard Gaussian.  In contrast, when the dimension is two, we construct a non-Gaussian self-similar solution.
%
%
%
\end{abstract}

%

\maketitle

\section{Introduction}

In this paper, we investigate the following collection of models:
\begin{equation}\label{e.main}
	\begin{cases}
		\partial_t g
			= \frac{1}{2} \Delta g
			+ g \left( \langle R * g, g\rangle - R*g\right)
				\qquad &\text{ in } (0,\infty)\times \R^d,\\
		g = g_0
				\qquad &\text{ on } \{0\} \times \R^d,
	\end{cases}
\end{equation}
where $R*g$ refers to convolution in the spatial variables only and the brackets $\langle \cdot, \cdot\rangle$ denote the $L^2$ inner product in the spatial variables.  We always assume that
\be\label{e.g_0}
	\int g_0(x) \dx = 1
		\qquad\text{ and }\qquad
		0 \leq g_0 \in C_c(\R^d),
\ee
which implies that, for all $t>0$,
\be
	\int g(t,x) dx = 1;
\ee
thus, \eqref{e.main} describes the evolution of a probability density.  Here, either $R$ is the delta distribution $\delta$ or $R$ is a continuous, nonnegative function such that $\int R \dx = 1$. 
In the latter case, we assume that there is a continuous even function $\phi \geq 0$ such that
\be\label{e.R}
	R(x)
		= \phi*\phi(x), \quad \text{ and } \int \phi(x)\dx=1.
\ee


\subsubsection*{The connection to directed polymers.} The equation arises from our study of directed polymers in a random environment, and we discuss the model below. For a Gaussian random field $\{V(t,x): (t,x)\in\R^{d+1}\}$ and an independent Brownian motion $B=\{B_t: t\geq0\}$, consider the Gibbs measure associated with the Hamiltonian $H_t(B)=\int_0^t V(s,B_s)ds$:
\[
\begin{aligned}
&\mu_t(dx)=q(t,x)\dx,\\
&q(t,x)=Z_t^{-1}\E_B[\delta(B_t-x)\exp(H_t(B))],\quad\quad\quad Z_t=\E_B[\exp(H_t(B))].
\end{aligned}
\]
Here $\E_B$ is the expectation only with respect to the Brownian motion $B$, with the random Gaussian field fixed. Thus, $\mu_t(dx)$ can be viewed as the endpoint distribution of the Brownian motion ``reweighted'' by the random environment, through the factor $\exp(H_t(B))$, which is to model the polymer path in a heterogeneous environment. The properties of $q(t,\cdot)$, in particular how the $x$-variable scales with respect to time, is a notoriously difficult problem in probability and statistical physics. It is conjectured that, in $d=1$ and if $V$ is sufficiently short-range correlated, we should have $\int |x|^p q(t,x)\dx\approx t^{2p/3}$ for large $t$, i.e., the endpoint of the polymer path is superdiffusive with an exponent $2/3$, which falls into the KPZ universality class. So far the conjecture is only proved for a few specific models with certain integrable structures, and the proof is very much model-dependent. In $d\geq2$, much less is conjectured and known, and even the correct superdiffusive exponent is unclear. We refer to the monograph \cite{comets2017directed} and the survey \cite{corwin2012kardar} for results and discussions in this direction. 

Our interest is in the averaged density $Q_1(t,x):=\E[q(t,x)]$, where $\E$ is now taken with respect to the Gaussian field $V$, and the ultimate goal is to study the asymptotic behavior of $Q_1$ by a robust analytic approach that covers all possible correlation structures of $V$. For example, in $d=1$ one would like to prove a \emph{universality} result saying that $\int |x|^p Q_1(t,x)\dx\approx t^{2p/3}$ for $t\gg1$. As $Q_1$ is the averaged density of the endpoint of a random path, it is tempting to try to  derive a PDE for its evolution, similar to the Fokker-Planck equations associated with diffusion processes. This motivated the study in  \cite{GuHenderson}. For a large class of Gaussian fields $V$, which has zero mean and is white in time and possibly colored in space, with the covariance function 
\[
\E[V(t,x)V(s,y)]=\delta(t-s)R(x-y)\footnote{Since $R$ is the spatial covariance function, there exists a function $\phi$ so that \eqref{e.R} holds, and the construction can be found in \cite[Page 2]{GuHenderson}.},
\] we find that, instead of solving a single Fokker-Planck equation, what governs the evolution of $Q_1$ is a hierarchical system: define the $n-$point correlation function  
\[
Q_n(t,x_1,\ldots,x_n):=\E[\prod_{j=1}^n q(t,x_j)],
\] 
then $Q_1$ solves the equation
\begin{equation}\label{e.eqQ1}
\begin{aligned}
\partial_t Q_1(t,x)=\tfrac12\Delta Q_1(t,x)&-\int Q_2(t,x,y)R(x-y)\, dy\\
&+\int Q_3(t,x,y,z)R(y-z)\,dydz.
\end{aligned}
\end{equation}
In fact, for any $n\geq1$, the equation of $Q_n$ contains $Q_{n+1}$ and $Q_{n+2}$. The  nonlocal terms in \eqref{e.eqQ1} describe the mutual intersection of multiple polymer paths as they wander in the random environment to maximize the collected energy, and the kernel $R$ corresponds to how the paths' intersection is measured. The hierarchical PDE system is similar to the BBGKY hierarchy in kinetic theory. Inspired by the molecular chaos assumption there, we assume that in large time $Q_n$ can be approximately factorized: $Q_n(t,x_1,\ldots,x_n)\approx \prod_{j=1}^n Q_1(t,x_j)$, and this helps to reduce \eqref{e.eqQ1} to \eqref{e.main}. Therefore, the equation we study in this paper, can be viewed   as an approximation of the hierarchical system which describes the actual evolution of the polymer endpoint density. While  it is unclear at the moment how to justify the factorization assumption used to link the true evolution \eqref{e.eqQ1} with the ``approximate'' evolution \eqref{e.main}, the results in \cite{GuHenderson} already show an intriguing connection, which we discuss in greater detail below. Despite this simplification, we expect \eqref{e.main} to retain several key features of the original equation \eqref{e.eqQ1}, and furthering our understanding of \eqref{e.main} may also help with the study of the hierarchy.  This motivates the current study.

\subsubsection*{Rough description of the main results}
In~\cite[Theorem~1.3]{GuHenderson}, we examined the special case of~\eqref{e.main} in which $R = \delta$ and $d=1$, which corresponds to the situation when the Gaussian environment is also white in space. Here, we showed that the KPZ scaling, $x\sim t^{2/3}$ is exhibited in the sense that, $\|g(t)\|_{\infty} = \mathcal{O}(t^{-2/3})$ and, for any $p\geq 1$, the $p$th moment of $g$ is bounded from above and below by $t^{2p/3}$, up to a constant, i.e.,
\begin{equation}\label{e.10241}
\int |x|^p g(t,x)dx\approx t^{\frac{2p}{3}}, \quad\quad \mbox{ for } t\geq 1.
\end{equation}  As a result, it follows that $g \approx t^{-2/3}$ for any $|x| \leq \mathcal{O}(t^{2/3})$ and $g \ll t^{-2/3}$ for any $|x| \gg t^{2/3}$

Our focus in this paper is in generalizing and refining the above result in one dimension and in investigating the behavior of $g$ in higher dimensions.  Roughly, we establish the following properties of~\eqref{e.main}.
\begin{enumerate}[(1)]
\item When $R$ is continuous and $d=1$, we show that the KPZ-scaling, $x\sim t^{2/3}$, conjectured for the full polymer model is exhibited by $g$.  In particular, we prove that $\|g(t)\|_{\infty} = \mathcal{O}(t^{-2/3})$ and, for any $p\geq 0$, the $p$th moment of $g$ is bounded above and below by $t^{2p/3}$, as in \eqref{e.10241}.  This extends the results of~\cite[Theorem~1.3]{GuHenderson}, and can be viewed as a universality result as the scaling exponent $2/3$ does not depend on the detailed expression of $R$. As discussed below, significant difficulties arise in generalizing the proof from the case $R=\delta$ to the present one. 
\item When $R=\delta$ and $d=1$, we establish more precise estimates on the long-time behavior of $g$.  Specifically, we identify the limit of $t^{2/3} g(t, xt^{2/3})$ as $t\to\infty$ to be, up to a multiplicative constant, the indicator function of an interval. In other words, if we let $X_t$ be a random variable with the density $g(t,\cdot)$, with \eqref{e.10241} we have $\E[|X_t|^p]\approx t^{2p/3}$, and here we further prove that $t^{-2/3}X_t$ converges to a uniform distribution, as $t\to\infty$.

\item When $d\geq 2$, the behavior is quite different, as expected for the directed polymers. In all cases, the diffusive scaling $x \sim t^{1/2}$ and $g \sim t^{-d/2}$ holds.  When $d\geq 3$, we show that $g$, under this scaling,  converges to a standard Gaussian. This is exactly the behavior of the heat equation; in other words, the effect of the nonlinearity is negligible.  The result is consistent with the diffusive behaviors of directed polymers in high temperatures in $d\geq3$ \cite{bolt,spencer} \footnote{The polymer model actually depends on a parameter $\beta$ which is the inverse temperature, and it goes into our approximate model \eqref{e.main} only as a multiplicative constant of the nonlinear terms (see \cite[Equation (1.10)]{GuHenderson}). As it does not play a role in our analysis, we do not specify it here.}. On the other hand, when $d=2$, there are solutions $g$ of~\eqref{e.main} that, in the same diffusive scaling, do not converge to a Gaussian. Thus, $d=2$ is the critical dimension for \eqref{e.main}, another feature of the polymer model.
\end{enumerate}

We note that the local well-posedness (small time existence and uniqueness) of~\eqref{e.main} with initial data in, say, $L^1(\R^d)\cap L^2(\R^d)$ is straightforward because it is a semilinear equation where the nonlinear terms are ``relatively smooth'' in $g$. On the other hand, the {\em a priori} estimates that we obtain in the course of this work show that blow-up cannot occur and (unique) solutions can be extended for all time.  In other words, a proof of global well-posedness is straightforward using the estimates we establish below. Hence, we omit it.


\subsection{Main Results}

We now state our results more precisely.

\subsubsection{Decay consistent with the conjectured scaling for directed polymers}

Our first theorem is the extension of the results in~\cite[Theorem~1.3]{GuHenderson} on the growth of moments to the setting where $R$ is continuous.

\begin{theorem}\label{t.decay}
	Let $d=1$.  Suppose that $g$ solves~\eqref{e.main} with $g_0$ satisfying~\eqref{e.g_0} and with $R$ satisfying~\eqref{e.R} or $R = \delta$.  Then, for any $p>0$,
	\be\label{e.c711}
		\left(\int_\R |x|^p g(t,x) dx\right)^{1/p}
			\approx t^\frac{2}{3}
				\qquad\text{ for all } t\geq 1.
	\ee
	As a consequence, we have $\displaystyle \liminf_{t\to\infty} t^{2/3} \|g(t)\|_\infty, t^{2/3}\|g(t)\|_2^2 >0$.  The bounds here depend only on $\supp(g_0)$.
\end{theorem}

In the statement of the above theorem, $a\approx b$ means $C^{-1}b\leq a\leq Cb$ for some constant $C>0$. Similarly, we will use $a\les b$ if $a\leq Cb$. We collect the precise meaning of all other notations  in \Cref{s.notation} below.

As alluded to above, the proof of \eqref{e.c711} when $R=\delta$ is exactly the content of \cite[Theorem~1.3]{GuHenderson}; however, the lower bound on the rescaled $L^\infty$ and $L^2$ norms of $g$ are new.  These are not difficult to prove after the other moment bounds have been established, but are stated here as these lower bounds are an essential ingredient in the establishing the precise behavior of $g$ (see \Cref{t.strong}).

The main new content of \Cref{t.decay} are the moment bounds in the case where $R$ is continuous.  The key step in proving this is the following $L^\infty$ bound on $g$.
\begin{proposition}\label{p.decay}
	Under the assumptions of \Cref{t.decay}, for all $t\geq 0$,
	\[
		\|g(t)\|_{\infty}
			\lesssim \min\{t^{-1/2}, t^{-2/3}\}.
	\]
\end{proposition}

While the connection between \Cref{p.decay} and \Cref{t.decay} is analogous to that in the case $R=\delta$, the proof of \Cref{p.decay} is significantly more difficult, which we discuss now.

The proof in the case $R=\delta$ contained in \cite[Theorem~1.3]{GuHenderson} relies on several exact identities that are no longer available.  In fact, the first step when $R=\delta$ is noticing that~\eqref{e.main} yields
\[
	\frac{d}{dt} \|g(t)\|_{\infty}
		\leq \|g(t)\|_{\infty} ( \langle R* g, g\rangle - R*g)
		= \|g(t)\|_{\infty} ( \|g\|_2^2 - \|g(t)\|_{\infty}).
\]
The inequality follows from the fact that $\Delta g$ is nonpositive at a maximum, and the equality follows from using that $R=\delta$.  By H\"older's inequality and the fact that $g(t)$ is a probability measure, the right hand side is clearly nonpositive.  It is then a matter of quantifying this non-positivity (cf.~\cite[Lemma~4.3]{GuHenderson}).  Unfortunately, when $R$ is continuous, the above equality does not hold.  In fact, it is not even clear if $\|g(t)\|_{\infty}$ is decreasing with respect to $t$.

To overcome this difficulty, we may try to work, instead, with the $L^2$-norm of $g(t,\cdot)$.  When multiplying~\eqref{e.main} by $g$, and integrating by parts, we obtain
\be\label{e.c1011}
	\frac{1}{2} \frac{d}{dt} \|g(t)\|^2_2
		+ \frac{1}{2} \int |\nabla g|^2 dx
		= \|g\|^2_2 \langle R* g, g\rangle - \int g^2 R*g\, dx.
\ee
When $R = \delta$, the right hand side has the form $\|g\|^4_2 - \|g\|_3^3$, which, using H\"older's inequality, one easily sees is nonpositive.  
However, when $R$ is continuous, it is no longer clear that the right hand side is even nonpositive.  Again, it may not be.

Given the convolution term $R*g$ appearing in~\eqref{e.main}, a better approach is to multiply the equation by $R*g$ and integrate by parts.  Then~\eqref{e.main} becomes
\[
	\frac{1}{2} \frac{d}{dt} \langle g, R*g\rangle
		+ \frac{1}{2} \int \nabla g \cdot \nabla(R*g) dx
		= \langle R*g, g\rangle^2 - \int g (R*g)^2 dx.
\]
The right hand side is once again nonpositive.  The goal of the proof is to establish a lower bound on
\be\label{e.c1012}
	\frac{1}{2} \int \nabla g \cdot \nabla(R*g) dx
		+ \left[- \langle R*g, g\rangle^2 + \int g (R*g)^2 dx\right].
\ee
After, we then must relate $\langle g,R*g\rangle$ back to $\|g\|_2$ and $\|g\|_\infty$.

The main tool in the analysis is the local-in-time Harnack inequality (see \Cref{p.harnack}) that was first established in \cite{BHR_toads_delay}.  This inequality, which quantifies how spread out the level sets of $g(t,\cdot)$ are, allows us to show that, roughly, either $g$ is flat, in which case the bracketed terms in~\eqref{e.c1012} are large, or $g$ is not flat and the gradient term in~\eqref{e.c1012} is large.

Interestingly, while the functional inequalities used in~\cite{GuHenderson} to bound~\eqref{e.c1012} from below in the $R=\delta$ case held for {\em any} $H^1$ function, the lower bound on~\eqref{e.c1012} in the case where $R$ is continuous relies strongly on the regularizing effect of the heat equation, seen through the local-in-time Harnack inequality.  In some sense we are showing that the key functional inequality \cite[Lemma~4.3]{GuHenderson} (see also \cite[Lemma 2]{Constantin_2000}) is stable with respect to convolutions {\em as long as} the function is suitably regular.

The proofs of \Cref{t.decay} and \Cref{p.decay} are contained in \Cref{s.R_continuous}.


\subsection{Long time dynamics in one dimension}

Restricting to the case $R=\delta$, we investigate the behavior of $g$ under the $t^{2/3}$ scaling.  In particular, \Cref{t.decay} suggests that we may see non-trivial limiting behavior of $t^{2/3} g(t, x t^{2/3})$ as $t\to\infty$.  We establish that here.

In order to describe the long-time behavior, we define two important constants:
\be\label{e.cc_thc}
	\cc = \left(\frac{3}{2}\right)^{2/3}
		 =  1.31037...
		\qquad\text{and}\qquad
	\thc = \left(\frac{1}{18}\right)^{1/3}
		 =  .38157...
\ee
Notice that $2 \thc \cc = 1$.  The precise meaning of $\cc$ and $\thc$ is touched on below and then described in detail in \Cref{s.heuristic}. We now state our main results.

\begin{theorem}\label{t.strong}
	Suppose that $g$ solves~\eqref{e.main} with initial data $g_0$ satisfying~\eqref{e.g_0} and $R=\delta$.  Suppose further that $g_0$ is even and radially decreasing.  Then
		\[
			\lim_{t\to\infty} t^{2/3}\|g(t)\|^2_{2}
			= \lim_{t\to\infty} t^{2/3}\|g(t)\|_{\infty}= \thc.
		\]
	Further, for all $x \neq \pm \cc$,
		\[
			\lim_{t\to\infty} t^{2/3} g(t,x t^{2/3})
				= \thc\1_{[-\cc, \cc]}(x),
		\]
		where the limit holds uniformly for $x$  away from $\pm \cc$.
\end{theorem}

Informally, \Cref{t.strong} implies that, for large $t$, we may write
\[
	g(t,x)
		 =  \frac{\thc}{t^{2/3}} \1_{[-\cc t^{2/3}, \cc t^{2/3}]}(x) + o(t^{-2/3}).
\]
It is clear that this is consistent with the results in \Cref{t.decay}.

We note that the convergence of the (rescaled) $L^2$- and $L^\infty$-norms is {\em almost} equivalent to the convergence of the profile.  Indeed, since $g$ is a probability density, its $L^2$- and $L^\infty$-norms can only be equal when the H\"older's inequality is an equality, which corresponds to functions that are a multiplicative constant of an indicator function.  Using then the symmetry and monotonicity of $g$, it follows that $g$ converges to the indicator function of an interval up to a multiplicative constant (though we note that this does not yield the exact constants $\cc$ and $\thc$).  Of course, we never have exact equality and thus, the above heuristics hinge on understanding the ``stability'' of H\"older's inequality.

We believe that the assumption that $g_0$ is even and decreasing is purely technical and can be removed at the expense of a significantly more involved proof.  As the proof is already quite complicated even with these assumptions, we opt to use them.  Throughout the proof of \Cref{t.strong}, we indicate where and how these assumptions are applied.

A key step in the proof of \Cref{t.strong} are the following bounds, which do not require these symmetry and monotonicity assumptions.

\begin{proposition}\label{p.weak}
	Suppose that $g$ solves~\eqref{e.main} with initial data $g_0$ satisfying~\eqref{e.g_0}.   Then
	\begin{enumerate}[(i)]
		\item  $\displaystyle \liminf_{t\to\infty} t^{2/3}\|g(t)\|_{2}^2, \ \liminf_{t\to\infty} t^{2/3}\|g(t)\|_{\infty} \leq \thc$;
		\item  $\displaystyle \thc \leq \limsup_{t\to\infty} t^{2/3}\|g(t)\|_{2}^2,\ \limsup_{t\to\infty}  t^{2/3}\|g(t)\|_{\infty}$.
	\end{enumerate}
\end{proposition}

The results in \Cref{p.weak} show that $t^{2/3}\|g(t)\|_{\infty}$ and $t^{2/3}\|g(t)\|^2_{2}$ get arbitrarily close to $\thc$ infinitely often; however, they do not rule out the possibility that these quantities make non-trivial oscillations around $\thc$.  We recall the discussion following \Cref{t.strong} that indicates the importance of the $L^2$- and $L^\infty$-norms in understanding the profile of $g$.

The intuition behind the constants $\thc$ and $\cc$ comes from various rescalings of the equation in which a non-local Fisher-KPP type equation arises.  Finding and heuristically interpreting the correct rescalings is a subtle issue and so is covered in detail in \Cref{s.heuristic}.  A discussion of the strategy of the proof and the major difficulties encountered is also contained there as it is best placed in the setting of the rescaled equations.

We note that a major difficulty in extending \Cref{t.strong} to the case when $R\neq \delta$ is that the comparison principle no longer holds for~\eqref{e.main} when $R\neq \delta$.

%
%
%

The proofs of \Cref{p.weak} and \Cref{t.strong} are contained in \Cref{s.weak} and \Cref{s.strong} respectively.

\subsection{Long time dynamics in higher dimensions}

We now discuss the behavior in higher dimensions.  Our first result is about the time decay of the moments of $g$, as well as the $L^2$- and $L^\infty$-norms of $g$.

\begin{theorem}\label{t.hd_moments}
	Suppose $d\geq 2$, $R = \delta$ or $R$ satisfies~\eqref{e.R}, and $g$ solves~\eqref{e.main} with initial data $g_0$ satisfying~\eqref{e.g_0}.
	Then, for any $p>0$
	\[
		\left(\int_\R |x|^p g(t,x) dx\right)^{1/p}
			\approx t^\frac{1}{2}
				\qquad\text{ for all } t\geq 1.
	\]
	In addition, we have
	\[
		\|g(t)\|_\infty, \|g(t)\|_2^2
			\approx t^{-\frac{d}{2}}
				\qquad \text{ for all } t \geq 1.
	\]
\end{theorem}

The proof of \Cref{t.hd_moments} uses classical techniques based on the Nash inequality, which, in its original form relates the $L^1$, $L^2$, and $\dot H^1$ norms of a function, in order to bound $\langle g, R*g\rangle$.  In fact, we slightly extend the Nash inequality to apply to the quantities $\langle g, R*g\rangle$ and $\langle \nabla g, R*\nabla g\rangle$ in place of the $L^2$ and $\dot H^1$ norms, though the proof is analogous to the usual one.  To understand why these convolved macroscopic quantities are more useful than the $L^2$ and $\dot H^1$ norms, we refer to the discussion around~\eqref{e.c1011} and~\eqref{e.c1012}.

To bootstrap the bound on $\langle g, R*g\rangle$ to one on the moments of $g$, we construct a barrier function $\overline g$ such that, if $g$ and $\overline g$ were to touch, at the touching point
\[
	\partial_t \overline g - \frac{1}{2} \Delta \overline g
		- \overline g \left( \langle g, R*g\rangle - R*g\right)
		> 0,
\]
which rules out any touching points.

We now investigate the self-similar behavior of $g$ for large times.  We show that, when $d\geq 3$, the nonlinear terms are asymptotically negligible and $g$ has the same Gaussian behavior as the usual heat equation.  On the other hand, when $d=2$, we show that this need not be true by constructing solutions of~\eqref{e.main} that do not have a Gaussian profile in rescaled variables.

\begin{theorem}\label{t.hd_steady}
Suppose that $R = \delta$.
\begin{enumerate}[(i)]
	\item If $d\geq 3$, $g$ solves~\eqref{e.main} with $g_0$ satisfying $0\leq g_0(x) \leq A e^{-|x|^2/ B}$ for some $A, B > 0$ and all $x\in\R^d$ and $\int g_0\, dx = 1$, then
		\[
			\limsup_{t\to\infty}
				t^{r_d} \left\|t^{d/2} g(t, x \sqrt t) - \frac{1}{(2\pi)^{d/2}} e^{- \frac{|x|^2}{2}}\right\|_\infty
				<\infty
		\]
		where $r_d = 1$ if $d>3$ and $r_d$ is any element of $(0,1)$ if $d=3$.
	\item If $d=2$, there exists $0 \leq G \in C^\infty(\R^2)$ such that, letting
	\[
		g(t,x)	
			= \frac{1}{t} G\left(\frac{x}{\sqrt t}\right),
	\]
	the following holds: $G$ is not a Gaussian (that is, $G(x) \neq e^{-|x|^2/2\sigma^2}/(2\pi\sigma^2)$ for any $\sigma>0$), $g$ solves~\eqref{e.main}, there exist $A,B>0$ such that $G(x) \leq A e^{-|x|^2/B}$ for all $x \in \R^2$, and
	\[
		\int g(t,x) \dx
			= \int G(x) \dx
			= 1
			\qquad\text{for all } t>0.
	\]
\end{enumerate}
\end{theorem}

The first step in proving \Cref{t.hd_steady} is to convert to the self-similar coordinates suggested by its statement.  Letting $\tilde g$ be $g$ in these new coordinates $(\tau,y)=(\log t,\tfrac{x}{\sqrt{t}})$, and $\tilde{g}(\tau,y)
		= e^{\frac{d}{2}\tau} g\left(e^\tau, e^{\tau/2} y\right)$, we see that
\be\label{e.c1084}
	\partial_\tau \tilde g
		= \left[\frac{1}{2}\Delta \tilde g
			+ \frac{y}{2} \cdot \nabla \tilde g
			+ \frac{d}{2} \tilde g \right]
			+ e^{-\frac{d-2}{2} \tau} \tilde g \left( \|\tilde g\|_{2}^2 - \tilde g\right).
\ee
The difference between $d=2$ and $d\geq 3$ is clear from the above equation.  When $d\geq 3$, the last term is exponentially decaying and we proceed by analyzing the spectrum of the operator in brackets, which is well-understood.

When $d=2$, the last term is non-negligible.  The construction then proceeds by finding a steady solution of~\eqref{e.c1084}.  To begin, we pose the problem on a ball of radius $r$ and examine the local (and slightly less nonlinear) problem where $\|\tilde g\|_{2}^2$ is replaced by a constant $E$.  After finding a solution $\tilde g_E$ to this problem, we show that there is a critical value of $E$ where $\|\tilde g_E\|_{2}^2$ is equal to $E$.  A difficulty with this is that the dependence of $\tilde g$ on $E$ is monotonic; that is, the larger $E$ is, the larger $\tilde g_E$ is.  Hence, it is difficult to simply look at small and large $E$ and show that the ordering of $E$ and $\|\tilde g_E\|_{2}^2$ switches.  We overcome this by showing that the operator in brackets in~\eqref{e.c1084} induces sufficient decay away from $x=0$ to limit the growth of $\|\tilde g_E\|_{2}^2$ as $E$ is increased.  After finding this critical $E$ value, the proof is concluded by taking $r\to\infty$.

We make two comments on the limitations of \Cref{t.hd_steady}.  First, we do not handle the case when $R$ is continuous.  In \Cref{t.hd_steady} (i), it is trivial to extend our proof to that case since the nonlinear terms are exponentially decaying in the self-similar variables.  We believe that \Cref{t.hd_steady} (ii) can also be extended to the case when $R$ is continuous albeit with more technical proofs using elliptic regularity theory.  However, our construction of $G$ is already quite involved  and we opt instead for a clearer, more succinct construction.

Second, we do not address the stability of $G$; that is, we do not have a convergence result of $\tilde g$ to $G$ when $d=2$ as we do in \Cref{t.hd_steady} (i) with the Gaussian.  The spectral theory based argument of part (i) does not apply to the stability of $G$ when $d=2$ because, as shown in~\eqref{e.c1084}, the nonlinear term plays a crucial role in the equation, thus, any stability result must use the nonlinearity in an essential way.   The initial difficulty in establishing the stability of $G$ is that $\|G\|_{2}$ is unknown and, it is not clear why $\|\tilde g(\tau)\|_{2}$ converges to a constant (and that that constant is $\|G\|_{2}$).  Indeed, multiplying~\eqref{e.c1084} by $\tilde g$ and integrating by parts yields
\[
	\frac{1}{2} \frac{d}{d\tau} \|\tilde g(\tau)\|_{2}^2
		= - \frac{1}{2} \|\nabla \tilde g(\tau)\|_{2}^2
			+ \frac{1}{2} \|\tilde g(\tau)\|_{2}^2
			+ \|\tilde g(\tau)\|_{2}^4
			- \|\tilde g(\tau)\|_{3}^3.
\]
Unfortunately, it is not clear that this would yield convergence of the $L^2$-norm of $\tilde g$.  For example, there is no obvious monotonicity of $\|\tilde g(\tau)\|_2$ imparted by the equation above.  Until a better understanding of the fluctuations or convergence of $\|\tilde g(\tau)\|_{2}$ is gained, stability remains open, although we believe that $G$ is stable.

The proofs of \Cref{t.hd_moments} and \Cref{t.hd_steady} are contained in \Cref{s.higher_dimensions}.

\subsubsection*{Connection (or lack thereof) with the nonlocal Fisher-KPP equation}

One might think that~\eqref{e.main} appears, on its face, to be similar to the non-local Fisher-KPP  that has considered in a huge number of works (e.g.,~\cite{ABVV, BerestyckiNadinPerthameRyzhik, Britton, HamelRyzhik} and articles referencing these):
\be\label{e.nlfkpp}
	\partial_\tau u
		= D \Delta u + u(r - \phi * u),
\ee
for some non-negative function $\phi$ and constants $D, r>0$.  However, there is very little connection one can draw between them.

Most obvious is the fact that their qualitative behavior is extremely different.  For example, the $L^1$ norm of solutions $g$ of~\eqref{e.main} is conserved, while there are no conserved quantities for solutions $u$ of~\eqref{e.nlfkpp}.  From a larger perspective, one sees that the main questions regarding each model are very different: \eqref{e.nlfkpp} is a model for front propagation leading to questions about the existence, stability, and qualitative properties of traveling wave solutions (solutions of~\eqref{e.nlfkpp} made up of a fixed profile in $x$ being translated at a constant speed), while the main questions for~\eqref{e.main} are about quantitative $L^p$ and moment bounds as well as self-similar behavior.

Finally, the mathematical techniques applied to each are necessarily unrelated.  As with the standard (local) Fisher-KPP equation, the basic behavior of solutions $u$ of~\eqref{e.nlfkpp} can be obtained by linearizing the equation around $0$, in which case the nonlocal term disappears.  In this sense, the large scale features of $u$ are {\em linearly determined}.  On the other hand, a linearization of~\eqref{e.main} around zero yields the heat equation, which does not yield the conclusions of the main theorems \Cref{t.decay} and \Cref{t.strong}.  In this sense, the dominant behavior of~\eqref{e.main} is {\em nonlinearly determined}.

\subsection{Notation} \label{s.notation}
Throughout the manuscript we use the notation $\lesssim$ for the following: $A \lesssim B$ if there exists $C>0$ such that $A \leq C B$, where $C$ is any constant that does not depend on $g$ (except for possibly on $\supp(g_0)$).  We write $A\approx B$ to mean that $A\lesssim B$ and $B \lesssim A$.

All $L^p$ norms are taken with respect to the spatial variable only unless explicitly indicated.  Hence, $\|g\|_{p}$ and $\|g(t)\|_{p}$ both refer to
\[
	\left( \int_{\R^d} g(t,x)^p\, dx \right)^{1/p},
\]
when $p \in [1,\infty)$.  The analogous notation is used when $p=\infty$.

Various quantities play a special role in our analysis.  We adopt the following notation: for any measurable function $f:[0,\infty)\times \R^d\to [0,\infty)$, we denote
\be\label{e.macro_notation}
	\begin{split}
		&E_f(t) = \langle f(t), R*f(t)\rangle
			= \|\phi * f(t)\|_2^2,\\
		&D_f(t) = \langle \nabla f(t), \nabla R*f(t)\rangle
			= \|\nabla \phi * f(t)\|_2^2,
		\quad\text{ and}\\
		&M_f(t) = \|f(t)\|_\infty,
	\end{split}
\ee
where we used~\eqref{e.R} to get the relationship between the first and second characterizations of $E_f$ and $D_f$ (this also uses that convolutions can pass between functions in the $L^2$-inner product; see \eqref{e.convolution_identity}).  We point out that, when $R=\delta$, $E_f=\|f(t)\|_2^2$ and $D_f = \|\nabla f(t)\|_2^2$.

When writing $\lim$, $\liminf$, and $\limsup$, we often omit the notation regarding the variable when no confusion will arise.  For example, the conclusion of \Cref{p.weak} (i) can be written
\[
	\liminf t^{2/3} \|g\|_{2}^2
		= \liminf t^{2/3} \|g\|_{\infty} \leq \thc.
\]


We use $B_r(x)$ to mean a ball of radius $r>0$ centered at $x$ in the spatial variables.  When the ball is centered at the origin, we simply write $B_r$ in place of $B_r(0)$.

\section{The $2/3$ power law when $R$ is continuous.}\label{s.R_continuous}


%
%
%
%
%

Before beginning the proof, we notice that, due to H\"older's inequality
\begin{equation}\label{e.c2274}
	E_g \leq M_g,
\end{equation}
since $\int R*g  = 1$, which comes from the assumption $\int R=1$ and the fact that $\int g=1$.  In addition, by Young's inequality for convolutions,
\be\label{e.c6304}
	E_g
		\leq (\|g\|_2 \|\phi\|_1)^2
		= \|g\|_2^2.
\ee

A key aspect of the proof is understanding the precise relationship between $g$, $\phi * g$, and $R*g$.  As such, it is useful to define more succinct notation for the latter two functions.  Let
\be\label{e.uw}
	u
		= \phi * g
	\quad\text{ and }\quad
	w = \phi * \phi * g = R*g.
\ee


We first show how to deduce \Cref{t.decay} from \Cref{p.decay} in the following subsection.  Afterwards, in \Cref{s.E_g_M_g_decay} and the following subsections, we prove \Cref{p.decay}.  This is where the bulk of the work is undertaken.

We note that, even when not explicitly mentioned, we assume that $d=1$ and $R$ is continuous and satisfies~\eqref{e.R} throughout this section.

\subsection{The proof of \Cref{t.decay} from \Cref{p.decay}}

We establish the bounds on the moments via arguments very similar to \cite[Theorem~1.3]{GuHenderson}; however, the slight alterations in the method here allows us to reduce the dependence of the estimates on $g_0$ to only  on $\supp(g_0)$.  When possible, we defer to the arguments in \cite[Theorem~1.3]{GuHenderson} and omit them here.

\begin{proof}[Proof of \Cref{t.decay}]
First, we obtain a pointwise upper bound on $g$.  We have, from \Cref{p.decay}, that
\[
	g(t,\cdot) \lesssim t^{-2/3};
\]
however, we require estimates on $g$ when $|x| \gtrsim t^{2/3}$.  To this end, let $L>0$ be such that $\supp(g_0) \subset [-L,L]$, and define
\[
	\overline g(t,x) = e^{\int_0^t E_g(s)\, ds} h(t,x),
\]
where $h$ is the solution of
\[
	\begin{cases}
		h_t = \frac{1}{2} \Delta h
			\qquad &\text{ in } (0,\infty)\times \R,\\
		h = g_0
			\qquad &\text{ on } \{0\}\times \R.
	\end{cases}
\]
It is straightforward to check that
\be\label{e.c10121}
	\partial_t \overline g
		= \frac{1}{2} \Delta \overline g + E_g \overline g.
\ee
While~\eqref{e.main} does not enjoy the comparison principle,~\eqref{e.c10121} does.  In addition, the non-negativity of $g$ and~\eqref{e.main} ensure that $g$ is a subsolution of~\eqref{e.c10121}.  Thus, the comparison principle implies that $g \leq \overline g$.  We deduce that, for all $t>0$,
\be\label{e.c8101}
\begin{split}
	g(t,x)
		&\leq \overline g(t,x)
		= e^{\int_0^t E_g(s)\, ds} \int_{\supp(g_0)} \frac{1}{\sqrt{2\pi t}}e^{- \frac{(x-y)^2}{2t}} g_0(y) \dy\\
		&\leq e^{\int_0^t E_g(s)\, ds} \int_{\supp(g_0)} \frac{1}{\sqrt{2\pi t}}e^{- \frac{x^2}{4t} + \frac{y^2}{t}} g_0(y) \dy\\
		&\leq \frac{e^{\int_0^t E_g(s)\, ds - \frac{x^2}{4t} + \frac{L^2}{t}}}{\sqrt{2\pi t}} \int_{\supp(g_0)} g_0(y) \dy
		= \frac{e^{\int_0^t E_g(s)\, ds - \frac{x^2}{4t} + \frac{L^2}{t}}}{\sqrt{2\pi t}}.
\end{split}
\ee
In the first equality, we used the kernel representation of solutions to the heat equation, and in the second inequality, we used Young's inequality in the exponent.

Applying \Cref{p.decay} and the fact that $E_g \leq M_g$, we find $C>0$ such that, if $t\geq 1$, then $\int_0^t E_g ds \leq C t^{1/3}$.  Using this in~\eqref{e.c8101} yields, for any $|x| \geq \sqrt{8C}t^{2/3}$,
\[
	g(t,x)
		\lesssim e^{C t^{1/3} - \frac{x^2}{8t} - \frac{x^2}{8t} + \frac{L^2}{t}}
		\leq e^{C t^{1/3} - \frac{\left|\sqrt{8 C} t^{2/3}\right|^2}{8t} - \frac{x^2}{8t} + \frac{L^2}{t}}
		= e^{ - \frac{x^2}{8t} + \frac{L^2}{t}}.
\]
Using this and \Cref{p.decay}, we obtain the crucial estimate:
\[
	g(t,x)
		\lesssim
			\begin{cases}
				t^{-2/3} \qquad &\text{ if } |x| \leq \sqrt{8C} t^{2/3},\\
				e^{\frac{L^2}{t} - \frac{|x|^2}{8t}}
					\qquad & \text{ otherwise.}
			\end{cases}
\]
A direct computation using this upper bound yields the upper bounds  on $\int |x|^p g(t,x) \dx$ for any $p>0$ and $t\geq 1$.  The proof of the lower bound is exactly as in \cite[Theorem~1.3]{GuHenderson} (that is, it uses the upper bound and a variational argument, see \cite[Lemma~4.2]{GuHenderson}).  As such, we omit the details.

The last step is, thus, to obtain lower bounds on the $L^2$- and $L^\infty$-norms of $g(t,\cdot)$.  By H\"older's inequality, $\|g(t)\|_2^2 \leq \|g(t)\|_\infty$ (recall that $\|g(t)\|_1 =1$). Hence, it is sufficient to find a lower bound on $\|g(t)\|_2$ in order to finish the claim.

To this end, we utilize the previously-established moment bounds.  Fix $L>0$ to be determined.  Then
\[\begin{split}
	1
		&= \int g(t,x) \dx
		\leq \int_{-Lt^{2/3}}^{Lt^{2/3}} g(t,x) \dx
			+ \int_{[-Lt^{2/3},Lt^{2/3}]^c} \frac{|x|^2}{L^2 t^{4/3}} g(t,x) \dx\\
		&\leq \sqrt{2L t^{2/3}} \|g(t)\|_2 + \frac{1}{L^2 t^{4/3}} \int |x|^2 g(t,x) \dx.
\end{split}\]
where, in the second inequality, we used H\"older's inequality.  
Choosing $L$ sufficiently large, the last term is smaller than $1/2$ by the moment bounds above.  Thus, $t^{1/3} \|g(t)\|_2$ is bounded below, as desired. This concludes the proof.
%
\end{proof}

\subsection{The decay of $M_g$: proof of \Cref{p.decay}}\label{s.E_g_M_g_decay}

Although \Cref{p.decay} and~\eqref{e.c2274} yields that $E_g(t) \lesssim t^{-2/3}$, we, in fact, must establish the decay of $E_g$ before proving \Cref{p.decay} as the decay of $E_g$ is a crucial element in its proof.  This is in contrast to the proof in~\cite{GuHenderson} for the case $R = \delta$, where we deduced the decay of $M_g$ directly and then used this to establish the decay of $E_g$.  We state this main ingredient in the following proposition.

\begin{proposition}\label{p.E_g_decay}
	When $d=1$ and $R$ is continuous and satisfies~\eqref{e.R}, we have, for all $t$,
	\[
		E_g(t)
			\lesssim \min\left\{\frac{1}{\sqrt t},\frac{1}{t^{2/3}}\right\}.
	\]
\end{proposition}

In addition, we require the following lemma, allowing us to show that, in some sense, $g$, $u$, and $w$ cannot be ``too different.''
\begin{lemma}\label{l.maxima_guw}
	When $d=1$ and $R$ is continuous and satisfies~\eqref{e.R}, we have, for all $t\geq 1$,
	\be\label{e.c351}
		M_w(t)
			\leq M_u(t)
			\leq M_g(t)
			\lesssim M_w(t).
	\ee
	Furthermore, for any $(t,x) \in [1,\infty)\times \R$,
	\be\label{e.c352}
		\frac{g(t,x)^2}{M_g(t)} \lesssim w(t,x), u(t,x).
	\ee
\end{lemma}

Finally, we also have a small time bound on $M_g$.
\begin{lemma}\label{l.small_time}
	When $d=1$ and $R$ is continuous and satisfies~\eqref{e.R}, we have, for all $t\in [0,1]$, 
	\[
		E_g(t)
			\leq M_g(t)
			\lesssim \frac{1}{\sqrt t}.
	\] 
	The implied constant in the above inequality does not depend on $\supp(g_0)$.
\end{lemma}

We establish \Cref{p.E_g_decay} in \Cref{s.E_g_decay} up to a technical lemma.  The technical lemma is proved in \Cref{s.integral_term}, and relies as well on \Cref{l.maxima_guw}.  The latter is proved in \Cref{s.maxima}.  Finally, \Cref{l.small_time} has an elementary proof, that relies on the identity \eqref{e.differential_inequality} established at the beginning of \Cref{s.E_g_decay}, which describes the evolution of  $E_g$. 

To reiterate the dependencies outlined above, we use all three results (\Cref{p.E_g_decay}, \Cref{l.maxima_guw}, and \Cref{l.small_time}) to prove \Cref{p.decay}.  The proofs of \Cref{l.maxima_guw} and \Cref{l.small_time} do {\em not} depend on \Cref{p.E_g_decay}.  The proof of \Cref{p.E_g_decay}, however, {\em does} depend on \Cref{l.small_time} .\

We now prove \Cref{p.decay} assuming \Cref{p.E_g_decay} and \Cref{l.maxima_guw,l.small_time}.

\begin{proof}[Proof of \Cref{p.decay}]
The bound for $t\leq 1$ follows directly from \Cref{l.small_time}.  
Hence, we need only address the case $t\geq 1$.  By evaluating~\eqref{e.main} at the location of a spatial maximum $(t,x_t)$, we find
	\be\label{e.c6301}
		\dot M_g \leq M_g (E_g - w(t,x_t)).
	\ee
	Using \Cref{p.E_g_decay} and \Cref{l.maxima_guw} in~\eqref{e.c6301} yields, for some $C>0$,
	\[
		\dot M_g \leq M_g \left( \frac{C}{t^{2/3}} - \frac{M_g}{C}\right).
	\]
	Let $\overline M(t) = A t^{-2/3}$ for $A$ to be determined.  By \Cref{l.small_time}, if $A$ is sufficiently large, $\overline M(1) \geq M_g(1)$.  In addition, we have
	\[
		\dot{\overline M}
			- \overline M \left( \frac{C}{t^{2/3}} - \frac{\overline M}{C}\right)
			= - \frac{2 A}{3 t^{5/3}}
				+ \frac{A^2}{t^{4/3}} \left(\frac{1}{C} - \frac{C}{A}\right)
			>0,
	\]
	where the last inequality holds as long as $A$ is sufficiently large.  The comparison principle implies that, for all $t\geq 1$,
	\[
		M_g(t) \leq \overline M(t).
	\]
	This concludes the proof.
\end{proof}

\subsection{A differential equation for $E_g$}\label{s.E_g_decay}

Before we embark on the proof, we note a useful identity for convolutions that is applied often in the sequel.  By the symmetry of $R$ (recall~\eqref{e.R}),
\be\label{e.convolution_identity}
	\int (R* f_1) f_2\dx = \int f_1 (R*f_2) \dx
		\qquad\text{ for any $f_1, f_2$.}
\ee

We begin our proof by deriving a differential equation for $E_g$.  Convolving~\eqref{e.main} with $R$, we find
\be
	\partial_t w
		- \frac{1}{2} \Delta w
		= w E_g - R*(gw).
\ee
Multiplying by $g$ and integrating yields
\[
	\int g \partial_t w \dx
		+ \frac{1}{2} D_g
		= E_g^2 - \int g R*(gw) \dx.
\]

Notice that
\[
	\dot E_g
		= \int (\partial_t g w + g \partial_t w) \dx
\]
and
\[
	\int g \partial_t w dx
		= \int g R* (\partial_t g) \dx
		= \int (R*g) \partial_t g \dx
		=\int w \partial_t g dx.
\]
Above we used~\eqref{e.convolution_identity}. Hence
\[
	\int g \partial_t w \dx
		= \frac{1}{2} \dot E_g.
\]
Thus, the above becomes
\be\label{e.c6302}
	\frac{1}{2} \dot E_g
		+ \frac{1}{2} D_g
		= E_g^2 - \int g R*(gw) \dx.
\ee

We now derive a simplified form for the right hand side of~\eqref{e.c6302}.  First, using~\eqref{e.convolution_identity} a second time, we find
\[
	\int g R*(gw) \dx
		= \int (R*g) g w \dx
		= \int g w^2 \dx.
\]
Recalling that $g$ is a probability measure and $E_g = \langle g, w\rangle$, we obtain
\[\begin{split}
	E_g^2 - \int g w^2 \dx
		&= - \int g(w^2 - E_g^2) \dx
		= - \int g( w^2 - 2 w E_g + E_g^2) \dx \\
		&= - \int g(w - E_g)^2 \dx.
\end{split}\]

Putting this together with~\eqref{e.c6302}, we find
\be\label{e.differential_inequality}
	\dot E_g + D_g
		= - 2 \int g(w-E_g)^2 \dx.
\ee
We note that~\eqref{e.differential_inequality} implies that $E_g$ is decreasing.  However, this is not sufficient for \Cref{p.E_g_decay}, and, instead, it is required to bound the integral term on the right hand side away from zero.  Before proceeding with this proof, we show how to conclude \Cref{l.small_time}.

We now establish the (much simpler) bound on $E_g$ and $M_g$ for small times.

\begin{proof}[Proof of \Cref{l.small_time}]
The bound on $E_g$ follows from a general argument that works in all dimensions that is based on a generalization of the Nash inequality.  As this decay is the focus of \Cref{t.hd_moments} and the one dimensional result is simply a side-effect of the analysis, we postpone it until \Cref{s.higher_dimensions}.  The bound is given in \Cref{p.hd_decay}.  We note that the bound in \Cref{p.hd_decay} is weaker than that of \Cref{p.E_g_decay} and its proof is independent of \Cref{p.E_g_decay}.  Hence, as we noted above, the proof of \Cref{l.small_time} is independent of \Cref{p.E_g_decay}, which is important because we use \Cref{l.small_time} to prove \Cref{p.E_g_decay}.  Hence, to finish the proof we need only derive the bound for $M_g$ from the bound for $E_g$.

Let $C$ be such that $E_g(t) \leq C t^{-1/2}$.  Define $h$ to be the solution of $\partial_t h = (1/2) \Delta h$ with $h(0,\cdot) = g_0$.  Let $G$ be the heat kernel; that is,
\[
	G(t,x)
		= \frac{1}{\sqrt{2\pi t}} e^{-\frac{x^2}{2t}}.
\]
Then $h = G(t) * g_0$.

We now define our barrier function
\[
	\overline g(t,x) = e^{\int_0^\tau E_g(s)\, ds} h(t,x).
\]
Notice that $\overline g$ solves
\be\label{e.c10141}
	\partial_\tau \overline g
		= \frac{1}{2} \overline g
			+ E_g \overline g.
\ee
By~\eqref{e.main} and the non-negativity of $g$, we see that $g$ is a subsolution of~\eqref{e.c10141}.  Hence, the comparison principle implies that $g\leq \overline g$.  Thus, for all $t\in [0,1]$,
\[\begin{split}
	M_g(t)
		&\leq M_{\overline g}(t)
		= e^{\int_0^t E_g(s)\, ds} M_h(t)
		\leq e^{2C t} M_h(t)\\
		&\lesssim \|G(t) * g_0\|_\infty
		\leq \|G(t)\|_\infty \|g_0\|_1
		\lesssim \frac{1}{\sqrt t},
\end{split}\]
which concludes the proof.
\end{proof}

Having finished the proof of the bounds for small time, we focus on the case $t\geq 1$ for the remainder of the section.  To that end, we seek to bound the integral term on the right hand side of~\eqref{e.differential_inequality} away from zero.  We establish this in the following lemma.
\begin{lemma}\label{l.integral_bound}
	When $d=1$ and $R$ is continuous and satisfies~\eqref{e.R}, we have, for every $t\geq 1$,
	\[
		\frac{E_g^5}{D_g}
			\lesssim \int g(w - E_g)^2 \dx.
	\]
\end{lemma}
Ideally, to prove \Cref{l.integral_bound}, we would apply similar methods as were used to establish the lower bound on $\int g(M-g)\dx$ in~\cite[Lemma 4.3]{GuHenderson} for the case $R=\delta$.  This is, in spirit, possible; however, it is complicated by the fact that the right hand side of~\eqref{e.differential_inequality} involves both $g$ and $w=R*g$, which take different values when $R\neq\delta$.  As a result, the proof of \Cref{l.integral_bound} is significantly more technical than the proof of its counterpart in \cite{GuHenderson}.  In fact, it is in the proof of this inequality that almost all of the difficulty lies.  We delay its proof until \Cref{s.integral_term}.  We now show how to conclude the upper bound on $E_g$ assuming \Cref{l.integral_bound}.

\begin{proof}[Proof of \Cref{p.E_g_decay}]
We first note that \Cref{l.small_time} yields a bound on $E_g(1)$.  This can be extended to a bound on $[1,2]$ since $E_g$ is decreasing (recall~\eqref{e.differential_inequality}).  Hence, we need only establish an upper bound on $[2,\infty)$.

We begin by applying~\eqref{e.differential_inequality} and \Cref{l.integral_bound} to find, for $t\geq 1$ and some $C>0$,
\[
	\dot E_g + D_g
		+ \frac{1}{C} \frac{E_g^5}{D_g}
			\leq 0.
\]
Applying Young's inequality, we have
\[
	\frac{2}{\sqrt{C}} E_g^{5/2}
		= 2\sqrt{D_g} \sqrt{\frac{E_g^5}{C D_g}}
		\leq D_g + \frac{1}{C} \frac{E_g^5}{D_g}.
\]
Combining the two above inequalities, we find
\[
	\dot E_g + \frac{2}{\sqrt{C}} E_g^{5/2}
		\leq 0.
\]
Solving this differential inequality yields, for all $t \geq 1$,
\[
	E_g(t)
		\lesssim \frac{1}{(t-1)^{2/3}}.
\]
This concludes the proof of the upper bound for all $t\geq 2$, which finishes the proof.
\end{proof}

\subsection{A lower bound on the integral term}\label{s.integral_term}

We begin by stating a key lemma related to the spatial regularity of $g$.  This is the ``local-in-time Harnack inequality.''  While this was introduced in~\cite{BHR_toads_delay}, we use the more precise statement of~\cite{BHR_nonlocal} as we require the flexibility in the parameter $t_0$ in the sequel.
\begin{proposition}\cite[Proposition 1.2]{BHR_nonlocal} \label{p.harnack}
	Fix any $t > t_0 > 0$.  Let $\eps \in (0,1)$.  There exists $C_\eps$, depending only on $\eps$, such that, for any $x,y\in\R$,
	\[
		g(t,x)
			\leq C_\eps \exp\Bigg\{C_\eps \|g\|_{L^\infty([t-t_0,t]\times \R)} t_0 + \frac{C_\eps |x-y|^2}{t_0}\Bigg\} g(t,y)^{1 - \eps} \|g\|_{L^\infty([t-t_0,t]\times \R)}^{\eps}.
	\]
\end{proposition}
The main difference between \Cref{p.harnack} and the standard parabolic Harnack inequality is the fact that we do not require a ``shift'' in time; that is, the $g$ terms on the right and left are both evaluated at the same $t$ (up to the $\|g\|_\infty^\eps$ error term).  We use \Cref{p.harnack} often in the sequel.  

%
%

\subsubsection{How the maxima evolve}\label{s.maxima}

We now establish a preliminary upper bound on $g$ and use it, along with the local-in-time Harnack inequality (\Cref{p.harnack}) to establish the comparison between $g$, $u$, and $w$ (\Cref{l.maxima_guw}).

We first show that the maxima of $g$, $u$, and $w$ are bounded independent of $t\geq 1$.  This allows us to obtain regularity of $g$ that is uniform in $t$ as $t\to\infty$.
\begin{lemma}\label{l.maxima_bounded}
	Suppose that $d=1$ and $R$ is continuous and satisfies~\eqref{e.R}.  For all $t \geq 1/4$,
	\[
		M_g(t) \lesssim 1.
	\]
\end{lemma}
\begin{proof}
Fix any $t_0 \geq 1/4$ we establish a bound of $M_g(t_0)$.  Let $\tilde g(t,x) = g(t + (t_0-1/4),x)$.  Since~\eqref{e.main} is an autonomous equation, $\tilde g$ satisfies~\eqref{e.main} with initial data $\tilde g(0,\cdot) = g(t_0-1/4,\cdot)$ and $\int \tilde{g}(0,\cdot) =1$.  Applying \Cref{l.small_time}, we find
\[
	M_g(t_0)
		= M_{\tilde g}(1/8)
		\lesssim \frac{1}{\sqrt{1/8}}
		\lesssim 1,
\]
which concludes the proof.
\end{proof}

We next require that, over finite time intervals, the $M_g$ does not change too much.  This allows us to apply \Cref{p.harnack} in such a way that we can replace the $\|g\|_{L^\infty([t-t_0,t]\times\R)}$ term by $M_g(t)$.
\begin{lemma}\label{l.maxima_change}
	For all $t_1 \geq 1$ and $t_0 \in [0,1/2]$, we have
	\[
			M_g(t_1)
				\lesssim M_g(t_1-t_0)
				\lesssim M_g(t_1).
	\]
\end{lemma}
\begin{proof}
	Fix $t_1\geq 1$ and $t_0 \in [0,1/2]$.  Let $h$ be the solution to
	\[
		\begin{cases}
			h_t = \frac{1}{2} \Delta h \qquad &\text{ in } (t_1-3/4,\infty)\times \R,\\
			h = g(t_1-3/4,\cdot) \qquad &\text{ on } \{t_1-3/4\} \times \R.
		\end{cases}
	\]
	Recall that $M_g(t)$ is bounded above by \Cref{l.maxima_bounded} for $t\geq 1/4$.  Let
	\be\label{e.c10142}
		A = \sup_{t\geq 1/4} M_g(t) < \infty
	\ee
	and define $\overline g(t,x) = e^{A(t-(t_1-3/4))} h(t,x)$ and $\underline g(t,x) = e^{-A(t-(t_1-3/4))} h(t,x)$.  We claim that
	\begin{equation}\label{e.c2271}
		\underline g(t,x) \leq g(t,x) \leq \overline g(t,x),
	\end{equation}
	for all $t\in [t_1 - 3/4,t_1]$, and that
	\begin{equation}\label{e.c2272}
		M_h(t_1)
			\leq M_h(t_1 - t_0)
			\lesssim M_h(t_1).
	\end{equation}
	We prove this in the sequel.  Let us momentarily assume both~\eqref{e.c2271} and~\eqref{e.c2272} are true and conclude the proof.
	
	Indeed, then we have that
	\[\begin{split}
		M_g(t_1)
			&\leq M_{\overline g}(t_1)
			= e^{3A/4} M_h(t_1)
			\leq e^{3A/4} M_h(t_1-t_0)\\
			&= e^{3A/4} \left(e^{A(3/4-t_0)} M_{\underline g}(t_1-t_0)\right)
			\leq e^{A(3/2-t_0)} M_g(t_1-t_0)
			\leq e^{3A/2} M_g(t_1-t_0),
	\end{split}\]
	and, similarly,
	\[\begin{split}
		M_g(t_1-t_0)
			&\leq M_{\overline g}(t_1-t_0)
			= e^{A(3/4 - t_0)} M_h(t_1-t_0)
			\lesssim e^{A(3/4 - t_0)} M_h(t_1)\\
			&= e^{A(3/4 - t_0)} \left(e^{3A/4}  M_{\underline g}(t_1)\right)
			\lesssim e^{3A/2} M_g(t_1).
	\end{split}\]
	This establishes the claim up to proving~\eqref{e.c2271} and~\eqref{e.c2272}.
	
	We first establish~\eqref{e.c2271}.  We show only the first inequality as the second is proved similarly.  Indeed, by the fact that $\partial_t \underline  g=\frac12\Delta \underline g-A \underline g$, we have
	\[\begin{split}
		\partial_t \underline g
			- \frac{1}{2} \Delta \underline g
			- \underline g \left(E_g - R*g\right)
			= - \underline g(A + E_g - R*g). 
	\end{split}\]
	Since $A + E_g\geq M_g$ by~\eqref{e.c10142} and $M_g \geq R*g$, $\underline g$ is a subsolution of~\eqref{e.main}.  Hence, by the comparison principle, $\underline g \leq g$ on $[t_1-3/4,\infty)\times \R$, as claimed.
	
	Now we establish~\eqref{e.c2272}.  The first inequality follows from the maximum principle; that is, the maximum of a solution to the heat equation is decreasing in time.  The second inequality follows indirectly from the parabolic Harnack inequality.  Indeed, let $x_m$ be the location of a maximum of $h(t_1-1/2,\cdot) = g(t_1-1/2,\cdot)$.  Then the parabolic Harnack inequality implies that
	\[
		M_h(t_1-1/2) = \sup_{x \in B_1(x_m)}h(t_1-1/2,x)
			\lesssim \inf_{x \in B_1(x_m)} h(t_1,x)
			\lesssim M_h(t_1).
	\]
	Using the maximum principle again, we find $M_h(t_1-t_0) \leq M_h(t_1-1/2)$ since $t_0 \in [0,1/2]$.  Combining this with the above inequality, establishes~\eqref{e.c2272}.  This concludes the proof.
\end{proof}

We are now able to utilize \Cref{p.harnack} to show that $g$, $u$, and $w$ are comparable; that is, we prove \Cref{l.maxima_guw}.

\begin{proof}[Proof of \Cref{l.maxima_guw}]
	The first and second inequalities in~\eqref{e.c351} are immediate by Young's inequality for convolutions: for all $(t,x)$,
	\[
		w(t,x)
			= \phi* u(t,x)
			\leq \|\phi\|_1 \|u(t,\cdot)\|_{\infty}
			= M_u(t).
	\]
	The proof of the second inequality is  the same and is thus omitted.
	
	Finally, we point out that the third inequality in~\eqref{e.c351} follows directly from~\eqref{e.c352} (when applied at $(t,x_t)$ that is the location of a spatial maximum of $g(t,\cdot)$).  In addition, the proof of~\eqref{e.c352} is the same for either right hand side ($w$ or $u$).  Thus, we show it only for  $u$.
	
	Fix $t \geq 1$ and any $x \in \R$.  Letting $I(r) = \int_{B_r} \phi$, we see that $I$ is continuous, $I(0) = 0$, and $I(\infty) = 1$. Hence, the intermediate value theorem implies the existence of
	$r_0>0$ such that
	\[
		\int_{B_{r_0}} \phi(y) dy = \frac{1}{2}.
	\]
	
	Fix any $y \in B_{r_0}$.  By \Cref{p.harnack} with $\eps = 1/2$ and $t_0 = 1/2$, we have
	\[
		g(t,x-y)^{1/2}
			\gtrsim \frac{g(t,x)}{\exp(C \sup_{s \in[t-1/2,t]} M_g(s) + \frac{C r_0^2}{1/2}) \sup_{s\in[t-1/2,t]} M_g(s)^{1/2}}.
	\]
	Using \Cref{l.maxima_bounded}, the exponential term in the denominator is bounded.  In addition, \Cref{l.maxima_change} implies that
	\[
		\sup_{s\in [t-1/2,t]} M_g(s)
			\lesssim M_g(t).
	\]
	Hence, we find, for all  $y\in B_{r_0}$,
	\be\label{e.c6306}
		g(t,x-y)
			\gtrsim \frac{g(t,x)^2}{M_g(t)}.
	\ee
	
	We now use this inequality to conclude.  Recalling the definition of $u$ and applying~\eqref{e.c6306} yields
	\[
		u(t,x)
			\geq \int_{B_{r_0}} \phi(y) g(t,x-y)dy
			\gtrsim \int_{B_{r_0}} \frac{g(t,x)^2}{M_g(t)} \phi(y) dy
			= \frac{g(t,x)^2}{2 M_g(t)},
	\]
	which concludes the proof.
\end{proof}


%
%

\subsubsection{The lower bound on the integral term: the proof of \Cref{l.integral_bound}}

Having established the relationship between $w$, $u$, and $g$, we are now in a position to prove the main technical lemma in the proof of \Cref{p.decay} apart from one final technical lemma.  This lemma shows that if $g(t,x_1)$ is sufficiently small compared to $g(t,x_0)$, then $w(t,x_1)$ and $u(t,x_1)$ will be also be small compared to $g(t,x_0)$.  We state this lemma now and prove it in the sequel.

\begin{lemma}\label{l.far_level_sets}
	Suppose that $d=1$ and $R$ is continuous and satisfies~\eqref{e.R}. For all $t\geq 1$, if $\frac{g(t,x_0)^2}{g(t,x_1) M_g(t)}$ is sufficiently large, depending only on $R$ and $M_g(t)/g(t,x_0)$, then
	\[
		w(t,x_1), u(t,x_1) \leq \frac{1}{2} g(t,x_0).
	\]
\end{lemma}

We now prove \Cref{l.integral_bound}.

\begin{proof}[Proof of \Cref{l.integral_bound}]
Since time plays no role here, except to allow us to apply the lemmas in \Cref{s.maxima}, we omit $t\geq 1$ notationally for the remainder of the proof.  We let $A> 1$ be a constant to be determined.  There are two cases to consider, and $A$ is determined in the second case.

{\bf Case one: $M_g \leq A E_g$.}  
The constant $A$ does not play a role in this case, and, as such, we absorb it into the constants in the $\lesssim$ notation.

Fix $1 = B_1 <B_2 < B_3<B_4$ to be determined.  We may find $x_1 < x_2 < x_3< x_4$ such that
\be\label{e.B_i_choices1}
\begin{aligned}
	&g(x_i) = \frac{E_g}{B_i}
		&\quad\text{ for each } i \in \{1,2,3,4\},\\
	&g(x) \leq \frac{E_g}{B_i}
		&\quad \text{ for all } x \in[x_i, x_4]
			\text{ and each } i \in \{1,2\}\\
	&g(x) \geq \frac{E_g}{B_i}
		&\quad \text{ for all } x \in[x_1, x_i]
			\text{ and each } i \in \{3,4\}\\
\end{aligned}
\ee
This can be achieved by making the choices
\begin{align*}
	&x_1 = \sup\{x \in \R : g(x) = E_g\}
		&\quad
	&x_4 = \inf\{x \geq x_1 : g(x) = E_g/B_4\}\\
	&x_2 = \sup\{x \leq x_4: g(x) = E_g/B_2\}
		&
	&x_3 = \inf\{x \geq x_1: g(x) = E_g/ B_3\}.
\end{align*}
Roughly, $x_1$ and $x_2$ are the ``last times'' $g$ takes the values $E_g/B_1$ and $E_g/B_2$, respectively, while $x_3$ and $x_4$ are the ``first times'' after $x_2$ that $g$ takes the values $E_g/B_3$ and $E_g/B_4$, respectively.  See \Cref{f.level_sets}.

\begin{figure}
	
	\begin{overpic}[scale=.325,angle=90]
		{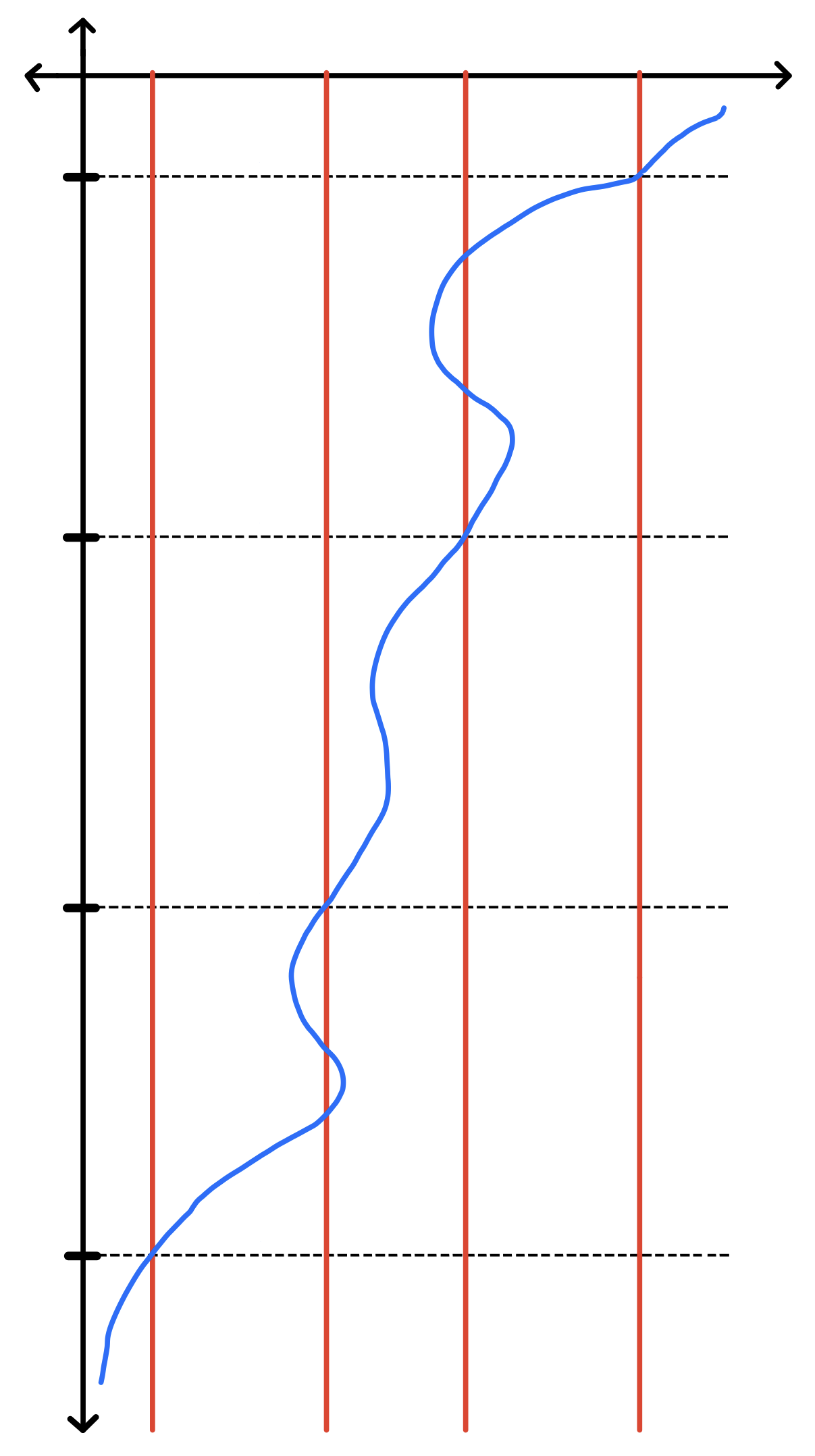}
		\put(0,9.5){\color{red} $\frac{E_g}{B_4}$}
		\put(0,22){\color{red} $\frac{E_g}{B_3}$}
		\put(0,32){\color{red} $\frac{E_g}{B_2}$}
		\put(0,45){\color{red} $\frac{E_g}{B_1}$}
		\put(9,50){\color{blue} $g$}
		\put(11.5,2){$x_1$}
		\put(35.5,2){$x_2$}
		\put(61,2){$x_3$}
		\put(85,2){$x_4$}
	\end{overpic}
	\caption{A cartoon depicting the relationship of $x_i$, $B_i$, and $g$.}\label{f.level_sets}
\end{figure}
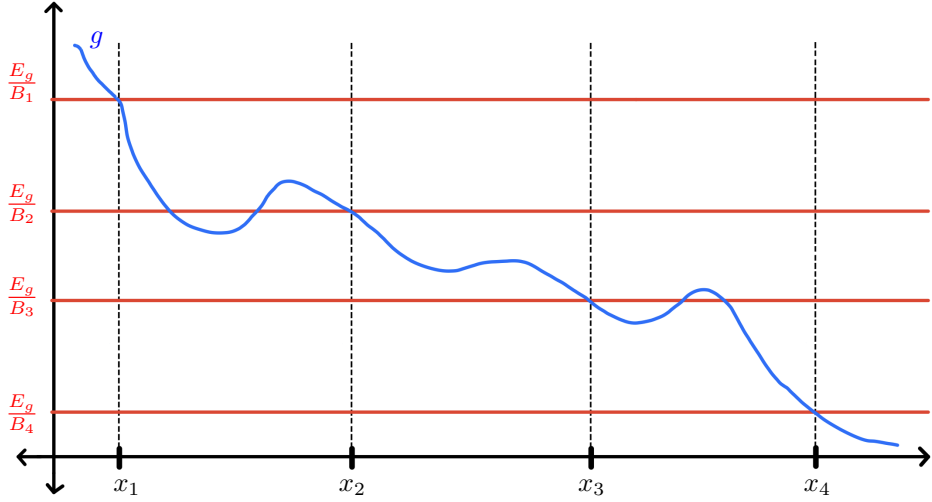


Fix any $x \in [x_2,x_4]$.  Applying \Cref{l.far_level_sets}, we have that, choosing $B_2$ so that
\[
	\frac{g(x_1)^2}{g(x) M_g}
		= \frac{E_g^2}{g(x) M_g}
		\geq \frac{B_2}{A}
\]
is sufficiently large, depending only on $M_g/g(x_0) = M_g/E_g \leq A$, then
\[
	w(x) \leq \frac{g(x_1)}{2}
		= \frac{E_g}{2}.
\]
We conclude that $E_g - w \geq E_g/2$.  Hence, 
%
%
%
\begin{equation}\label{e.c2282}
	\int_{x_2}^{x_4} g (E_g - w)^2 dx
		\geq \int_{x_2}^{x_4} \frac{E_g}{B_4} \frac{E_g^2}{4} dx
		= \frac{E_g^3}{4B_4} |x_4 - x_2|
		\gtrsim E_g^3 |x_4 - x_2|.
\ee
Here we used that $g \geq E_g/B_4$ on $[x_2,x_4]$.  On the other hand, we have
\be\label{e.c361}
	\begin{split}
		|u(x_2) - u(x_4)|^2
			&= \left(\int_{x_2}^{x_4} \partial_x u dx\right)^2
			\leq |x_4 - x_2| \int_{x_2}^{x_4} |\partial_x u|^2 dx\\
			&\leq |x_4 - x_2| \int |\partial_x u|^2 dx
			= |x_4 - x_2| D_g.
	\end{split}
\ee

We now seek a lower bound on $|u(x_2) - u(x_3)|$.  Using \Cref{l.maxima_guw} to bound $u(x_2)$ from below, we find
\be\label{e.c884}
	u(x_2)
		\gtrsim \frac{g(x_2)^2}{M_g}
		= \frac{E_g^2}{M_g B_2^2}
		\gtrsim \frac{E_g}{AB_2^2}.
\ee
Choosing now $B_3$ to be $A B_2^2$ multiplied by a universal constant, we have
\be\label{e.c884}
	u(x_2)
		\geq \frac{E_g}{B_3}.
\ee
Next, we again apply \Cref{l.far_level_sets} to conclude that, choosing $B_4$ so that
\[
	\frac{g(x_3)^2}{g(x_4) M_g}
		= \frac{E_g}{M_g} \frac{B_4}{B_3^2}
		\geq \frac{1}{A} \frac{B_4}{B_3^2}
\]
is sufficiently large, depending only on $M_g/g(x_3) \leq A B_2$, then
\be\label{e.c885}
	u(x_4)
		\leq \frac{g(x_3)}{2}
		= \frac{E_g}{2 B_3}.
\ee
Putting the bounds~\eqref{e.c884} and~\eqref{e.c885} together, we find
\[
	|u(x_2) - u(x_4)|
		\geq \frac{E_g}{2B_3}.
\]

Including the above inequality in~\eqref{e.c361}, we find
\be\label{e.c2283}
	E_g^2
		\lesssim |x_4 - x_2| D_g,
\ee
where we have absorbed the dependence of the $B_i$ into the $\lesssim$ notation.  Combining~\eqref{e.c2282} and~\eqref{e.c2283} finishes the proof in case one.  Note that, while the estimate above depends on $A$, in the second case, we choose $A$ to be a fixed large number depending only on $R$ and independent of $E_g$.

{\bf Case two: $M_g \geq A E_g$.}  
The argument in this case is similar; however, instead of bounding $w$ {\em above} by $E_g/2$, we bound it {\em below} by $3E_g/2$.  Indeed, let $1 = B_1< B_2 < B_3$ and find $x_1 < x_2 < x_3$ such that
\[
	\begin{aligned}
	&g(x_i) = \frac{M_g}{B_i}
		\quad &\text{for each } i \in \{1,2,3\},\\
	&g(x) \leq \frac{M_g}{ B_2}
		\quad &\text{ for all } x \in[x_2,x_3], &\quad\text{ and}\\
	&g(x) \geq \frac{M_g}{B_3}
		\quad &\text{ for all } x \in [x_1,x_3].
\end{aligned}
\]
Note that $B_1 < B_2 < B_3$.

Arguing exactly as in~\eqref{e.c361} in the previous case, we find
\[
	|u(x_1) - u(x_3)|^2
		\leq |x_3 - x_1| D_g.
\]
From \Cref{l.maxima_guw}, we have
\be\label{e.c893}
	u(x_1)
		\gtrsim \frac{g(x_1)^2}{M_g}
		= M_g.
\ee
The constant above does not depend on $A$ or on any of the $B_i$'s.  Hence, we may select $B_2$ such that $g(x_2) \leq u(x_1)$.  

Then, after increasing $B_3$ such that
\[
	\frac{g(x_2)^2}{g(x_3) M_g}
		= \frac{B_3}{B_2^2}
\]
is sufficiently large, depending only on $M_g/g(x_2) = B_2$, \Cref{l.far_level_sets} implies that
\[
	u(x_3)
		\leq \frac{g(x_2)}{2}
		\leq \frac{u(x_1)}{2}.
\]
Using this and~\eqref{e.c893}, we conclude that
\be\label{e.c892}
	M_g^2
		\lesssim \left|\frac{u(x_1)}{2}\right|^2
		\leq |u(x_3)-u(x_1)|^2
		\leq |x_3 - x_1| D_g.
\ee

Applying \Cref{l.maxima_guw}, we have that, for $x\in [x_1,x_3]$,
\[
	w(x)
		\gtrsim \frac{g(x)^2}{M_g}
		\geq \frac{M_g}{B_3^2},
\]
In the second inequality we used that $g\geq M_g/B_3$ on $[x_1,x_3]$.  Since $E_g \leq M_g/A$, then, choosing $A$ such that $A/ B_3^2$ is sufficiently large, we have $w(x) - E_g \gtrsim M_g$ for all $x \in [x_1, x_3]$.  Thus, we find
\be\label{e.c891}
	\int_{x_1}^{x_3} g (E_g - w)^2 dx
		\gtrsim M_g^3 |x_1 - x_3|.
\ee

Combining~\eqref{e.c891} and~\eqref{e.c892} yields
\[
	\int g (E_g - h)^2 dx
		\gtrsim \frac{M_g^5}{D_g}.
\]
The proof is finished in this case by using the fact that $M_g \geq A E_g$.  Thus, we have concluded the proof in all cases.
\end{proof}

We now finish this section by proving the final lemma, \Cref{l.far_level_sets}.  The idea behind this proof is that, by \Cref{p.harnack}, if the ratio $g(x_0)^2/g(x_1) M_g$ is large enough, then $g$ is closer to $g(x_1)$ than $g(x_0)$ on a large set.  Plugging this into the convolutions defining $u$ and $w$ yields the result.

\begin{proof}[Proof of \Cref{l.far_level_sets}]
	We show the result for $w$.  The proof is exactly the same for $u$ and, hence, we omit it.  We assume without loss of generality that $x_0 < x_1 = 0$ since the equation is invariant by reflection and translation.  We suppress the time dependence in the proof as $t$ plays no role.
	
	Let
	\be\label{e.c881}
		\bar\varepsilon = \frac{g(0) M_g}{g(x_0)^2}.
	\ee
	We may assume that $\bar\varepsilon < 1/4$.  Define
	\[
		r = \sup \{ s > 0 : g(y) \leq \frac{g(x_0)}{4} \text{ for all } y \in [-s,s]\}.
	\]
	From~\eqref{e.c881} and the fact that $\bar\varepsilon <1/4$, $g(0) < g(x_0)/4$, we conclude that $r>0$ and  $g \leq g(x_0)/4$ on $[-r,r]$.  In addition, by continuity, either $g(r) = g(x_0)/4$ or $g(-r) = g(x_0)/4$.  We assume the former, although the proof is similar in the latter case.
	
	Using this, we estimate the below directly:
	\be\label{e.c882}
	\begin{split}
		w(0)
			&= \int R(-y) g(y) dy
			\leq \int_{[-r,r]} R(-y) \frac{g(x_0)}{4} dy
				+ \int_{[-r,r]^c} R(-y) M_g dy\\
			&\leq \frac{g(x_0)}{4}
				+ 2M_g \int_r^\infty R(y) dy.
	\end{split}
	\ee
	The final inequality follows from the fact that $R$ is symmetric\footnote{Symmetry is not necessary for this result, but it simplifies the notation in the proof.}.
	
	We seek an estimate on the second term on the last line of~\eqref{e.c882}.  Indeed, we wish to show that the integral term is small, which corresponds to $r$ being large.  In order to conclude this, we appeal to \Cref{p.harnack} with $\eps = 1/2$ and $t_0 = 1/2$, along with \Cref{l.maxima_change}, to find
	\[
		\frac{1}{4\sqrt \eps}
			= \frac{1}{4} \frac{g(x_0)}{\sqrt{g(0) M_g}}
			= \frac{g(r)}{\sqrt{g(0) M_g}}
			\lesssim \exp\left\{ C(M_g + r^2)\right\}
			\lesssim \exp\left\{C r^2\right\}.
	\]
	The last line follows from \Cref{l.maxima_bounded}.  Choosing $\bar\varepsilon$ sufficiently small, makes $r$ sufficiently large that
	\[
		\int_r^\infty R(y) dy
			\leq \frac{g(x_0)}{8 M_g}.
	\]
	Plugging this into~\eqref{e.c882}, we conclude that $w(0) \leq \frac{g(x_0)}{2}$, which finishes the proof.
\end{proof}

\section{Long time dynamics in one dimension}\label{s.weak}


For technical reasons that become clear in the sequel, we notice that, since~\eqref{e.main} is autonomous, we may shift the initial data in time and assume that $g(1,\cdot) = g_0 \in C_c(\R)$ without losing generality.  In addition, recall that here we make the choice $R = \delta$ whence~\eqref{e.main} becomes
\be\label{e.main2}
	\begin{cases}
		g_t = \frac{1}{2} \Delta g + g \left( E_g - g\right)
			\qquad&\text{in } (1,\infty)\times \R,\\
		g = g_0
			\qquad&\text{on } \{1\}\times \R.
	\end{cases}
\ee

We begin with a heuristic argument that motivates the constants $\thc$ and $\cc$ defined in~\eqref{e.cc_thc}.  This argument also yields the main rescalings and objects of study for us and allows to outline our strategy and the main difficulties encountered in the proof.

\subsection{The rescalings and the heuristic argument}\label{s.heuristic}

We appeal to two different scalings revealing different features of the dynamics.   One suggests that $g$, properly rescaled, converges to a constant multiple of an indicator function while the second unveils the exact constant involved.  We call the rescaled functions, respectively, $h$ and $u$.  Importantly, and somewhat surprisingly, $u$ satisfies an equation reminiscent of the Fisher-KPP equation.

We begin with the new variables $\tau \sim t^{1/3}$ and $y \sim x/t^{2/3}$; that is, define
\be\label{e.h}
	h(\tau, y)
		= (\tau+1)^2 g((\tau+1)^3, y (\tau+1)^2).
\ee
Notice that this respects the scaling in $x$ and the decay of $g$ shown in \Cref{t.decay} and \Cref{p.decay}.  Thus, we have that $C^{-1} < M_h, E_h < C$.  We see that
\be\label{e.heq}
	\partial_\tau h
		= \frac{3}{2 (\tau+1)^2} \Delta h + 3 h(E_h - h) + \frac{2}{\tau+1} \left(h + y \partial_y h\right).
\ee
The meaning of the time shift of the initial conditions to be defined at $t=1$, above, is now clear:~\eqref{e.heq} does not degenerate as $\tau \searrow 0$ and the initial data of $h$ is defined at $\tau = 0$.

Notice that, heuristically, this indicates that if $h(\tau,y)\to H(y)$ as $\tau\to\infty$, then $H (E_{H} - H) \equiv 0$, implying that $H$ has to be of the form $|A|^{-1} \1_A$ for some set $A$ (recall that $g$ is a probability measure and, hence, $h$ is one too).  Of course, we expect $A = [-c,c]$ for some $c>0$.

Unfortunately, as the Laplacian term in~\eqref{e.heq} degenerates as $\tau\to\infty$, we cannot lean on compactness in order to obtain convergence of $h$.  Instead we need to obtain sharp estimates on $h$ away from the boundary points $y = \pm c$, which leads us to the next scaling.

We pull apart the dynamics of $g$ near $c t^{2/3}$ (i.e., $h$ near $c$) to see the transition between $(2c)^{-1} t^{-2/3}$ and $0$.  Since the coefficient in front of the Laplacian is $O(\tau^{-2})$, we expect that this transition layer has width $O(\tau^{-1})$.  Hence, for any $c\in \R$, we let
\be\label{e.u}
	u_c(\tau, z)
		= h(\tau, c + z (\tau+1)^{-1}).
\ee
Using~\eqref{e.heq}, we find
\be\label{e.ueq}
	\partial_\tau u_c
		= \frac{3}{2} \Delta u_c + 3 u_c( E_h - u_c) + 2c \partial_z u_c + \frac{2}{\tau+1} u_c  + \frac{z}{\tau+1}   \partial_z u_c.
\ee

We expect that, as $\tau\to\infty$, $E_h(\tau) \to \theta$ for some $\theta \geq 0$ and $h(\tau,y) \to H(y) = (2c)^{-1} \1_{[-c,c]}$.  For these to be consistent, we have $\theta = (2c)^{-1}$.  In addition, looking at~\eqref{e.ueq} with $\tau = \infty$, we formally find
\[
	- c \partial_z u_c = \frac{3}{4} \Delta u_c + \frac{3\theta}{2} u_c\left( 1 - \frac{u_c}{\theta}\right).
\]
This is the equation for a traveling wave solution (of speed $c$) of the Fisher-KPP equation.  Although there is a family of traveling wave solutions, whose speeds make up an infinite half-line, we expect the correct long-time dynamics to correspond to the minimal speed wave because our initial data is compactly supported (see, e.g.~\cite{Fisher,KPP}).  The minimal speed is given by a formula in terms of the coefficients (see, e.g.,~\cite[(1.25)]{Xin_book}), which, in our setting, yields
\[
	c = 2 \sqrt{\frac{3}{4} \cdot \frac{3\theta }{2}}= \frac{3}{\sqrt{2}} \sqrt{\theta}.
\] 
Recalling that $\theta = (2c)^{-1}$, the unique solution to the two equations is $(c,\theta) = (\cc,\thc)$, where $\cc$ and $\thc$ are defined in~\eqref{e.cc_thc}.

We now discuss the difficulties in establishing the above heuristics for $u_c$.  The first issue is a subtle one.  While the last two terms~\eqref{e.ueq} appear to be error terms, this is only true for the last term with the correct choice of $c$, that is, {\em only in the correct moving frame}.  Indeed, the coefficient of last term is approximately $z/\tau$.  If the transition from $\theta$ to $0$ occurs at $\tilde c \tau$ with $\tilde c \neq c$, then non-trivial behavior for $u_c$ occurs at $z\approx (\tilde c-c) \tau$.  In this case, both $\partial_z u_c$ and $z/(\tau+1)$ are non-trivial at $(\tilde c - c) \tau$.  Thus, if we have changed to the ``incorrect'' moving frame,
\[
	\liminf_{\tau\to\infty} \left\| \frac{z}{\tau+1} \partial_z u_c\right\|_\infty >0,
\]
and, hence, the heuristics above are no longer useful.  It is, thus, crucial to work in the frame with $c=\cc$.

Although the equation~\eqref{e.ueq} is a non-local Fisher-KPP type equation similar to that considered in~\cite{ABVV, BerestyckiNadinPerthameRyzhik, Britton, HamelRyzhik}, which takes the form
\be\label{e.nlfkpp1}
	\partial_\tau u
		= D \Delta u + u(r - \phi * u),
\ee
for some non-negative function $\phi$ and some constants $D, r>0$, there is an important difference that prevents us from using techniques developed for~\eqref{e.nlfkpp1} to study~\eqref{e.ueq}: while $\phi*u(x) \sim u(x)$ for each $x$ (when appropriately interpreted), $E_h \not\sim u_c(x)$ (recall that $E_h$ is independent of $x$ and is expected to converge to a constant).  As such, the role of the nonlocality in the two equations is quite different.  On a more technical level, this difference is manifested in the following way.  The spreading speed in~\eqref{e.nlfkpp1} is identified using the linearization of the equation around zero $\partial_\tau u = D \Delta u + ru$, in which the nonlocal term does not appear.  
On the other hand, the spreading speed of $u_c$ will be almost entirely determined by the long time behavior of the nonlocal term $E_h$.  
As a result, we are not able to draw inspiration from the techniques introduced  in the previous work  to deal with the nonlocal term of~\eqref{e.nlfkpp1}.

The next complication is due to the coupling between $E_h$ and $u_c$.  Their codependence makes it impossible to first show convergence of $E_h$ and then analyze $u_c$ when $\tau\gg 1$ and $E_h$ is almost constant.  On the other hand, one might be tempted to consider $E_h$ to be a time dependent prescribed coefficient; however, there is little robust theory on how the front of a reaction diffusion equation depends on the coefficients when their oscillations have no specific structure such as periodicity.  This is in part because one can construct coefficients to have a diverse range of fronts if the coefficients oscillate in a complicated manner~\cite{BerestyckiHamelNadin}.  Indeed, there are even quite simple settings where there is no defined spreading speed due to oscillations of the coefficients~\cite{GarnierGilettiNadin}.


To overcome this, we derive a differential inequality that shows that $E_h$ can only increase slowly (although we do not rule out it decreasing arbitrarily quickly).  Focusing our analysis on the resulting long upslopes, we can work with an almost constant $E_h$ term.  In this case, if $E_h$ is too large, it will have been too large for a large time interval beforehand.  Then the front, on that time interval, will move too quickly due to the large $E_h$ term in~\eqref{e.ueq}.  This corresponds to $h$ having wide support and contradicts the fact that the integral of $h$ is one.  Similarly, if $E_h$ is too small, it will be small for a long time afterwards.  Then the front, on that time interval, will move too slowly, corresponding to $h$ having too narrow of support and contradicting the fact that the integral of $h$ is one.

Before proceeding with the proof, we note that the rescalings above, combined with \Cref{p.decay} and the fact that $\int g \dx = \int h \dy = 1$, yield the following bounds, used often below.
\begin{lemma}\label{l.rescaled_bounds}
	For any $c\in \R$, We have, for all $\tau \geq 1$,
	\[
		1
			\lesssim E_h(\tau)
			\leq M_h(\tau) = M_{u_c}(\tau)
			\lesssim 1.
	\]
\end{lemma}

\subsection{The weak bounds: proof of \Cref{p.weak}}

A key step in establishing the strong bounds in \Cref{t.strong} is first establishing the weaker bounds in \Cref{p.weak}.  We show this argument here.

\subsubsection{The lower bound on the limsup}
We begin with the lower bound on the lim sup.  We require a simple relationship between $M$ and $E$ for large times, proved at the end of this section.

\begin{lemma}\label{l.limsup}
	We have $\displaystyle
		\limsup_{\tau\to\infty} M_h(\tau) = \limsup_{\tau\to\infty} E_h(\tau).$
\end{lemma}

The need for this lemma is the following.  We argue by contradiction, assuming that $\limsup E_h$ is small.  \Cref{l.limsup} implies that $M_h$ is eventually small as well.  However, to be consistent with the requirement that $\int h \dy = 1$, this forces $h$ to be nontrivial near some $y_0 > \cc$.  Using the connection to~\eqref{e.ueq}, this corresponds to front propagation of speed $y_0$ starting from compactly supported initial data; however, $y_0$ is greater than the speed $2 \sqrt{(3/4)\cdot(3E_h/2)} = 2 \sqrt{9 E_h/8}$ of the minimal speed traveling wave, which is known to be impossible.  The last step is achieved by constructing a supersolution that is typical in studying Fisher-KPP.

We now show how to conclude \Cref{p.weak}.(ii) using \Cref{l.limsup}.

\begin{proof}[Proof of \Cref{p.weak}.(ii)]
We prove this by contradiction, assuming that, for some $\eps>0$,
\[
	\limsup_{\tau\to\infty} E_h(\tau)
		\leq \theta_{\rm crit} - \eps.
\]
Then, by \Cref{l.limsup}, $\limsup M_h \leq \theta_{\rm crit} - \eps$, and, hence, there exists $\tau_0$ such that $E_h(\tau), M_h(\tau) \leq \theta_{\rm crit} - \eps/2$ for all $\tau \geq \tau_0$.

First we show that, for any $c$ and $\tau_0$, there exists $D_c>0$, depending on $\tau_0$ and $c$, such that
\be\label{e.c1224}
	u_c(\tau_0,z) \leq D_c \exp\left\{ - \frac{z^2}{D_c}\right\}.
\ee
Notice that, by inverting the change of variables relating $u_c$ and $h$, it is enough to show this for $c=0$.   The bound of $u_{c=0}$ of the form in~\eqref{e.c1224} follows from changing variables from $g$ to $u_{c=0}$ and using~\eqref{e.c8101}.

Fix $A>0$ to be chosen.  Let $c = \cc$, $\lambda = 2 c/3$, and let
\[
	\overline u_c(\tau, z) = A e^{-\lambda z}.
\]
We aim to show that $u_c \leq \overline u_c$ on $[\tau_0,\infty)\times \R$.  

First, we consider the domain $[\tau_0,\infty)\times (-\infty,0]$.  By \Cref{l.rescaled_bounds}, $u_c$ is uniformly bounded above, and $\overline u_c \geq A$ for any $z \leq 0$.  Hence, if $A$ is sufficiently large, then we have that $\overline u_c \geq u_c$ on $[\tau_0,\infty)\times(-\infty,0]$.

Next we show that $\overline u_c \geq u_c$ on $[\tau_0,\infty)\times (0,\infty)$.  We do so via the comparison principle.  The first step is to show that $\overline u_c$ is a super-solution of~\eqref{e.ueq} on $(\tau_0,\infty)\times (0,\infty)$.  Indeed,
\[\begin{split}
	&\partial_\tau\overline u_c
		- \frac{3}{2} \Delta \overline u_c
		- 3 \overline u_c(E_h - \overline u_c)
		- 2 c \partial_z \overline u_c
		- \frac{2}{\tau+1} \overline u_c
		- \frac{z}{\tau+1} \partial_z \overline u_c\\
	&\qquad
		= \overline u_c\left(
			- \frac{3}{2} \lambda^2
			- 3E_h
			+ 3\overline u_c
			+ 2 c \lambda
			- \frac{2}{\tau+1}
			+ \frac{\lambda}{\tau+1}z		
		 \right)\\
	&\qquad
		\geq \overline u_c\left(
			- \frac{3}{2} \lambda^2
			- 3(\theta_{\rm crit} - \eps/2)
			+ 2 c \lambda
			- \frac{2}{\tau+1}
		 \right)
		= \overline u_c \left(
			\frac{2c^2}{3}
			- 3 (\theta_{\rm crit} - \eps/2)
			- \frac{2}{\tau+1}
		\right).
\end{split}\]
By our choice of $c$, it is clear that, up to increasing $\tau_0$ if necessary, the last line is non-negative and, thus, $\overline u_c$ is a supersolution of~\eqref{e.ueq}.

In order to apply the comparison principle, we address the parabolic boundary.  First, we have $u_c \leq \overline u_c$ on $[\tau_0,\infty)\times \{0\}$, as established above.  Second, up to increasing $A$, we have, via~\eqref{e.c1224},
\[
	u_c(\tau_0,z)
		\leq D_c \exp \left\{ - \frac{z^2}{D_c}\right\}
		\leq D_c \exp\left\{ D_c \lambda^2 - \lambda z\right\}
		\leq A e^{-\lambda z}
		\qquad\text{ for all } z\geq 0.
\]
Hence, $u_c \leq \overline u_c$ on $\{\tau_0\} \times (0,\infty)$.

We conclude that $u_c \leq \overline u_c$ on the parabolic boundary of $(\tau_0,\infty)\times (0,\infty)$ and that $\overline u_c$ is a supersolution of~\eqref{e.ueq}.  We can, thus, apply the comparison principle, which implies that $\overline u_c \geq u_c$ on $(\tau_0,\infty)\times \R$.  This concludes the proof that $u_c \leq \overline u_c$ on $[\tau_0,\infty)\times\R$.

From the definition of $\overline u_c$, it is clear that, for any $c' > c$,
\[\begin{split}
	\limsup_{\tau\to\infty} \int_{y > c'} h(\tau, y) \dy
		&= \limsup_{\tau \to \infty} \frac{1}{\tau+1}\int_{z > (c'-c)(\tau+1)} u_c(\tau,z) \dz\\
		&\leq \limsup_{\tau \to \infty} \frac{1}{\tau+1}\int_{z > (c'-c)(\tau+1)} \overline u_c(\tau,z) \dz
		= 0.
\end{split}
\]
A similar argument as above can be used to obtain a supersolution of $u_{-c}$ and bound the integral for $y< -c'$.  Thus, we find
\[
	\limsup_{\tau\to\infty} \int_{|y| > c'} h(\tau,y) dy = 0.
\]
In addition, we notice that
\[
	\limsup_{\tau \to \infty}
		\int_{|y|\leq c'}  h(\tau,y) \dy
		\leq \limsup_{\tau\to\infty} 2c' M_h(\tau)
		= 2c' \limsup_{\tau\to\infty} E_h(\tau)
		\leq 2c' (\theta_{\rm crit} - \eps).
\]
Here we used \Cref{l.limsup}.

On the other hand, $\int h(\tau,y) \dy = 1$.  Combining all above estimates, we have
\[
	1
		\leq \limsup_{\tau\to\infty} \int_{|y| > c'} h(\tau,y) \dy + \limsup_{\tau\to\infty} \int_{|y|\leq c'} h(\tau,y) \dy
		\leq 2c' (\theta_{\rm crit} - \eps).
\]
Since this is true of all $c'>c=\cc$, it follows for $c' = \cc$ in the limit.  Hence,
\[
	1
		\leq 2 \cc (\theta_{\rm crit} - \eps).
\]
Recall that $2 \cc \thc = 1$.  It follows that the right hand side is strictly less than one.  This is a contradiction, which concludes the proof.
\end{proof}

We now prove \Cref{l.limsup}.

\begin{proof}[Proof of \Cref{l.limsup}]
	Since $M_h(\tau) \geq E_h(\tau)$ for all $\tau$, it is enough to establish only the ``$\leq$'' in the claim above.   For notational reasons, let $\theta_M = \limsup M_h$ and $\theta_E= \limsup E_h$. 
	
	There are two cases.  Either $M_h$ is eventually monotonic in $\tau$ or not.  Consider first the latter case.  
	
	If $M_h$ is not eventually monotonic, then we may select $\tau_n$ that are local maxima of $M_h$ such that $M_h(\tau_n) \to \theta_M$.  It follows that $u_{c=0}$ has local maxima at $(\tau_n, z_n)$ for some $z_n$.  Using the equation~\eqref{e.ueq}, we find, at $(\tau_n, z_n)$,
	\[
		0 \leq \partial_\tau u_{c=0} - \frac{3}{2} \Delta u_{c=0}
			= 3 u_{c=0} \left( E_h(\tau_n) - u_{c=0} + \frac{2}{3(\tau_n+1)} u_{c=0}\right).
	\]
	Taking a limit as $n\to\infty$ and using the fact that $u_{c=0}(\tau_n,z_n) = M(\tau_n) \to \theta_M$ and $\limsup E_h(\tau_n) \leq \theta_E$, we find
	\be\label{e.c10151}
		0 \leq 3 \theta_M \left( \theta_E - \theta_M\right).
	\ee
	Since $\theta_M >0$ by \Cref{t.decay}, it follows that $\theta_M \leq \theta_E$, which concludes the proof in this case.

	We now consider the case where $M_h$ is eventually monotonic.  The case where $M_h$ is eventually nonincreasing and the case where it is eventually nondecreasing are handled similarly, so we show only the argument for when it is eventually nonincreasing.  In this case, we can select $\tau_n$ tending to infinity such that $\partial_\tau M_h(\tau_n) \to 0$ as $n\to\infty$, since, otherwise, $\partial_\tau M_h$ is uniformly negative, which implies that $M_h(\tau)\to-\infty$ as $\tau \to \infty$, a contradiction.  Since $M_h$ is eventually monotonic, $M_h(\tau_n) \to \theta_M$ as $n\to\infty$.  Note that $\Delta u_{c=0}(\tau_n,z_n) \leq 0$ when $z_n$ is the location of the spatial maximum of $u_{c=0}(\tau_n, \cdot)$. 
Using the equation~\eqref{e.ueq}, we find, at $(\tau_n, z_n)$,
	\[
		0\leq - \frac{3}{2} \Delta u_{c=0}
			= 3 u_{c=0} \left( E_h(\tau_n) - u_{c=0} + \frac{2}{3(\tau_n+1)} u_{c=0}\right) - \partial_\tau u_{c=0}.
	\]
	Taking a limit as $n\to\infty$ and using the fact that $u_{c=0}(\tau_n,z_n) = M(\tau_n) \to \theta_M$, $\limsup E_h(\tau_n) \leq \theta_E$, and $\lim \partial_\tau u_{c=0} = 0$, we find, again,~\eqref{e.c10151} in this case.  The proof is then concluded in the same way as in the previous case.
\end{proof}

\subsubsection{The upper bound on the liminf}

We begin by stating the proposition that is the crucial step.  It roughly states that if $E_h$ remains large enough (depending on $c$) for long enough, then $u_c$ grows up to $E_h$.  This is shown by approximating from below by solutions of the Fisher-KPP equation.

\begin{proposition}\label{p.growing_u_c}
	Fix any $\tau_0$, $\tau_1$, $R$, $\mu$, $\eps$, $\delta$, and $c$.  Suppose that
	\[
		u_c(\tau_0, \cdot) \geq \mu \1_{B_R }
			\quad\text{ and }\quad
		18 E_h(\tau) - 4 c^2 > \delta \ \ \text{ for all } \tau\in[\tau_0,\tau_1].
	\]
	If $\tau_0$ is sufficiently large, depending only on $\delta$ and $\eps$, then there exists $\underline \tau$, depending only on $\mu$, $R$, $\eps$, and $\delta$, such that, if $\tau_1 - \tau_0 \geq \underline \tau$, then
	\[
		u_c(\tau_1,0) \geq E_h(\tau_1) - \eps.
	\]
\end{proposition}

A crucial point is that $\underline\tau$ does not depend on $c$; that is the lower bound is uniform over all admissible $c$.  We prove \Cref{p.growing_u_c} in \Cref{s.technical_lemmas}.  We show how to conclude \Cref{p.weak}.(i) assuming \Cref{p.growing_u_c}.

\begin{proof}[Proof of \Cref{p.weak}.(i)]
We prove this by contradiction.  Assume that $\liminf E_h \geq \theta_{\rm crit} + \eps$ for some $\eps>0$.  We may find $\tau_0>0$ such that $E_h(\tau) \geq \thc + \eps/2$ for all $\tau \geq \tau_0$.

Notice that, for $\tau \geq \tau_0$,
\[
	18 E_h(\tau)
		- 4 \cc^2
		\geq 18(\thc + \eps/2) - 4 \cc^2
		= 18 \thc + 9 \eps - 4(2\thc)^{-2}
		= 9 \eps.
\]
In the last line we used that $18 \thc^3 =1$.  Thus, a straightforward application of \Cref{p.growing_u_c} yields
\be\label{e.c1223}
	\liminf_{\tau\to\infty} \inf_{|y| \leq \cc} h(\tau,y)
		= \liminf_{\tau\to\infty} \inf_{|c| \leq \cc} u_c(\tau,0)
		\geq \theta_{\rm crit} + \eps/4.
\ee

Using~\eqref{e.c1223} and that $2 \cc \thc = 1$, we find
\[
	1
		= \liminf_{\tau\to\infty} \int h(\tau, y) \dy
		\geq \liminf_{\tau\to\infty} \int_{-\cc}^{\cc} h(\tau,y) \dy
		\geq 2\cc (\theta_{\rm crit} + \eps/4)
		= 1 + \frac{\cc \eps}{2}.
\]
This is clearly a contradiction, which concludes the proof.
\end{proof}

\section{The strong bounds with radial symmetry} \label{s.strong}

We now show the significantly stronger bounds under the assumption that $g_0$ is even and radially decreasing.  There are two major steps here.  First, we show that $E_h$ converges as $\tau\to\infty$.  Then we use that to obtain convergence of $h$.  The first step is restated in the following proposition:

\begin{proposition}\label{p.E_limit}
If $g$ solves~\eqref{e.main2} with initial data $0 \leq g_0 \in C_c(\R)$ satisfying $\int g_0 \dx = 1$ that is even and radially decreasing, then
\[
	\lim_{\tau\to\infty} E_h(\tau)  = \lim_{\tau\to\infty} M_h(\tau) = \thc.
\]
\end{proposition}

We perform these steps in reverse order.  First, in \Cref{s.h_converges}, we show that $h$ converges as claimed assuming that $E_h$ and $M_h$ do; that is, \Cref{p.E_limit} holds.  Then, in \Cref{s.E_limit} we prove \Cref{p.E_limit}.

%
%

\subsection{Convergence of $h$ given convergence of $E_h$ and $M_h$}
\label{s.h_converges}

We now show how to conclude \Cref{t.strong} via \Cref{p.E_limit}.  While we use radial symmetry in this step, it is not required and is only used here to simplify the proof.  Indeed, a close inspection of the arguments reveals how to construct explicit sub- and supersolutions of $u_c$ (and, therefore, $h$) in order to obtain such pointwise bounds without using radial symmetry.  

A main idea in the proof is that, by H\"older's inequality, $M_h$ and $E_h$ can only be equal if $h$ is an indicator function.  Since $h$ is even and radially decreasing, this implies that $h = \theta \1_{[-c,c]}$ for some $\theta$ and $c$, which must be $\thc$ and $\cc$ by the rescalings described in \Cref{s.heuristic}.  Of course, for all finite $t$, $M_h$ and $E_h$ are not exactly equal, so we must understand the stability of H\"older's inequality in our setting.  This is encoded in the following general lemma, which essentially yields that if $M_h \approx E_h$, then $h \approx \theta \1_{[-c,c]}$.
\begin{lemma}\label{l.holder_stability}
	Fix $d\in \N$, and let $\omega_d$ be the volume of the unit ball in $\R^d$.  Suppose that $f: \R^d \to [0,\infty)$ is a rotationally symmetric and radially decreasing function with $\int f \dx = 1$.  Then, for every $r>0$:
	\begin{enumerate}[(i)]
		\item if $x_0\in B_r $ and $\int_{B_r} f \dx < 1$,
			\[
				f(x_0) \geq M_f \frac{\frac{E_f}{M_f} - \int_{B_r} f \dx}{1 - \int_{B_r} f \dx}.
			\]
		\item if $x_0\notin B_r $,
			\[
				f(x_0) \leq \frac{1 - \int_{B_r} f \dx}{\omega_d(|x_0|^d - r^d)}.
			\]
	\end{enumerate}

\end{lemma}
%
%
%
Notice that if $\|f\|_2^2 = E_f = M_f = \|f\|_\infty$, we find $f(x) \geq M_f$ in (i) for all $|x| < (\omega_d M_f)^{-1/d}$.   Then, using this with the constraint $\int f \dx = 1$, we see that $f(x) = M_f$ for all $|x| < (\omega_d M_f)^{-1/d}$ and $f(x) \equiv 0$ for all $|x| > (\omega_d M_f)^{-1/d}$.  Hence, $f$ is an indicator function, which is the only function for which H\"older's inequality is an equality:
\[
	E_f
		= \int f^2 \dx
		\leq \|f\|_\infty \|f\|_1
		= M_f \int f \dx = M_f.
\]
This is why we describe \Cref{l.holder_stability} as a stability estimate for H\"older's inequality.  We note that when $E_f < M_f$, both (i) and (ii) are necessary to obtain precise bounds on $f$.  This is in contrast to the special case described above of $E_f = M_f$ where only (i) is used.

\Cref{l.holder_stability} is proved in \Cref{s.technical_lemmas}.  We now show how to deduce \Cref{t.strong} using the above results.

\begin{proof}[Proof of \Cref{t.strong}]
	First, we notice that \Cref{p.E_limit} yields the claim about the long time limits of $E_g$ and $M_g$ after suitable rescaling.  Next, we notice that the claim about the profile of $g$ in long times is equivalent to showing that
	\begin{equation}\label{e.c2131}
		h(\tau,y) \to \thc \1_{[-\cc, \cc]}
			\qquad \text{ as } \tau\to\infty,
	\end{equation}
	which we now show.

	Fix any $\eps>0$.  Applying \Cref{l.holder_stability}.(i) and using \Cref{p.E_limit}, we see that
	\[
		\liminf_{\tau\to\infty} \min_{|y| \leq \cc - \eps} h(\tau,y) \geq \thc,
	\]
	and, recalling that $h$ is even and radially decreasing,
	\[
		\limsup_{\tau\to\infty} \max_{|y| \leq \cc - \eps} h(\tau,y)
			\leq \limsup_{\tau\to\infty} h(\tau,0)
			= \limsup_{\tau\to\infty} M_h(\tau)
			= \thc.
	\]
	The last equality follows by \Cref{p.E_limit}.  Thus we have established~\eqref{e.c2131} in the case where $|y| < \cc$.  
	
	We now investigate the upper bound.  Applying \Cref{l.holder_stability}.(ii), we have
	\[
		h(\tau,\cc+\eps)
			\leq \frac{1 - \int_{-\cc}^{\cc} h(\tau,y') dy'}{2\eps}.
	\]
	From the lower bound established above, we have that $\lim \int_{-\cc}^{\cc} h(\tau,y) dy \to 1$ as $\tau \to \infty$.  Thus, we conclude that
	\[
		h(\tau,\cc + \eps) \to 0
			\qquad \text{ as } \tau \to\infty.
	\]
	Since $h$ is even and radially decreasing, we find
	\[
		\limsup_{\tau\to\infty} \sup_{|y|\geq \cc + \eps} h(\tau,y) = 0,
	\]
	which concludes the proof by the arbitrariness of $\eps$.
\end{proof}

\subsection{The convergence of $E_h$ and $M_h$}\label{s.E_limit}

In order to provide clearer references and to more closely mirror the structure of the proof, we split this into two separate propositions.
\begin{proposition}\label{p.E_limsup}
Under the assumptions of \Cref{p.E_limit},
\[
	\limsup_{\tau\to\infty} E_h(\tau)  = \limsup_{\tau\to\infty} M_h(\tau) = \thc.
\]
\end{proposition}

\begin{proposition}\label{p.E_liminf}
Under the assumptions of \Cref{p.E_limit},
\[
	\liminf_{\tau\to\infty} E_h(\tau)  = \liminf_{\tau\to\infty} M_h(\tau) = \thc.
\]
\end{proposition}
Clearly \Cref{p.E_limsup,p.E_liminf} imply \Cref{p.E_limit}.  Hence, we focus on proving each of the above propositions in turn in \Cref{s.E_h_above,s.E_h_below}, respectively.  Interestingly, the proof of \Cref{p.E_liminf} uses \Cref{p.E_limsup}, and, hence, the order in which these propositions are proved is important.

\subsubsection{Technical lemmas}

%

We begin by stating and proving two crucial lemmas.

\begin{lemma}\label{l.slow_E_growth}
	For any $\tau'>\tau>0$,
	\[
		E_h(\tau') \leq \left(\frac{\tau'+1}{\tau+1}\right)^2 E_h(\tau).
	\]
\end{lemma}
\begin{proof}
	Multiplying~\eqref{e.heq} by $h$ and integrating by parts implies that $\dot E_h	\leq \frac{2}{\tau + 1} E_h.$
	Solving this differential inequality yields the claim: $E_h(\tau') \leq \left( \frac{\tau'+1}{\tau+1}\right)^2 E_h(\tau).$
\end{proof}

\begin{lemma}\label{l.slow_M_growth}
	For any $\tau'>\tau>0$,
	\[
		M_h(\tau') \leq \left(\frac{\tau'+1}{\tau+1}\right)^2 M_h(\tau).
	\]
\end{lemma}
The proof of this fact is exactly as in \Cref{l.slow_E_growth}, though using the comparison principle in place of energy estimates, so we omit its proof.

We require another technical lemma; however, its proof is quite involved and so we postpone it until \Cref{s.technical_lemmas}.
\begin{lemma}\label{l.M_follows_E}
	Fix any $\eps>0$.  There exists $\tau_\eps>0$ and $\tau_{\rm shift}$, depending only on $\eps$, such that if $\tau \geq \tau_\eps$, then
	\[
		M_h(\tau + \tau_{\rm shift})
			\leq (1+ \eps)E_h(\tau).
	\]
	Moreover, there exists a sequence $\tau_n$ tending to infinity such that
	\[
		\lim_{n\to\infty} M_h(\tau_n)
			= \lim_{n\to\infty} E_h(\tau_n)
				= \liminf_{\tau\to\infty} E_h
				= \liminf_{\tau\to\infty} M_h.
	\]
\end{lemma}
Roughly, we require this, in conjunction with the two previous lemmas, to find large time intervals on which $M_h$ and $E_h$ take approximately the same value.

\subsubsection{The upper bound on $\limsup E_h$ and $\limsup M_h$: \Cref{p.E_limsup}}\label{s.E_h_above}

\

\begin{proof}[Proof of \Cref{p.E_limsup}]
By \Cref{l.limsup}, $\limsup M_h = \limsup E_h$; hence, for the remainder of this proof, we focus our attention only on $\limsup E_h$. We prove the bound on $E_h$ by contradiction. Suppose that
\[
	\limsup_{\tau\to\infty} E_h(\tau) = \theta > \thc.
\]
The main idea is the following.  First, find a large time interval where $E_h$ transitions between $\thc$ to $\theta$.  When $E_h \approx \thc$, the closeness of $M_h$ and $E_h$ forces $h$ to be $O(1)$ near $\cc$ (recall \Cref{l.holder_stability}).  Then, as $E_h$ grows to be approximately $\theta$, so does $h$, at least on the set $[-\cc,\cc]$.  It follows that the integral of $h$ is at least approximately $2\theta \cc > 2 \thc \cc = 1$, which is a contradiction.

Before beginning we set two parameters.  Fix
\be\label{e.c10161}
	\eps \in \left(0, \frac{\theta-\thc}{100(1+\thc)}\right)
		\quad\text{ and }\quad
	\underline c
		= \frac{\cc}{(1+2\eps)^2}.
\ee
such that
\be\label{e.c10162}
	2(\underline c - \eps)(\theta - 2\eps)
		> 1.
\ee
The above is possible since $\theta > \thc$ and $2 \thc \cc = 1$.

{\bf Step 1: a large time interval in which $E_h$ transitions from $\thc$ to $\theta$.}  

Let $\overline \tau>0$ be a large time to be determined.  By \Cref{p.weak}, we have that $\liminf E_h \leq \thc$.  Hence, we can find $\overline \tau < \tilde\tau_1 < \tau_2$ such that
\be\label{e.c8114}
	E_h(\tilde \tau_1) = \thc(1 + \eps),
		\quad
	E_h(\tau_2) = \theta - \eps,
		\quad\text{ and }\quad
	E_h(\tau) \in \left[ \thc(1 + \eps), \theta - \eps\right]
\ee
for all $\tau \in [\tilde \tau_1, \tau_2]$.  By \Cref{l.slow_E_growth}, 
\[
	\tau_2 \geq \tilde \tau_1 \sqrt{\tfrac{\theta - \eps}{\thc(1 + \eps)}}.
\]
We note that, by the choice of $\eps$~\eqref{e.c10161}, the coefficient of $\tilde \tau_1$ is greater than $1$ so that $\tau_2 \geq (1+2\rho) \tilde \tau_1$ for some $\rho >0$.  We note that $\rho$ can be chosen independent of $\eps$ over all $\eps$ satisfying~\eqref{e.c10161}.

Using \Cref{l.M_follows_E} and increasing $\overline \tau$ if necessary, we find a universal constant $\tau_{\rm shift}$ 
such that
\be\label{e.c8113}
	M_h(\tilde \tau_1 + \tau_{\rm shift})
		\leq \frac{1 + 2\eps}{1+\eps} E_h(\tilde \tau_1)
		= (1+2\eps)\thc.
\ee
Let $\tau_1 = \tilde\tau_1 + \tau_{\rm shift}$.  Up to increasing $\overline \tau$ in a way depending only on $\rho$, we have that $(1+\rho)\tau_{\rm shift} \leq \rho \tilde \tau_1$ and, hence
\be\label{e.c10152}
	(1+\rho)\tau_1
		= (1+\rho)(\tilde\tau_1 + \tau_{\rm shift})
		\leq (1+\rho)\tilde\tau_1 + \rho \tilde \tau_1
		= (1+2\rho)\tau_1
		\leq \tau_2.
\ee
%

Combining~\eqref{e.c8113} with the fact that $E_h(\tau_1) \geq \thc(1+\eps)$ by~\eqref{e.c8114}, we deduce the two following inequalities:
\be\label{e.c8112}
	M_h(\tau_1)
		\leq \thc(1+2\eps)
	\qquad\text{ and }\qquad
	\frac{E_h(\tau_1)}{M_h(\tau_1)}
		\geq \frac{1+\eps}{1+2\eps}.
\ee

%
%
%
%
%

{\bf Step 2: $h = O(1)$ near $\cc$.}    In this step, we obtain a preliminary bound on $h$ so that we may apply \Cref{p.growing_u_c}.

Applying \Cref{l.holder_stability}.(i), we find
\be\label{e.c8111}
	h(\tau_1, \underline c)
		\geq M_h(\tau_1) \frac{\frac{E_h(\tau_1)}{M_h(\tau_1)} - \int_{-\underline c}^{\underline c} h \dy}{1 - \int_{-\underline c}^{\underline c} h \dy}.
\ee
In order to obtain a lower bound on the right hand side above, we define the following auxiliary function.  For any $s< E_h/M_h \leq 1$, let
\[
	\phi(s) = \frac{\frac{E_h(\tau_1)}{M_h(\tau_1)} - s}{1-s}.
\]
Notice that $\phi$ is decreasing on its domain.  In addition, recalling the definition of $M_h$, the choice of $\underline c$~\eqref{e.c10161}, the fact that $2\cc \thc = 1$, and the bounds on $M_h$ and $E_h/M_h$~\eqref{e.c8112}, yields
\[\begin{split}
	\int_{-\underline c}^{\underline c} h \dy
		&\leq 2 \underline c M_h(\tau_1)
		= 2 \frac{\cc}{(1+2\eps)^2} M_h(\tau_1)\\
		&\leq 2 \frac{\cc}{1+2\eps} \thc
		= \frac{1}{1+ 2\eps}
		< \frac{1 + \eps}{1+2\eps}
		\leq \frac{E_h(\tau_1)}{M_h(\tau_1)}.
\end{split}\]
Thus, $\int_{-\underline c}^{\underline c} h \dy \leq 1/(1+2\eps)$ and both quantities are in the domain of $\phi$.  Hence, 
\[
	\frac{\frac{E_h(\tau_1)}{M_h(\tau_1)} - \int_{-\underline c}^{\underline c} h \dy}{1 - \int_{-\underline c}^{\underline c} h \dy}
		= \phi\left( \int_{-\underline c}^{\underline c} h \dy\right)
		\geq \phi\left(\frac{1}{1+2\eps}\right)
		= \frac{ \frac{E_h(\tau_1)}{M_h(\tau_1)} - \frac{1}{1+2\eps}}{1 - \frac{1}{1+2\eps}}.
\]
Using this in~\eqref{e.c8111} and applying~\eqref{e.c8112} yields
\[
	h(\tau_1, \underline c)
		\geq M_h(\tau_1) \frac{ \frac{E_h(\tau_1)}{M_h(\tau_1)} - \frac{1}{1+2\eps}}{1 - \frac{1}{1+2\eps}}
		\geq M_h(\tau_1) \frac{ \frac{1+\eps}{1+2\eps} - \frac{1}{1+2\eps}}{1 - \frac{1}{1+2\eps}}
		= \frac{M_h(\tau_1)}{2}.
\]
By \Cref{l.rescaled_bounds} and the fact that $h$ is radially decreasing, we conclude that
\be\label{e.c10163}
	\min_{|x| \leq \underline c} h(\tau_1, x)
		= h(\tau_1, \underline c)
		\gtrsim 1.
\ee

{\bf Step 3: $h$ grows up to $\theta$.}  Up to increasing $\underline\tau$, which also increases $\tau_2 - \tau_1$ (see~\eqref{e.c10152}), we now apply \Cref{p.growing_u_c} to conclude that
\[
	u_c(\tau_2, 0) \geq \theta - 2\eps
\]
for all $c \in [0,\underline c-\eps]$.  Here it was crucial that the lower bound~\eqref{e.c10163} on $h$ was $\mathcal{O}(1)$ and uniform over all $x\in [-\underline c, \underline c]$. 
Thus, we find
\[
	1
		= \int_{-\infty}^\infty h(\tau_2, c) dc
		\geq \int_{-\underline c + \eps}^{\underline c - \eps} u_c(\tau_2, 0) dc
		\geq 2(\underline c - \eps)(\theta - 2\eps).
\]
However, $2(\underline c - \eps)(\theta - 2\eps) > 1$ by~\eqref{e.c10162}.  This is a contradiction, concluding the proof.
\end{proof}

\subsubsection{Lower bound on liminf}\label{s.E_h_below}

We now prove \Cref{p.E_liminf}, which completes the proof of \Cref{p.E_limit}.  As in the previous section, we note that we do not directly use the symmetries of $g_0$ in a strong way here.  It only arises in this section through \Cref{p.E_limsup}, which was established using the symmetry.

A useful quantity in the sequel is the time average of $E_h$.  In general, given a function $f: [0,\infty) \to \R$, we define its time average to be, for any $\tau>0$,
\[
	\overline f(\tau) = \frac{1}{\tau} \int_0^\tau f(s) ds.
\]

The main heuristic of the proof is the following.  If $E_h$ oscillates below $\thc$ by a fixed amount at a large time, then so does $M_h$, by \Cref{l.M_follows_E}.  Hence, on $[-\cc, \cc]$, $h$ has to be smaller than $\thc$, which implies that
\[
	u_{\cc} < \thc \quad \text{ for } z<0.
\]
On the other hand, using \Cref{p.E_limsup}, the ``accumulated reaction'' $3 \int_0^\tau E_h(s) ds = 3 \tau \overline E_h(\tau)$ can be no larger than $3\tau \thc$ plus a small correction due to the difference in time domains between the integral and the interval on which $E_h$ is small.  Heuristically, this means the front $u_{\cc}$ in~\eqref{e.ueq} cannot exceed the speed $0$ in the moving frame (speed $\cc$ in the physical variables).  This yields
\[
	u_{\cc} = o(1) 
		\quad \text{ for } z> 0.
\]
Thus, after suitably quantifying everything, we obtain the contradiction:
\[\begin{split}
	\frac{1}{2}
		&= \int_0^\infty h(\tau,y) \dy
		= \frac{1}{\tau}\int_{-\cc (\tau+1)}^\infty u_{\cc}(\tau,z) \dz\\
		&< \underbrace{\frac{1}{\tau} \int_{-\cc (\tau+1)}^0 \thc \dz}_{h < \thc \text{ on } [-\cc,\cc]}
			+ \underbrace{\frac{1}{\tau} \int_0^\infty o(1) \dz}_{\substack{\text{$u_{\cc}$ is small}\\\text{beyond the front}}}
		\approx \thc \cc = \frac{1}{2}.
\end{split}\]


%
%


\begin{proof}[Proof of \Cref{p.E_liminf}]
	Recall that $E_h \leq M_h$, and hence that $\liminf E_h \leq \liminf M_h$.  By \Cref{p.E_limsup}, we have that $\limsup M_h = \thc$.  Hence, if we prove that $\liminf E_h = \thc$, we may conclude that $\liminf M_h = \thc$.  As such, we consider only $\liminf E_h$ for the remainder of the proof.

%

	We proceed by contradiction.  Assume that $\liminf E_h < \thc$.    By \Cref{l.M_follows_E}, there exists $\tau_n$ tending to infinity such that
	\[
		\lim_{n\to\infty} M_h(\tau_n), \lim_{n\to\infty} E_h(\tau_n)
			< \thc.
	\]

	Define $E_+(\tau) = \max\{E_h(\tau),\thc\}$. 	We note that $\overline E_+(\tau) \geq \thc$ for all $\tau$ and that $\lim \overline E_+ = \thc$.  The latter follows from the fact that, by \Cref{p.E_limsup}, $\limsup E_h = \thc$. Recall that $\overline E_+$ was defined as the time average of $E_+$.

	We now construct a super-solution of $u_{\cc}$ on the parabolic domain $\cP = (0,\infty)\times(0,\infty)$.  Indeed, for $A>0$ to be determined and any $(\tau,z) \in \overline\cP$, let
	\[
		\overline u(\tau,z)
			= A(\tau + 1) \exp\left\{3 \int_0^\tau E_+(s) ds - \frac{2 \cc^2}{3}\tau - \frac{2\cc}{3} z - \frac{z^2}{2(\tau+1)}\right\}.
	\]
	The first three factors in the exponential are typical of supersolutions for Fisher-KPP. 
	The last factor is to cancel the contributions 
	due to the $\frac{z}{\tau+1} \partial_z$ term in~\eqref{e.ueq}.
	
	Define the parabolic domain $\cP = (0,\infty)\times(0,\infty)$.  We aim to use the comparison principle to bound $u_{\cc}$ from above by $\overline u$ on $\cP$.  To that end, we first check that $\overline u \geq u_{\cc}$ on the parabolic boundary of $\cP$.  There are two components to check: when $\tau = 0$ and when $z = 0$.  In the former case, this is clearly satisfied by choosing $A$ sufficiently large as $u_{\cc}(\tau=0,\cdot) = g_0(\cc+ \cdot) \in C_c(\R)$.  In the latter case, we note that
	\[
		\overline u(\tau, 0)
			= A(\tau+1)\exp\left\{\tau \left( 3 \overline E_+(\tau) - \frac{2\cc^2}{3}\right)\right\}
			= A(\tau+1)\exp\left\{ 3\tau \left(\overline E_+(\tau) - \thc\right)\right\}.
	\]
	where the second equality follows by the definition of $\thc$ and $\cc$ (see~\eqref{e.cc_thc}).  Since, by construction, $\overline E_+ \geq \thc$, it follows that $\overline u(\tau,0) \geq A$ for all $\tau \geq 0$.  By \Cref{t.decay} and its definition, $u_{\cc}$ is bounded above uniformly for all $\tau$.  Hence, there is some $A$ such that $\overline u(\tau,0) \geq u_{\cc}(\tau,0)$ for all $\tau\geq 0$.  We conclude that $\overline u \geq u_{\cc}$ on the parabolic boundary of $\cP$.
	
	The last step is to check that $\overline u$ is, in fact, a supersolution of~\eqref{e.ueq} in $\cP$.  We compute:
	\[\begin{split}
		\overline u_\tau
			&- \frac{3}{2} \Delta \overline u
			- 3\overline u(E_h - \overline u)
			- 2\cc \partial_z \overline u
			- \frac{2}{\tau+1} \overline u
			- \frac{z}{\tau+1} \partial_z \overline u\\
			&=
				\overline u\ \Bigg(
					\frac{1}{\tau+1} + 3 E_+ - \frac{2\cc^2}{3} + \frac{z^2}{2(\tau+1)^2}
					- \frac{3}{2} \left( \left(\frac{2\cc}{3} + \frac{z}{(\tau+1)}\right)^2 - \frac{1}{\tau+1}  \right)\\
					&\qquad - 3 E_h + 3\overline u
					+ 2\cc \left( \frac{2\cc}{3} + \frac{z}{\tau+1}\right)
					- \frac{2}{\tau+1}
					+ \frac{2\cc z}{3(\tau+1)} + \frac{z^2}{(\tau+1)^2} \Bigg)\\
			&=
				\overline u\ \Bigg( \frac{1}{2(\tau+1)} + 3(E_+ - E_h) + \frac{2 \cc z}{\tau+1} + 3\overline u\Bigg)
			\geq 0.
	\end{split}\]
	In the inequality, we used that $E_+ \geq E_h$ and $\overline u \geq 0$.  Hence $\overline u$ is a supersolution of~\eqref{e.ueq}.
	
	Combining all the work above, we apply the comparison principle and conclude that $\overline u \geq u_{\cc}$ on $\cP$.
	
	We now conclude the proof.  Let
	\[
		\overline c = \frac{1}{\thc + \lim M_h(\tau_n)} > \frac{1}{\thc + \thc} = \cc
	\]
	Using that $\int h \dy = 1$ and that $h$ is even, we have that
	\begin{equation}\label{e.c2141}
	\begin{split}
		\frac{1}{2}
			&= \int_0^\infty h(y) \dy
			= \int_0^{\overline c} h(y) \dy
				+ \frac{1}{\tau+1} \int_{(\overline c - \cc)(\tau+1)}^\infty u_{\cc}(\tau,z) \dz\\
			&\leq \overline c M_h(\tau) + \frac{1}{\tau+1}\int_{(\overline c - \cc)(\tau+1)}^\infty \overline u(\tau,z)\dz\\
			&= \overline c M_h(\tau) + A  e^{3 \tau\left(\overline E_+ - \frac{2 \cc^2}{9}\right)} \int_{(\overline c - \cc)(\tau+1)}^\infty e^{- \frac{2\cc}{3} z - \frac{z^2}{2(\tau+1)}} \dz\\
			&\leq \overline c M_h(\tau) + \frac{3A}{2 \cc} \exp\left\{3 \tau\left(\overline E_+ - \frac{2 \cc^2}{9} - \frac{2\cc}{9}(\overline c - \cc) \frac{\tau+1}{\tau}\right)\right\}.
	\end{split}
	\end{equation}
	Recall that $\lim \overline E_+ = \thc = 2 \cc^2/9$ and $\overline c > \cc$.  Evaluating the above at $\tau_n$ and taking $n$ to infinity, we find
	\[
		\frac{1}{2}
			\leq \overline c \lim M_h(\tau_n) + 0 
			= \frac{\lim M_h(\tau_n)}{\thc + \lim M_h(\tau_n)}
			< \frac{1}{2},
	\]
	where the last inequality follows because $\lim M_h(\tau_n) < \thc$.  This is a contradiction.  Hence, the proof is complete.
\end{proof}

\subsection{Proof of technical lemmas}\label{s.technical_lemmas}

\subsubsection{The proof of \Cref{l.holder_stability}}

There are two parts to prove.  We show them in order.

\begin{proof}[Proof of \Cref{l.holder_stability}.(i)]
Fix $r$ as in the statement of the lemma.  We abuse notation by denoting $f(r) = f(x)$ for any $|x| = r$.  We compute:
	\[\begin{split}
		E_f
			\leq \int_{B_r} f^2 \dx  + \int_{B_r^c} f^2 \dx
			&\leq M_f \int_{B_r} f \dx + f(r) \int_{B_r^c} f \dx\\
			&= M_f \int_{B_r} f \dx + f(r) \left(1 - \int_{B_r} f \dx\right).
	\end{split}\]
	Recall that $\int_{B_r} f dx < 1$.  Thus, rearranging the above yields
	\[
		M_f \frac{\frac{E_f}{M_f} - \int_{B_r} f \dx}{1 - \int_{B_r} f \dx}
			= \frac{E_g - M_f \int_{B_r} f \dx}{1 - \int_{B_r} f \dx}
			\leq f(r).
	\]
	This completes the proof.
\end{proof}
%
%
%
%
%

%
%
%
%

\begin{proof}[Proof of \Cref{l.holder_stability}.(ii)]
	First we compute:
	\[
		1
			= \int f \dx
			\geq \int_{B_r} f \dx + \int_{B_{|x_0|}\setminus B_r} f \dx
			\geq \int_{B_r} f \dx + \omega_d(|x_0|^d - r^d) f(x_0).
	\]
	The last inequality uses the fact that $f$ is rotationally symmetric and radially decreasing.  The proof is concluded by rearranging the above inequality.
\end{proof}

\subsubsection{The proof of \Cref{l.M_follows_E}}\

\begin{proof}
We begin with the first claim.  Fix $\tau_\eps$ and $\tau_{\rm shift}$ to be determined.  In order to prove this, we first note that $E_h$ cannot stray too far above $E_h(\tau)$.  Indeed, by \Cref{l.slow_E_growth}, we have
\[
	E_h(\tau')
		\leq \left( \frac{\tau'}{\tau}\right)^2 E_h(\tau).
\]
Hence, there is $\mu_\eps>0$ such that if $\tau_{\rm shift} < \mu_\eps \tau$ then $E_h(\tau') \leq (1+ \eps/4) E_h(\tau)$ for all  $\tau' \in [\tau, \tau + \tau_{\rm shift}]$.  Notice that $\mu_\eps$ depends only on $\eps$.

We now construct a simple super-solution that ``pushes'' $M_h$ down to $E_h$.  Indeed, for $\tau' \geq \tau$, let
\[
	\overline h(\tau',y)
		= \left(1+\frac{\eps}{2}\right)\left[ M_h(\tau) e^{-\gamma (\tau'-\tau)}
			+ \left(1 -  e^{-\gamma (\tau'-\tau)}\right) E_h(\tau)\right],
\]
where $\gamma>0$ is a constant to be determined.  Notice that $\overline h(\tau,y) \geq M_h(\tau) \geq h(\tau,y)$ for every $y$.

We next show that $\overline h$ is a supersolution of~\eqref{e.heq}.  Indeed, for any $\tau' \in [\tau, \tau+\tau_{\rm shift}]$,
\[\begin{split}
	\partial_{\tau'} \overline h
		&- \frac{3}{2(\tau'+1)^2} \Delta \overline h
		- 3 \overline h (E_h(\tau') - \overline h)
		- \frac{2}{\tau'+1}\left( \overline h + y \partial_y \overline h\right)\\
		&= \gamma (1+\eps/2) E_h(\tau) + 3\overline h \left( \overline h - \frac{\gamma}{3} -  E_h(\tau') - \frac{2}{3(\tau' + 1)}\right)\\
		&\geq \gamma (1+\eps/2) E_h(\tau) + 3\overline h \left( \overline h - \frac{\gamma}{3} -  \Big(1 + \frac{\eps}{4}\Big)E_h(\tau) - \frac{2}{3(\tau_\eps + 1)}\right).
\end{split}\]
In the last step, we used that $E_h(\tau') \leq (1+\eps/4)E_h(\tau)$ for all $\tau' \in [\tau, \tau+\tau_{\rm shift}]$.  
Recall that $M_h(\tau) \geq E_h(\tau)$.  Thus, $\overline h \geq (1+\eps/2) E_h(\tau)$, and we find
\[\begin{split}
	\partial_\tau \overline h
		&- \frac{3}{2(\tau'+1)^2} \Delta \overline h
		- 3 \overline h (E_h(\tau') - \overline h)
		- \frac{2}{\tau'+1}\left( \overline h + y \partial_y \overline h\right)\\
	&\geq \gamma (1+\eps/2) E_h(\tau) + 3\overline h \left( \frac{\eps}{4} E_h(\tau) - \frac{\gamma}{3} - \frac{2}{3(\tau_\eps + 1)}\right).
\end{split}\]
From \Cref{t.decay}, there exists $C>0$ such that $C^{-1} < E_h(\tau)\leq M_h(\tau) < C$ for all $\tau \geq 1$.  Thus, choosing $\gamma = \eps/(24 C)$ and increasing $\tau_\eps$ if necessary, we find
\[
	\partial_\tau \overline h
		- \frac{3}{2(\tau'+1)^2} \Delta \overline h
		- 3 \overline h (E_h(\tau') - \overline h)
		- \frac{2}{\tau'+1}\left( \overline h + y \partial_y \overline h\right)
		\geq 0.
\]

It follows from the comparison principle that $\overline h(\tau',\cdot) \geq h(\tau',\cdot)$ for all $\tau' \in [\tau, \tau + \tau_{\rm shift}]$.  Thus, we conclude that
\[
	M_h(\tau + \tau_{\rm shift})
		\leq \overline h(\tau + \tau_{\rm shift},0)
		\leq (1+\eps/2) C e^{-\frac{\eps}{24 C} \tau_{\rm shift}} + (1+ \eps/2) E_h(\tau).
\]
It is clear that if $\tau_{\rm shift}$ is chosen large enough, depending only on $\eps$, then the right hand side above is bounded by $(1+\eps)E_h(\tau)$, as desired.  This finishes the proof of the first claim.

We now prove the second claim.  First note that $\liminf E_h \leq \liminf M_h$ since $E_h \leq M_h$.  Hence, if we find the sequence $\tau_n$ in the statement of the lemma, then it follows that $\liminf E_h = \liminf M_h$ and the proof is concluded.  In order to find such a sequence, it is enough to establish the following claim:
\[
	\text{for every $\eps, \tau_0>0$, there is $\tau_1\geq \tau_0$ such that }
		M_h(\tau_1), E_h(\tau_1) \leq (1+\eps)^2\liminf E_h.
\]
It is clear that the procedure used to establish the first claim yields the desired $\tau_1$.  Indeed, after choosing, in the notation above, $\tau_\eps >\tau_0$, $\tau \geq \tau_\eps$ such that $E_h(\tau) \leq (1+\eps)\liminf E_h$, and letting $\tau_1 = \tau + \tau_{\rm shift}$, we find
\[\begin{split}
	&M_h(\tau_1) \leq (1+\eps) E_h(\tau)
		\leq (1+\eps)^2\liminf E_h
		\quad\text{ and }\quad\\
	&E_h(\tau_1) \leq (1+\eps/4)E_h(\tau)
		\leq (1+\eps/4)(1+\eps) \liminf E_h.
\end{split}\]
This concludes the proof.
\end{proof}

\subsubsection{The proof of \Cref{p.growing_u_c}} \label{s.growing_u_c}

%
%

We break this proposition into two smaller lemmas. The first (\Cref{l.weak_lower_bound}) shows that $u_c$ remains locally uniformly bounded below over the entire time interval $[\tau_0,\tau_1]$, while the second (\Cref{l.strong_lower_bound}) shows that $u_c$ grows from this initial weak lower bound to the claimed value $E(\tau_1) - \eps$ over a terminal boundary layer $[\tau_1(1-\beta_\eps), \tau_1]$ for $\beta_\eps \approx \eps$.  In this final step, it is crucial that $E_h$ grows ``slowly'' as in \Cref{l.slow_E_growth}.  We state these lemmas here and then show how to use them to conclude \Cref{p.growing_u_c}.   Afterwards, we prove them.

Our first lemma is below.  Similar results exist in the literature (see, e.g., a very general work of Berestycki, Hamel, and Nadin~\cite{BerestyckiHamelNadin}); however, we are unable to find one that, applied out-of-the-box, yields the result below with the uniformity in all parameters and allows for the particular assumptions that we require.  As such, we provide a proof below, although the ideas are standard.
\begin{lemma}\label{l.weak_lower_bound}
	Fix any $T_0 < T_1$, $\mu$, $\delta$, and $c$.  Suppose that $R > 1 + 10\pi/\sqrt\delta$ and 
	\begin{equation}\label{e.ueq_subsol}
	\begin{cases}
		\partial_\tau u
			= \frac{3}{2} \Delta u + 3 u(f(\tau) - u)
				+ 2c \partial_z u + a(\tau,z) \partial_z u
					\qquad &\text{ in } (T_0,T_1)\times B_{2R},\\
		u \geq \mu \1_{B_R }
					\qquad &\text{ on } \{T_0\}\times B_{2R}( 0),
	\end{cases}
	\end{equation}
	where $f$ and $a$ are continuous and $f$ satisfies
	\begin{equation}\label{e.c2174}
		18 f(\tau) - 4 c^2 > \delta \ \ \text{ for all } \tau\in[T_0,T_1].
	\end{equation}
	If $\|a\|_{L^\infty([T_0,T_1]\times B_{2R})}$ is sufficiently small depending only on $\delta$ and $c$, then there exists $C_{\delta,f}$, depending only on $\delta$ and $\|f\|_\infty$, and $T$, depending only on $\|f\|_{\infty}$, $\|a\|_\infty$, $\mu$, $\delta$, and $R$, such that, if $\tau \in [T_0 + T, T_1]$, then
	\[
		\min_{z \in B_R } u(\tau,z) \geq \frac{1}{C_{\delta,f}}.
	\]
\end{lemma}

The heuristic of \Cref{l.weak_lower_bound} is that if the reaction $f$, advection $a$, and speed $c$ satisfy a sub-minimal speed condition~\eqref{e.c2174} in a uniform way, then $u$ propagates at speed $c$ in the sense that it remains $O(1)$ regardless of the particular fluctuations of $f$ and $g$.  We note that it is crucial for our estimates that this is uniform in $c$ and $f$.

Our next lemma is the following.
\begin{lemma}\label{l.strong_lower_bound}
	Fix any $\tau_0$, $\tau_1$, $\mu$, $R$, $\rho$, $\eps$, $\delta$, and $c$.  Suppose that $u$ solves
	\[\begin{cases}
		\partial_\tau u
			= \frac{3}{2} \Delta u + 3 u(\rho - u)
				+ 2c \partial_z u + \frac{2}{\tau + 1} u + \frac{z}{\tau+1} \partial_z u
					\qquad &\text{ in } (\tau_0,\tau_1)\times B_R ,\\
		u = \mu
					\qquad &\text{ on } \{\tau_0\}\times B_R,\\
		u = 0
			\qquad &\text{ on } (\tau_0, \tau_1)\times \partial B_R
	\end{cases}\]
	and $18 \rho - 4 c^2 > \delta$.  If $\tau_0$ and $R$ are sufficiently large, depending only on $\delta$ and $\epsilon$, then there exists $\underline \tau$, depending only on $\mu$, $\eps$, and $\delta$, such that, if $\tau_1 - \tau_0 \geq \underline \tau$, then
	\[
		u(\tau_1,0) \geq \rho - \eps.
	\]
\end{lemma}

The main difference between the following and \Cref{l.weak_lower_bound} is that the lower bound at the final time is much stronger at the cost of having a constant reaction term $3\rho$.  In particular, as long as the sub-minimal speed condition ($18\rho-4c^2 > \delta$) is satisfied and $\tau$ is sufficiently large that the last two terms in~\eqref{e.ueq} are negligible, then $u$ must grow to the carrying capacity $\rho$ {\em uniformly} among all speeds.

We first show how to conclude \Cref{p.growing_u_c} from \Cref{l.weak_lower_bound,l.strong_lower_bound}.

\begin{proof}[Proof of \Cref{p.growing_u_c}]
First we notice that, by \Cref{l.slow_E_growth}, there is $\beta_\eps\in(0,1)$, depending only on $\eps$, such that
	\[
		E_h(\tau_1)
			\leq \left(\frac{\tau_1+1}{\tau+1}\right)^2 E_h(\tau)
			\leq \frac{1}{(1-\beta_\eps)^2} E_h(\tau)
			\leq \frac{1}{1 - \frac{\eps}{2E_h(\tau_1)}} E_h(\tau),
	\]
 	for all $\tau \in [\tau_1(1 - \beta_\eps), \tau_1]$.  Recall that $E_h(\tau_1)$ is bounded uniformly above and below by \Cref{t.decay}, and, hence, $\beta_\eps$ can be chosen independent of $E_h(\tau_1)$.  Thus,
	\begin{equation}\label{e.c2171}
		E_h(\tau) \geq E_h(\tau_1) - \frac{\eps}{2}
			\qquad\text{ for all } \tau \in [\tau_1(1-\beta_\eps), \tau_1].
	\end{equation}
	
	Fix $R> 1 + 10\pi/\sqrt \delta$ large enough such that \Cref{l.strong_lower_bound} can be applied with the choice
	\be\label{e.c761}
		\rho 
			= \min_{\tau'\in[\tau_1(1-\beta_\eps),\tau_1]} E_h(\tau')
			\ \ \left(\geq \max \left\{E(\tau_1) - \frac{\eps}{2},\frac{1}{C}\right\}\right)
	\ee
	and the $\eps$ in \Cref{l.strong_lower_bound} taking the value of $\eps/2$ in the current proof.  As $18 E_h - 4c^2>\delta$, by assumption, it follows that $18 \rho - 4c^2 > \delta$.
	
	Fix an intermediate time $\tau_{\rm int} \in [\tau_1(1-\beta_\eps),\tau_1]$ to be determined.  Applying \Cref{l.weak_lower_bound} on the time interval $[\tau_0,\tau_{\rm int}]$, we find that, up to increasing $\tau_0$ if necessary (to make the coefficients in~\eqref{e.ueq} sufficiently small) and choosing $\tau_{\rm int} - \tau_0$ sufficiently large, there is $C_\delta>0$ such that
	\begin{equation}\label{e.c2172}
		u_c(\tau_{\rm int}(1-\beta_\eps), z)
			\geq \frac{1}{C_\delta}
				\qquad\text{ for all } z \in B_R.
	\end{equation}
	
	We now let $\underline u$ be the solution of
	\[\begin{cases}
		\partial_\tau \underline u = \frac{3}{2} \Delta \underline u + 3 \underline u(\rho - \underline u) + 2 c \partial_z \underline u + \frac{2}{\tau+1}\underline u + \frac{z}{\tau+1} \partial_z \underline u,
			~~ &\text{ in } (\tau_{\rm int}, \infty) \times B_R ,\\
		\underline u = \frac{1}{C_\delta},
			~~ & \text{ on } \{\tau_{\rm int}\}\times B_R ,\\
		\underline u = 0
			~~ &\text{ on } (\tau_{\rm int}, \infty) \times \partial B_R .
	\end{cases}	\]
	By~\eqref{e.c2172}, $\underline u \leq u_c$ on $\{\tau_{\rm int}\}\times B_R $.  By the choice of boundary conditions $\underline u \leq u_c$ on $[\tau_{\rm int},\tau_1]\times \partial B_R $.  By our choice of $\rho$ and by~\eqref{e.c2171}, $\underline u$ is a subsolution of~\eqref{e.ueq}.  Thus, the comparison principle implies that $\underline u \leq u_c$ on $[\tau_{\rm int},\tau_1]\times B_R $.
	
	Up to increasing $\tau_1-\tau_0$ (recall that $\tau_{\rm int}$ has already been fixed relative to $\tau_0$), we have that $\tau_1 - \tau_{\rm int}$, the length of the time interval $[\tau_{\rm int},\tau_1]$, is sufficiently large to apply \Cref{l.strong_lower_bound} and conclude that
	\[
		E(\tau_1) - \eps
			\leq \rho - \frac{\eps}{2}
			\leq \underline u(\tau_1,0)
			\leq u_c(\tau_1,0).
	\]
	The first inequality is by~\eqref{e.c761}, the second by \Cref{l.strong_lower_bound}, 
	and the third by the ordering $\underline u\leq u_c$ outlined in the previous paragraph.   This concludes the proof.
\end{proof}


We now prove the two lemmas.

\begin{proof}[Proof of \Cref{l.weak_lower_bound}]
	Let $\rho = 10 \pi / \sqrt \delta$.  By assumption $R + \rho + 1 < 2R$.  By the Harnack inequality, there exists $\underline \mu$, depending only on $\|f\|_{L^\infty[T_0,T_1]}$, $\|g\|_{L^\infty([T_0,T_1]\times B_{2R})}$, $\mu$, and $\rho$, such that
	\[
		u(T_0 + 1, \cdot) \geq \underline \mu\1_{B_{R + \rho}}.
	\]
	In particular, we have the bound $u(T_0 + 1, \cdot) \geq \underline \mu \1_{B_\rho(z_0)}$ for any $z_0 \in B_R$.  Our goal is to obtain a lower bound on $u(T_1,z_0)$ via the construction of a subsolution $\underline u$ on $[T_0+1,T_1]\times B_\rho(z_0)$.  For notational ease, we assume $z_0 = 0$ and $T_0 + 1 = 0$ for the remainder of the proof since all arguments are translation invariant.

	We now construct $\underline u$.  Let $\lambda>0$ and $\phi: [0,T_1]\to \R$ be a constant and function determined, respectively.  Then let
	\[
		\underline u(\tau,z)
			= \underline \mu e^{-\lambda (\rho + z)} e^{\int_0^\tau \phi(s)ds} \cos\left(\frac{z\pi}{2\rho}\right)^2.
	\]
	Up to decreasing $\underline \mu$, we can choose $\phi$ in the sequel such that
	\begin{equation}\label{e.c2173}
		\underline u \leq \frac{\delta}{100}
			\qquad \text{ on } [0,T_1]\times B_\rho.
	\end{equation}
	
	Notice that $\underline u \leq\underline\mu \leq u$ on $\{0\}\times B_\rho$.  Moreover, $\underline u \leq u$ on $\partial B_\rho$ since $u$ is positive.  Hence, we need only check that $\underline u$ is a subsolution of~\eqref{e.ueq_subsol} in order to conclude, via the comparison principle, that $\underline u \leq u$ on $[0,T_1]\times B_\rho$.  To this end, we compute that, using~\eqref{e.c2173},
	\[\begin{split}
		&\partial_\tau \underline u
			- \frac{3}{2} \Delta \underline u
			- 3 \underline u (f - \underline u)
			- (2 c + a) \partial_z \underline u\\
		&\leq
			\phi \underline u
				- \frac{3}{2}
					\Big(\lambda^2 \underline u
					+ \frac{2 \pi\lambda \underline\mu e^{- \lambda( \rho+z) + \int_0^\tau \phi(s) ds}}{\rho} \cos \sin
%
				+ \frac{\pi^2 \underline\mu e^{- \lambda (\rho +z)+ \int_0^\tau \phi(s) ds }}{2 \rho^2} \left(\sin^2 -\cos^2\right)\Big)\\
		&\qquad
				- 3 \underline u \left(f  - \frac{\delta}{100}\right)
				+ (2c + a) \left(\lambda \underline u + \frac{\pi\underline\mu e^{-
 \lambda (\rho+z) + \int_0^\tau \phi(s) ds}}{\rho} \cos  \sin\right)\\
 		&= -\underline u \left( \frac{3 \lambda^2}{2} + 3 f  - 3 \frac{\delta}{100} - \frac{\pi^2}{2\rho^2} - (2c + a) \lambda - \phi\right)
 			+ \underline u \frac{\sin}{\cos} \frac{(2c + a - 3\lambda )\pi}{\rho}
 			- \underline u  \frac{\sin^2}{\cos^2} \frac{\pi^2}{2\rho^2}.
%
	\end{split}\]
	We now select $\lambda = 2c/3$.  Using this choice of $\lambda$ and condition~\eqref{e.c2174} to find
	\[\begin{split}
		&\partial_\tau \underline u
			- \frac{3}{2} \Delta \underline u
			- 3 \underline u (f - \underline u)
			- (2 c + a) \partial_z \underline u\\
			&\leq
				-\underline u \left( \frac{\delta}{10}  - \frac{\pi^2}{2\rho^2} - \frac{2c a}{3} - \phi\right)
 			+ \underline u \frac{\sin}{\cos} \frac{a \pi 2c}{3\rho}
 			- \underline u  \frac{\sin^2}{\cos^2} \frac{\pi^2}{2\rho^2}.
	\end{split}\]
	Next,  see that, by Young's inequality
	\[
		\frac{\sin}{\cos} \frac{a \pi 2c}{3\rho}
			\leq \frac{\sin^2}{\cos^2} \frac{\pi^2}{2\rho^2}
				+ \frac{2 a^2 c^2}{9}.
	\]
	Thus,
	\be\label{e.c391}
		\partial_\tau \underline u
			- \frac{3}{2} \Delta \underline u
			- 3 \underline u (f - \underline u)
			- (2 c + a) \partial_z \underline u\\
			\leq
				-\underline u \left( \frac{\delta}{10} - \frac{\pi^2}{2\rho^2} - \frac{2c a}{3} - \phi - \frac{2 a^2 c^2}{9}\right).
	\ee
	Using the choice $\rho = 10 \pi / \sqrt{\delta}$ and letting
	\[
		\phi(\tau) = \frac{\delta}{100} \1_{[0,T]}(\tau)
			\qquad \text{ with } T = \frac{100}{\delta} \log\left(\max\left\{1,\frac{\delta}{100 \underline\mu}\right\}\right),
	\]
	(clearly $T\leq T_1$ if $T_1$ is sufficiently large), we find
	\[
		\partial_\tau \underline u
			- \frac{3}{2} \Delta \underline u
			- 3 \underline u (f - \underline u)
			- (2 c + g) \partial_z \underline u\\
			\leq -\underline u \left( \frac{\delta}{25} - \frac{2c a}{3} - \frac{2 a^2 c^2}{9}\right)
	\]
	Thus $\partial_\tau \underline u - \frac{3}{2} \Delta \underline u - 3 \underline u (f - \underline u) - (2 c + a) \partial_z \underline u \leq 0$ if $\|a\|_\infty$ is sufficiently small.  We conclude, via the comparison principle, that $\underline u \leq u$.  Using the form of $\phi$ yields
	\[
		\underline u(T_1,0)
			\geq \frac{\delta}{100} e^{-\lambda \rho}.
	\]
	Recall that $\rho$ is given explicitly in terms of $\delta$ and $\lambda$ is bounded by $3 \sqrt{\|f\|_\infty/2}$ due to~\eqref{e.c2174}.  Hence $u(T_1,0)$ is bounded below in a way only depending on $\delta$ and $\|f\|_\infty$, as claimed.  This concludes the proof.
%
%
\end{proof}

\begin{proof}[Proof of \Cref{l.strong_lower_bound}]
	We prove this by contradiction.  Suppose there exists $\underline \tau_n$, $\overline \tau_n$, and $R_n \leq \sqrt{\underline\tau_n}$ all tending to infinity such that $\underline u(\underline \tau_n, \cdot) \geq \mu$ on $B_{R_n}(0)$ and $\underline u(\underline\tau_n + \overline \tau_n, 0) < \rho-\eps$.
	
	Let $u_n(\tau, z) = \underline u(\underline\tau_n + \overline \tau_n + \tau, z)$.  By the maximum principle, $\underline u \leq \max\{\mu, \rho + 2/3\}$. Thus $u_n$ is uniformly bounded in $L^\infty$, which, using parabolic regularity theory\footnote{See, e.g.,~\cite[Theorem 4.9, Theorem 6.9]{Lieberman} which correspond to Schauder and De Giorgi estimates, respectively.  Also, $C^{2+\alpha}_{\rm parabolic}$ mentioned here refers to the standard parabolic H\"older spaces.  Roughly, this corresponds to $C^{1+\alpha/2}$ regularity in $t$ and $C^{2+\alpha}$ regularity in $z$.}, yields a uniform bound in $C^{2+\alpha}_{\rm parabolic}$ for any $\alpha \in(0,1)$.  Thus, there exists $u_\infty$ such that $u_n \to u_\infty$ locally uniformly in $C^2_{\rm parabolic}$ and $u_\infty$ solves
	\[
		\partial_\tau u_\infty
			= \frac{3}{2} \Delta u_\infty
				+ 3 u_\infty(\rho - u_\infty)
				+ 2 c \partial_z u_\infty
					\qquad\text{ in } \R \times \R.
	\]
	Applying \Cref{l.weak_lower_bound}, we have that $u_\infty \geq \underline \mu$ on $\R\times \R$ for some $\underline \mu > 0$.  
	
	Notice that
	\[
		u_\infty(0,0)
			= \lim_{n\to\infty} \underline u(\underline\tau_n + \overline \tau_n, 0)
			\leq \rho - \eps,
	\]
	where the inequality follows by assumption, it follows that $\inf_{\R\times\R} u_\infty \in [\underline\mu, \rho-\eps]$.  There are two cases to consider.
	
	{\bf Case one: $u_\infty$ achieves its minimum at a point $(\tau_{\min}, z_{\min})$.}  Then, $u_\infty(\tau_{\min},z_{\min}) \in [\underline\mu,\rho-\eps]$ and, at $(\tau_{\min}, z_{\min})$, we have
	\[
		0
			\geq \partial_\tau u_\infty - \frac{3}{2} \Delta u_\infty - 2c \partial_z u_\infty
			= 3 u_\infty (\rho - u_\infty).
	\]
	This is a contradiction since $u_\infty(\tau_{\min},z_{\min}) \in [\underline\mu,\rho-\eps]$ and, hence, at $(\tau_{\min},z_{\min})$,
	\[
		u_\infty (\rho - u_\infty )
			\geq \min\{\underline \mu(\rho - \underline \mu), \eps(\rho-\eps)\}
			>0.
	\]
	Thus, this case cannot occur.
	
	{\bf Case two: $u_\infty$ does not achieve a minimum.}  Here we use a fairly standard re-centering trick.  Indeed, let $(\tau_n, z_n)$ be a sequence such that $u_\infty(\tau_n, z_n) \to \inf u_\infty$ as $n \to\infty$.  Let
	\[
		v_n (\tau,z) = u_\infty(\tau_n + \tau, z_n+z)
			\qquad \text{ for all } (\tau,z) \in \R\times\R.
	\]
	We conclude as above that $v_n \to v_\infty$ for some smooth function $v_\infty$ solving the same equation as $u_\infty$.  In addition, $v_\infty(0,0) = \inf v_\infty \in [\underline \mu, \rho - \eps]$.  At this point, the proof proceeds exactly as in case one, leading to a contradiction.
	
	Since we obtain a contradiction in all cases, the proof is finished.
\end{proof}

\section{Long time dynamics of solutions in higher dimensions}\label{s.higher_dimensions}

We begin with the (somewhat simpler) proof of the moments estimate \Cref{t.hd_moments}.  Afterwards we proceed with the construction of non-Gaussian self-similar solutions for the equation in two dimensions.  



\subsection{Higher dimensional moments estimates}

%
As in the one-dimensional case, the key estimate to establish is an upper bound on $E_g$.  We state this here and prove it in \Cref{s:hd_decay}.

\begin{proposition}\label{p.hd_decay}
	For any $d\geq 1$ and $t>0$,
	\[
		E_g(t)
			\lesssim \left(t+ E_g(0)^{-\frac{2}{d}}\right)^{-\frac{d}{2}}.
	\]
\end{proposition}

We note that, in the course of establishing \Cref{t.hd_moments} from \Cref{p.hd_decay}, we obtain a similar bound on $M_g$.

\subsubsection{Moment bounds}

We show how to conclude bounds on the moments of $g$ using \Cref{p.hd_decay}.  The main difficulty is in establishing the upper bounds on the moments as all other conclusions in \Cref{t.hd_moments} are either obtained along the way or a simple consequence of \Cref{p.hd_decay} and the moment upper bound.

We establish these upper bounds through the construction of a supersolution.  In the one dimensional case, this was made up of a solution to the heat equation with an exponential integrating factor depending on $E_g$.  
%
%
Trying to apply this directly here yields an issue in the 2d case: $\int_1^t E_g(s) ds$ grows logarithmically in $t$ (it is bounded if $d>2$).  As such, a simple proof mirroring the 1d proof closely can be established in dimensions $d\geq 3$ but will not be sharp when $d=2$.  Thus, we mainly focus below on the case $d=2$.  The key step here is to obtain and use a lower bound on the $R*g$ term when $|x|^2/t = \mathcal{O}(1)$.

\begin{proof}[Proof of \Cref{t.hd_moments}]
We begin with the proof of the upper bound on the moments in \Cref{t.hd_moments}.  We claim that there exists $A>0$ such that for all $t\geq 0$,
\be\label{e.hd_gaussian_bound}
	g(t,x)
		\leq \frac{A}{(t+1)^\frac{d}{2}} e^{- \frac{x^2}{A(t+1)}}.
\ee
Before establishing this, we note that the proof of the upper bound follows immediately via a direct computation using~\eqref{e.hd_gaussian_bound} (indeed, this is, up to a time change, equivalent to the fact that the $p$th moment of a Brownian motion is $\mathcal{O}( t^{p/2})$).


We now establish~\eqref{e.hd_gaussian_bound} via the construction of a supersolution.  Let
\[\begin{aligned}
	&\overline g(t,x) = \frac{A}{t+1} e^{- \frac{x^2}{2A(t+1)}}
		\qquad\qquad&\text{ if } d = 2,\\
	&\overline g(t,x) = \frac{A}{(t+1)^{d/2}} e^{\int_0^t E_g(s) \, ds - \frac{x^2}{2(t+1)}}
		\qquad &\text{ if } d\geq 3.
\end{aligned}
\]
Again, note that the choice of $\overline g$ for $d\geq 3$ would yield an extra logarithmic factor were we to use it in the case $d=2$ as $E_g(s) \approx (s+1)^{-1}$ and, hence, would not yield the sharp asymptotics.   In fact, that $\overline g$ is a supersolution when $d\geq 3$ is clear as it is simply a solution to the heat equation along with an integrating factor.  Thus, we focus our efforts on the case $d=2$.

Up to increasing $A$, we have that $\overline g > g$ at $t=0$ since $g_0$ is compactly supported and bounded.  We show that $\overline g \geq g$ on $(0,\infty)\times \R^2$ by contradiction, taking $t_0>0$ to be the first time that $\overline g$ and $g$ ``touch.''   Let $x_0$ be the point at which they touch.  It follows that, at $(t_0,x_0)$,
\be\label{e.c1081}
	\partial_t \overline g - \frac{1}{2} \Delta \overline g
		\leq \partial_t g - \frac{1}{2} \Delta g.
\ee
Our goal is to use~\eqref{e.c1081} to obtain a contradiction.

We claim that, at $(t_0,x_0)$,
\be\label{e.c1083}
	\partial_t \overline g - \frac{1}{2} \Delta \overline g - \overline g\left(E_g - R*g\right) > 0.
\ee
Postponing the proof of~\eqref{e.c1083} momentarily, we show how to conclude using it.  Indeed, at $(t_0,x_0)$, we have, from~\eqref{e.c1081},~\eqref{e.main}, and then~\eqref{e.c1083},
\[
	\partial_t \overline g - \frac{1}{2} \Delta \overline g
		\leq \partial_t g - \frac{1}{2} \Delta g
		= g \left( E_g - R*g\right)
		= \overline g \left( E_g - R*g\right)
		< \partial_t \overline g - \frac{1}{2} \Delta \overline g.
\]
This is clearly a contradicton.  Hence, \eqref{e.hd_gaussian_bound} follows from~\eqref{e.c1083}, which we prove now.

With arguments reminiscent of those in \Cref{s.R_continuous} (see, e.g., the proof of \Cref{l.small_time}), it is easy to check that, up to further increasing $A$, we may assume that $t_0 >1$.  Next, using arguments exactly as in \Cref{l.maxima_guw} (cf.~\eqref{e.c352}), there exists $C_R>0$, depending only on $R$, such that
\be\label{e.c1082}
	R*g(t_0,x_0)
		\geq \frac{A}{C_R(t_0+1)} e^{-\frac{2 x_0^2}{A(t_0+1)}}.
\ee
The above relies on the fact that $g(t_0,\cdot) \leq \overline g(t_0,\cdot) \leq (1+t_0)^{-1}$ and that
\[
	g(t_0,x_0) = \overline g(t_0,x_0)
		= \frac{A}{t_0+1} e^{- \frac{x_0^2}{A(t_0+1)}},
\]
which hold by contradictory assumption.

Notice that, for any $(t,x)$,
\[
	\partial_t \overline g
		- \frac{1}{2} \Delta \overline g
		= \overline g \left(
			\frac{x^2}{2A(t+1)^2} \left(1 - \frac{1}{A}\right)
			+ \frac{1}{t+1} \left(-1+\frac{1}{A} \right)
			\right).
\]
Hence, using \Cref{p.hd_decay} and~\eqref{e.c1082}, we find, at $(t_0,x_0)$,
\be\label{e.c765}
\begin{split}
	&\partial_t \overline g
		- \frac{1}{2} \Delta \overline g
		- \overline g \left(E_g - R*g\right)\\
		&\geq \overline g \left(
			\frac{x^2}{2A(t_0+1)^2} \left(1 - \frac{1}{A}\right)
			+ \frac{1}{t_0+1} \left(-1 + \frac{1}{A} \right)
			- \frac{C}{(t_0+1)}
			+ \frac{A e^{-\frac{2 x_0^2}{A(t_0+1)}}}{C_R(t_0+1)}
			\right).
\end{split}
\ee
where $C$ is the implied constant in \Cref{p.hd_decay}.

Increasing $A$ if necessary, we have
\be\label{e.c766}
	A
		> \max\{2, C_R(C+1)e^{8(C+1)}\}.
\ee
We consider first the case when $|x_0|^2 \geq 4(C+1)A (t_0+1)$.  Then~\eqref{e.c765} becomes
\[
	\partial_t \overline g
		- \frac{1}{2} \Delta \overline g
		- \overline g \left(E_g - \overline g\right)
		> \overline g \left(
			\frac{2(C+1)}{t_0+1}
			- \frac{1}{t_0+1}
			- \frac{C}{(t_0+1)}\right)
		\geq 0.
\]
On the other hand, if $|x_0|^2 \leq 4 (C+1) A (t_0+1)$, then~\eqref{e.c765} becomes
\[
	\partial_t \overline g
		- \frac{1}{2} \Delta \overline g
		- \overline g \left(E_g - \overline g\right)
		> \overline g \left(
			- \frac{1}{t_0+1}
			- \frac{C}{(t_0+1)}
			+ \frac{A}{C_R(t_0+1)} e^{- 8 (C+1)}\right)
		\geq 0.
\]
where the last line follows from~\eqref{e.c766}.  Thus,~\eqref{e.c1083} follows from the two cases above.  This finishes the proof of~\eqref{e.hd_gaussian_bound}.


The proof of the upper bound on $M_g$ follows from the fact that $g \leq \overline g$, while the upper bound on $E_g$ is the content of \Cref{p.hd_decay}.  The proof of the lower bounds on the moments and on the $L^2$- and $L^\infty$-norms of $g(t,\cdot)$ follows exactly as in the proof of \Cref{t.decay} using the bounds established above.  As such, we omit the details.  The proof is, thus, finished.
\end{proof}

\subsubsection{The upper bound on $g$}\label{s:hd_decay}

We now establish the key upper bound on $E_g$ on which the previous section depends.  We use classical methods based on the Nash inequality to order to establish the $t^{-d/2}$ decay of $E_g$; however, the Nash inequality must be slightly adapted to our macroscopic quantities $E_g$ and $D_g$.  We state this updated Nash inequality here, its proof is left until after the proof of \Cref{p.hd_decay}.

\begin{lemma}\label{l.Nash}
	Let $h\in H^1(\R^d) \cap L^1(\R^d)$ and suppose that $R$ satisfies~\eqref{e.R} or $R = \delta$.  Then
	\[
		E_h^{1 + 2/d}
			\lesssim \|h\|_{1}^{4/d} D_h.
	\]  
\end{lemma}

\begin{proof}[Proof of \Cref{p.hd_decay}]
Applying the convolved Nash inequality from \Cref{l.Nash}, we find that 
\[
	E_g^{1+2/d}
		\lesssim \|g\|_{1}^{4/d} D_g
		= D_g.
\]
Using this inequality in~\eqref{e.differential_inequality}, we have
%
\be\label{e.c762}
	\dot E_g
		= - D_g - 2\int g(R*g- E_g)^2\dx \leq -D_g
		\lesssim - E_g^{1 + 2/d}.
\ee
Integrating this in time, we find
\be\label{e.E_small}
	E_g(t)
		\lesssim \left( t + E_g(0)^{-2/d}\right)^{-d/2},
\ee
which concludes the proof.
\end{proof}

We now prove \Cref{l.Nash}.  The proof given is almost exactly as in the classical case; however, the new $\hat R$ terms in the Fourier transform must be addressed.

\begin{proof}[Proof of \Cref{l.Nash}]
	We note that the case $R=\delta$ is the standard Nash inequality; hence we omit its proof and focus only on the case when $R$ is continuous and satisfies~\eqref{e.R}.  We use the Fourier transform: for any function $h$, we denote its Fourier transform by $\hat h$.  We begin with Plancherel's identity
	\[
		E_h
			= \langle h, R*h \rangle
			= \int \hat h \overline{\widehat{R*h}}\, d\xi
			= \int |\hat h|^2 \overline{\hat R }\, d\xi.
	\]
	Notice that, due to the form of $R$ in terms of $\phi$ (see~\eqref{e.R}), we have $\hat{R}=\hat{\phi}^2$ and since $\phi$ is assumed to be even, it actually implies $\hat{R}=|\hat{\phi}|^2$. 
	In other words, $\hat R$ is real valued and non-negative (this is not surprising, because $R$ comes from the covariance function of a Gaussian process so it is positive definite).
	
	Hence, for $L$ to be chosen, we have
	\[\begin{split}
		E_h
			&\leq \int_{B_L} |\hat h|^2 \hat R\, d\xi
				+ \int_{B_L^c} \frac{|\xi|^2}{L^2} |\hat h|^2 \hat R\, d\xi
			\lesssim L^d \|\hat h\|_{\infty}^2 \|\hat R\|_{\infty} + \int_{B_L^c} \frac{|\xi|^2}{L^2} |\hat h|^2 \hat R\, d\xi
			\\&\leq L^d \|h\|_{1}^2 + \frac{1}{L^2} \langle \nabla h, R*\nabla h\rangle
			= L^d \|h\|_1^2 + \frac{1}{L^2} D_h.
	\end{split}\]
	The proof is then finished by choosing $L^{d+2} = D_h / \|h\|_{1}^2$.
\end{proof}

\subsection{Gaussian and non-Gaussian self-similar dynamics}

The goal of this section is to prove Theorem~\ref{t.hd_steady}. Recall that we assumed the convolution kernel $R=\delta$. In order to attack this problem, we begin with a few transformations of the function $g$.   We use self-similar variables here; that is, 
we define
\be\label{e:change_variables}
	G(\tau,y)
		= e^{\frac{d}{2}\tau} g\left(e^\tau, e^{\tau/2} y\right)
\ee
and find that
\be\label{e:G}
	G_\tau
		= \frac{d}{2} G
			+ \frac12\Delta G + G e^{-\frac{d-2}{2}\tau}
				\left(
					\|G\|_2^2
					- G
				\right)
			+ \frac{y}{2}\cdot\nabla G.
\ee
First, we note that, in the above equation, the linear operator $\frac12\Delta+\frac{y}{2}\cdot\nabla+\frac{d}{2}$ is actually the adjoint of $\frac12\Delta-\frac{y}{2}\cdot\nabla$, which is the generator of an 
Ornstein–Uhlenbeck process with the standard Gaussian invariant density $(2\pi)^{-d/2}\exp(-|y|^2/2)$. Thus, without the nonlinear term, the above equation is actually the Fokker-Planck equation for the Ornstein-Uhlenbeck process which converges to its invariant density. We now see the reason for the different behavior in $d=2$: when $d\geq3$, the nonlinear terms are lower order terms decaying exponentially in $\tau$, hence we get a Gaussian behavior as expected, while when $d=2$, the nonlinear terms are $O(1)$ since, in these variables, \Cref{t.hd_moments} yields
\be\label{e:decay_ss}
	\|G\|^2_2,
		\|G\|_\infty
		\approx 1.
\ee

\subsubsection{Decay to a Gaussian in higher dimensions $d\geq 3$}

We now use the above change of variables to obtain the convergence to a Gaussian; that is, we prove \Cref{t.hd_steady} (i).  First, we make a few reductions.  Up to shifting in time, we may assume that, for $\tau_0 \geq 0$, $g(e^{\tau_0}, x) = g_0(x)$, which yields
\be\label{e.c771}
	G(\tau_0, y)
		= g_0( e^{\tau_0/2} y)
		\leq e^{\frac{d\tau_0}{2}}A e^{- \frac{e^{\tau_0} y^2}{B}}
				\qquad\text{ for all } y \in \R^d.
\ee
Hence, up to increasing $\tau_0$ and increasing $A$, we may assume that
\[
	G(\tau_0,y) \leq A e^{- \frac{y^2}{2}}
		\qquad\text{ for all } y \in \R^d.
\]

Summing up the previous reductions, we assume that $G$ solves
\be\label{e.mainG}
	\begin{cases}
		G_\tau
			=  \frac{1}{2}\Delta G
			+ \frac{y}{2}\cdot\nabla G
			+ \frac{d}{2} G
			+ e^{-\frac{d-2}{2}\tau} G
				\left(
					\|G\|_2^2
					- G
				\right)
				\qquad &\text{ in } (\tau_0,\infty)\times \R^d,\\
		G = G_0 \leq A e^{-\frac{y^2}{2}}
				\qquad &\text{ on } \{\tau_0\}\times \R^d.
	\end{cases}
\ee

The above is not self-adjoint and, thus, not amenable to spectral analysis.  Hence, we define a new function
\[
	W(\tau,y)
		= \frac{G(\tau,y)}{\psi_0(y)},
		\qquad\text{where } \psi_0(y) = \frac{1}{Z} e^{-\frac{y^2}{4}}
\]
with $Z = (2\pi)^{d/4}$ is a normalization constant chosen so that $\|\psi_0\|_2 = 1$.  It is clear that $\psi_0^2$ is the standard Gaussian density. We note that, due to the bound on $G_0$ in~\eqref{e.mainG}, we have that $W_0 := G_0 / \psi_0 \in L^2$.

We notice that
\be\label{e:W}
	W_\tau
		= \frac{1}{2} \Delta W + W\left(\frac{d}{4} - \frac{|y|^2}{8}\right)
			+ W e^{-\frac{d-2}{2}\tau}\left( \|G\|_2^2 - G\right)
\ee
For simplicity, we write the linear operator
\[
	M = - \frac{1}{2} \Delta + \left(\frac{|y|^2}{8} - \frac{d}{4}\right).
\]

We understand the behavior of $G$ through the properties of $M$.  First, note that $M$ is an unbounded, symmetric operator on $L^2$.  For each multi-index $\alpha \in \N_0^d$, let
\[
	\psi_\alpha
		= \psi_0^{-1} \partial_{x_1}^{\alpha_1}\cdots \partial_{x_n}^{\alpha_n} \psi_0^2
		= H_\alpha(y) \psi_0(y),
\]
where $H_\alpha$ is the $\alpha$ Hermite polynomial that is implicitly defined above.  It is easy to check that
\[
	M \psi_\alpha
		= \frac{|\alpha|}{2} \psi_\alpha.
\]
In addition, it is well-known\footnote{This is usually stated in the following way: the set of (rescaled) Hermite polynomials $H_\alpha$ form a basis of the weighted space $L^2(\psi_0^2)$.  This is, however, equivalent to our statement.} that $\psi_\alpha$ form an orthogonal basis of $L^2$.  In particular, we conclude that if $\langle \psi, \psi_0 \rangle = 0$, then
\be\label{e.772}
	\langle M \psi, \psi\rangle
		\geq \frac{1}{2} \|\psi\|_2^2.
\ee
We are now in a position to complete the proof of \Cref{t.hd_steady} (i).


\begin{proof}[Proof of \Cref{t.hd_steady} (i)]
We write
\[
	W_\perp(\tau,y)
		= W(\tau,y) - \psi_0(y).
\]
Notice that 
\[
	\langle W_\perp(\tau), \psi_0\rangle
		= \langle W(\tau), \psi_0\rangle - \langle \psi_0,\psi_0\rangle
		= \int G(\tau,y) \dy - 1
		= 0.
\]
The proof proceeds by showing that $W_\perp \to 0$ using this orthogonality.

%
%
Multiplying~\eqref{e:W} by $W_\perp$, integrating, and noticing that $\langle W, W_\perp\rangle = \|W_\perp\|_2^2$ by orthogonality, yields
\[
	\frac{1}{2}
		\partial_\tau \|W_\perp\|_2^2
		= - \langle MW , W_\perp\rangle
			+ e^{-\frac{d-2}{2} \tau } \langle W \left(\|G\|_2^2 -G\right), W_\perp\rangle.
\]
Next, using that $MW = M\psi_0 + M W_\perp = MW_\perp$ and~\eqref{e.772}, yields
\be\label{e.c773}
\begin{split}
	\frac{1}{2}
		\partial_\tau \|W_\perp\|_2^2
		&= - \langle MW_\perp , W_\perp\rangle
			+ e^{-\frac{d-2}{2} \tau } \langle W \left(\|G\|_2^2 -G\right), W_\perp\rangle\\
		&\leq - \frac{1}{2} \|W_\perp\|_2^2
			+ e^{-\frac{d-2}{2} \tau } \langle W \left(\|G\|_2^2 -G\right), W_\perp\rangle.
\end{split}
\ee

Using the bounds in \Cref{t.hd_moments}, we find
\[
	\langle W(\|G\|_2^2 - G), W_\perp\rangle
		\lesssim \|W\|_2 \|W_\perp\|_2
		\leq \left(1 + \|W_\perp\|_2\right) \|W_\perp\|_2.
\]
Hence,~\eqref{e.c773} becomes, for some $C>0$,
\[
	\partial_\tau \|W_\perp\|_2^2
	 	+ \left(1 - C e^{-\frac{d-2}{2}\tau}\right) \|W_\perp\|_2^2
		\lesssim e^{-\frac{(d-2)}{2}\tau} \|W_\perp\|_2.
\]
Solving this differential inequality yields
\[
	\|W_\perp\|_2
		\lesssim
		R_d(\tau)
\]
where we define 
\[
	R_d(\tau)
		= \begin{cases}
				(\tau+1) e^{- \tau/2}
					\qquad &\text{ if } d = 3,\\
				e^{-\tau/2}
					\qquad &\text{ if } d\geq 4.
			\end{cases}
\]

Returning to $G$, we find
\[
	\int e^\frac{y^2}{2} \left( G(\tau,y) - \psi_0^2\right)^2 dy \lesssim R_d(\tau)^2.
\]
Using parabolic regularity theory, it is standard to conclude, for any $\sigma < 1/4$,
\[
	\|e^{\sigma y^2} \left(G(\tau) - \psi_0^2\right)\|_\infty
		\lesssim R_d(\tau),
\]
which, after returning to the original variables, concludes the proof.
%
%
%
%
\end{proof}

\subsubsection{A non-Gaussian steady state in two dimensions}

In the following, we use $R$ as the variable for the radius of a ball, which is not to be confused with the convolution kernel (the kernel is fixed to be $\delta$ in this section). We construct a steady solution of the self-similar problem~\eqref{e.mainG} when $d=2$ that is not the Gaussian $\psi_0^2$ from the previous subsection.  The construction occurs in multiple steps.  First, we replace the $\|G\|_2^2$ in~\eqref{e:G} with a constant term $E$ to ``localize'' the equation.  For any $E>0$ and $R\gtrsim 1/\sqrt E$, we construct radial (rotationally symmetric) steady solutions $G_R$ of the localized equation on $B_R$ with Dirichlet boundary conditions on $\partial B_R$.  Second, we show that, choosing $E=\mathcal{E}_R$ well guarantees that $\int_{B_R} G_R(y) dy = 1$.  Finally, we show that, in the limit $R\to\infty$, $\mathcal{E}_R - \|G_R\|_2^2\to 0$ and that $G_R$ converges to a steady solution $G$ of~\eqref{e:G} on $\R^2$.

\subsubsection*{Constructing a steady solution of the localized problem on a ball}

\begin{lemma}\label{l.underline_G}
Fix $E>0$ and $R \geq \max\{4,20/\sqrt E\}$.  There exists a radial function $\underline G_{E,R}: B_R \to [0,E/2]$ such that
\be\label{e.G_bounded_sub}
	\begin{cases}
		\frac{1}{2}\Delta \underline G_{E,R} + \frac{y}{2} \cdot \nabla \underline G_{E,R} + \underline G_{E,R}(1 + E - \underline G_{E,R})
			\geq 0
			\qquad &\text{ in } B_R,\\
		\underline G_{E,R} = 0
			\qquad &\text{ on } \partial B_R,
	\end{cases}
\ee
and
\[
	\int_{B_R} \underline G_{E,R} dy
		\gtrsim E.
\]
\end{lemma}

\begin{proof}
Let $\phi:[0,R]\to \R$ be a $C^2$ cut-off function such that
\be\label{e.cutoff}
	\begin{split}
		&(i)
			\ \ 0 \leq \phi \leq 1,\ \phi'(R) = \phi(R) = 0,
		\qquad 
		(ii) \ \ 
			\phi \equiv 1 \text{ on } \left[0, \frac{R}{3}\right],\\
		&(iii)\ \
			-\frac{20}{R^2} \leq \phi'',\
				- \frac{10}{R} \leq \phi',
				\text{ and } \phi(y) \geq 1/2
				\text{ on } \left[\frac{R}{3}, \frac{2R}{3}\right], \text{ and}\\
		&(iv)\ \
			\phi'' \geq \frac{2}{R^2},\
			\phi'(r) \geq - \frac{10}{R^2}(R-r)
				\text{ on } \left[\frac{2R}{3},R\right].
	\end{split}
\ee
Let 
\be\label{e.underline_G}
	\underline G_{E,R}(y)
		= \frac{E}{2} e^{-\frac{y^2}{2}} \phi(|y|).
\ee
Notice that, by~\eqref{e.cutoff}.(i), $\underline G_{E,R} \leq E/2$.  Using polar coordinates, we find 
\be\label{e.c8117}
\begin{split}
	-\frac{1}{2}\Delta \underline G_{E,R} - &\frac{y}{2} \cdot \nabla \underline G_{E,R} - \underline G_{E,R}(1 + E - \underline G_{E,R})\\
		&= -\frac{E e^{-\frac{y^2}{2}}}{2} \left(
			\frac{1}{2} \phi''
			+ \frac{1}{2r} \phi'
			+ \left(E - \frac{E}{2} e^{-\frac{r^2}{2}}\phi\right) \phi
		\right).
\end{split}
\ee

It is clear that the last line is non-positive when $r \in [0,R/3]$ since $\phi \equiv 1$ on this set.  When $r \in [R/3,2R/3]$, we deduce, from~\eqref{e.cutoff}.(iii), that
\[
	-\left(\frac{1}{2} \phi''
			+ \frac{1}{2r} \phi'
			+ \left(E - \frac{E}{2} e^{-\frac{r^2}{2}}\phi\right) \phi\right)
		\leq \frac{10}{R^2} + \frac{15}{R^2} - \frac{E}{2}
		= \frac{E}{2R^2} \left( \frac{50}{E} - R^2\right).
\]
Since, $R \geq \max\{4, 20/\sqrt{E}\}$, this is non-positive.  Hence, the right hand side of~\eqref{e.c8117} is non-positive.

Finally consider the case when $r \in [2R/3,R]$.  First notice that, due to~\eqref{e.cutoff}.(i) and (iv), we have
\[
	\phi(y) \geq \frac{1}{R^2}(R-y)^2.
\]
Using this lower bound, as well as~\eqref{e.cutoff}.(iv) again, yields
\be\label{e.c8118}
	-\left(\frac{1}{2} \phi''
		+ \frac{1}{2r} \phi'
		+ \left(E - \frac{E}{2} e^{-\frac{r^2}{2}}\phi\right) \phi\right)
		\leq - \frac{1}{R^2}
			+ \frac{3}{4R} \frac{10}{R^2} (R-y)
			- \frac{E}{2} \frac{(R-y)^2}{R^2}.
\ee
We now use Young's inequality and then the fact that $R \geq 20/\sqrt E$ to find
\[
	\frac{3}{4R} \frac{10}{R^2} (R-y)
		\leq \frac{225}{8E R^4} + \frac{E}{2} \frac{(R-y)^2}{R^2}
		\leq \frac{225}{3200R^2} + \frac{E}{2} \frac{(R-y)^2}{R^2}.
\]
Plugging this into~\eqref{e.c8118} implies that the right hand side of~\eqref{e.c8117} is non-positive on $[2R/3,R]$.


Hence, in all cases, the right hand side of~\eqref{e.c8117} is non-positive, which implies that $\underline G_{E,R}$ is a subsolution; that is, it satisfies~\eqref{e.G_bounded_sub}, as claimed.  In addition, the lower bound on the integral of $\underline G_{E,R}$ is clear by~\eqref{e.cutoff} and~\eqref{e.underline_G}.
\end{proof}

We now use $\underline G_{E,R}$ to construct a radial solution to the local problem on $B_R$.
\begin{proposition}\label{p.solution_bounded}
	Suppose that $E>0$ and $R \geq \max\{4, 20/\sqrt E\}$.  There exists a radial function $G_{E,R}: B_R \to [0, 1+E)$ of
	\be\label{e.G_bounded}
	\begin{cases}
		0 = \frac{1}{2}\Delta G_{E,R} + \frac{y}{2} \cdot \nabla G_{E,R} + (1 + E - G_{E,R})G_{E,R}
			\qquad &\text{ in } B_R,\\
		G_{E,R} = 0
			\qquad &\text{ on } \partial B_R,
	\end{cases}
\ee
such that $\int_{B_R} G_{E,R} dy \gtrsim E$.  This is the unique nontrivial solution of~\eqref{e.G_bounded}.
\end{proposition}
\begin{proof}
	Let $H$ be the solution of
	\be\label{e.H}
	\begin{cases}
		H_t = \frac{1}{2}\Delta H + \frac{y}{2} \cdot \nabla H + (1 + E - H)H
			\qquad &\text{ in } (0,\infty) \times B_R,\\
		H = \underline G_{E,R}
			\qquad &\text{ on } \{0\}\times B_R,\\
		H = 0
			\qquad &\text{ on } [0,\infty)\times \partial B_R,
	\end{cases}
	\ee
	where $\underline G_{E,R}$ is from \Cref{l.underline_G}.  The comparison principle immediately yields that $H \leq 1+E$.
	
	We claim that $H_t \geq 0$ for all $t>0$.  Since $\underline G_{E,R}$ satisfies~\eqref{e.G_bounded_sub}, then $H_t(0,\cdot) \geq 0$.  In addition, differentiating~\eqref{e.H} in time yields a parabolic equation for $H_t$ that enjoys the comparison principle and of which $0$ is a solution.  We conclude that $\min_{t,y} H_t(t,y) \geq 0$ by applying the comparison principle to $H_t$ and $0$.
	
	Since, for all $y$, $H(t,y)$ is increasing in $t$, there exists $G_{E,R}(y)$ such that $H(t,y) \to G_{E,R}(y)$ as $t\to\infty$.  In addition, we have that $H(t,\cdot) \leq G_{E,R}$ for all $t$.  Finally we point out that $G_{E,R} \leq 1 + E$ since $H \leq 1 + E$.
	
	We also note that, by parabolic regularity theory, for any $\alpha\in(0,1)$, there exists $C>0$, depending only on $\alpha$ and $E$, such that
	\be\label{e.c6191}
		\|H \|_{C^{2+\alpha}_{\rm parabolic}([0,\infty) \times B_R)}
			\leq C,
	\ee
	where $C^{2 + \alpha}_{\rm parabolic}$ is the standard parabolic H\"older space.
	
	We claim that $\|\partial_t H(t_n)\|_{L^2(B_R)} \to 0$ along some subsequence $t_n \to \infty$.  If not, then there exists $\delta>0$ and $t_0>0$ such that, for all $t\geq t_0$, $\|\partial_t H(t)\|_2 \geq \delta$.  Using~\eqref{e.c6191} and the nonnegativity of $\partial_t H$, we find, for all $t\geq t_0$,
	\[
		\delta^2
			\leq \int_{B_R} |\partial_t H(t,y)|^2 dy
			\leq C \int_{B_R} \partial_t H(t,y) dy.
	\]
	Integrating this and using that $0 \leq H \leq G_{E,R} \leq 1+E$, we find, for any $T>0$,
	\[\begin{split}
		\frac{\delta^2 T}{C}
			&\leq \int_{t_0}^{t_0+T} \int_{B_R} \partial_t H(t,y) dy dt
			= \int_{t_0}^{t_0+T}  \partial_t \int_{B_R} H(t,y) dy dt\\
			&= \int_{B_R} H(t_0+T,y) dy - \int_{B_R} H(t_0,y) dy
			\leq \int_{B_R} G_{E,R}(y) dy
			\leq \pi R^2 (1+E).
	\end{split}\]
	Taking $T\to\infty$ yields a contradiction.  Hence, there exists a sequence $t_n\to\infty$ as $n\to\infty$ such that $\|\partial_t H(t_n)\|_2 \to 0$.
	
	Up to taking a subsequence, the bounds in~\eqref{e.c6191} and the compactness of $C^{2,\alpha}_{\rm parabolic}$ in $C^2_{\rm parabolic}$, imply that $H(t_n) \to G_{E,R}$ in $C^2_{\rm parabolic}$.  In addition, the resulting convergence of $\partial_t H(t_n)$ in $L^\infty$ and its convergence to zero in $L^2$ implies that $\partial_t H(t_n) \to 0$ in $L^\infty$.  We conclude that
	\[
		0
			= \frac{1}{2}\Delta G_{E,R} + \frac{y}{2} \cdot\nabla G_{E,R}
				+ \left(1 + E - G_{E,R}\right) G_{E,R}.
	\]
	
	To conclude the proof, we check the various properties of $G_{E,R}$.  First, the nonnegativity of $G_{E,R}$ follows from the fact that $H$ is increasing in time and $H(0,\cdot) = \underline G_{E,R} \geq 0$.
	
	Second, parabolic regularity theory implies that $H$ is uniformly (in time) small near $\partial B_R$, which implies that $G_{E,R} = 0$ on $\partial B_R$.
	
	Third, recall the earlier observation that $G_{E,R} \leq 1+E$.  The strong maximum principle applied to~\eqref{e.G_bounded} implies that this inequality is strict.
	
	Fourth, the lower bound on the integral of $G_{E,R}$ follows from the lower bound on $\underline G_{E,R} = H(0,\cdot)$ and the fact that $H$ is increasing.
	
	Finally, $H$ is radial at $t=0$ by construction of $\underline G_{E,R}$.  Let $M$ be any rotation matrix and define $H_M(t,y) = H(t,My)$.  It is easy to see that $H_M$ solves~\eqref{e.H}.  Thus, by the uniqueness of solutions of parabolic equations, we find $H_M = H$.  We conclude that $H(t,\cdot)$ is radial for all $t$, from which it follows that $G_{E,R}$ is radial.
	
	The last step is to check the uniqueness of $G_{E,R}$.  We drop the $E$ and $R$ subscripts for ease.  Suppose that $H$ is another nontrivial solution~\eqref{e.G_bounded}.  Define $G_A(y) = A G(x)$.  It is easy to verify that
	\[
		0
			= \frac{1}{2}\Delta G_A + \frac{y}{2} \cdot\nabla G_A
				+\left(1 + E - \frac{1}{A} G_A\right) G_A.
	\]
	The Hopf maximum principle implies that the outward point normal derivative of $G$ is negative on $\partial B_R$.  Hence, we have that $G_A > H$ for all $A$ sufficiently large.  We define
	\[
		A_0 = \inf \{A \geq 1 : G_A \geq H\}.
	\]
	If $A_0 =1$, we conclude that $G \geq H$, which is our goal.  If not, 
	let $\Psi(y) = G_{A_0}(y) - H(y)$ for all $y$.  We then find that
	\[
		-H\left( H - \frac{G_{A_0}}{A_0}\right)
			= \frac{1}{2} \Delta \Psi + \frac{y}{2} \cdot \nabla \Psi
				+ \left(1 + E - \frac{1}{A_0} G_{A_0}\right) \Psi.
	\]
	
	By the choice of $A_0$, we either have that there exists $y_0 \in B_R$ such that $\Psi(y_0) = 0$ or there exists $y_0 \in \partial B_R(0)$ such that $y_0 \cdot \nabla G_{A_0}(y_0) = y_0 \cdot \nabla H(y_0)$.
	
	Consider the first case.  Let $\Sigma$ be the maximal open connected component of $B_R \cap \{G_{A_0} < A_0 H\}$ containing $y_0$.  Then $\Psi > 0$ on $\partial \Sigma$ and $\Psi$ satisfies
	\[
		0 \geq \frac{1}{2} \Delta \Psi + \frac{y}{2} \cdot \nabla \Psi
				+ \left(1 + E - \frac{1}{A} G_{A_0}\right) \Psi
					\qquad \text{ on } \Sigma.
	\]
	The strong maximum principle implies that $\Psi>0$ on $\Sigma$, which contradicts the fact that $\Psi(y_0)=0$.
	
	The second case proceeds similarly, except using the Hopf lemma to conclude that $y_0 \cdot \nabla \Psi(y_0) < 0$ to obtain a contradiction.  We omit the details.
	
%
%
	
	Above we showed that $G \leq H$.  The argument used to establish this used nothing about $G$ from its construction; we only used that it is nontrivial.  Thus, an identical argument implies that $H \leq G$ as well, which yields $G=H$.  Hence, non-trivial solutions of~\eqref{e.G_bounded} are unique.  This concludes the proof.
\end{proof}

We show that $G_{E,R}$ decays exponentially away from the origin independently of $R$.  We require this to show that $G$, the limiting object, is exponentially decaying as stated in \Cref{t.hd_steady}.(ii) and in order to establish a relationship between the sizes of $E$ and $\int G_{E,R} \dy$ in the sequel.

\begin{lemma}\label{l.G_ER_upper_bound}
For any $E$ and $R$ as in \Cref{p.solution_bounded},
\[
	G_{E,R}(y)
		\leq (1+E) e^{4(1+E) - \frac{y^2}{4}}
			\qquad\text{ for all } |y| \geq 4 \sqrt{1+E}.
\]
\end{lemma}
\begin{proof}
For any $A\geq 0$, let $\psi_A(y) = A e^{-\frac{y^2}{4}}$.  We first claim that, for $|y| \geq 4\sqrt{1+E}$,
\be\label{e.c6193}
	-\frac{1}{2}\Delta \psi
		- \frac{y}{2} \cdot \nabla \psi
		- \left(1 + E\right) \psi
		> 0.
\ee
To this end, we compute:
\[
\begin{split}
	-\frac{1}{2}\Delta \psi_A
		- \frac{y}{2} \cdot\nabla \psi_A
		- (1+E) \psi_A
		&= - \frac{1}{2}\left( \frac{y^2}{4} \psi_A -  \psi_A \right)
		- \frac{y}{2} \cdot \left( - \frac{y}{2}\psi_A\right)
		- (1+E) \psi_A \\
		&= \frac{1}{8}\psi_A \left( y^2 - 2(2 + 4E) \right).
\end{split}
\]
Hence, if $y^2 \geq 16(1+E)$, then we conclude~\eqref{e.c6193}.

If $A \geq (1+E) e^{R^2/4}$, then $\psi_A \geq 1+E > G_{E,R}$  (recall the upper bound on $G_{E,R}$ from \Cref{p.solution_bounded}).  Thus, let
\[
	A_0 = \inf\left\{A > 0 : \psi_A \geq G_{E,R} \mbox{ on } |y|\in[4\sqrt{1+E}, R]\right\}
\]
is well-defined.  We claim that $A_0 \leq \left(1+E\right) e^{4(1+E)}$.  We argue by contradiction assuming that $A_0 > (1+E) e^{4(1+E)}$.

By continuity, there exists $y_0$ such that $|y_0| \in [4\sqrt{1+E}, R]$ such that $\psi_{A_0}(y_0) = G_{E,R}(y_0)$.  Since
\[
	\psi_{A_0}(4\sqrt{1+E}) = A_0e^{-4(1+E)}
		> (1+E) > G_{E,R}
	\quad\text{ and }\quad
	\psi_{A_0}(R) > 0 = G_{E,R}(R),
\]
it must be that $|y_0| \in (4 \sqrt{1+E}, R)$.  In addition, by construction, $y_0$ is the location of a minimum of zero of $\psi_{A_0} - G_{E,R}$.  Hence $\Delta (\psi_{A_0} - G_{E,R}) \geq 0$, $\nabla (\psi_{A_0} - G_{E,R}) = 0$, and $\psi_{A_0}(y_0) = G_{E,R}(y_0)$.  Hence, at $y_0$,
\[\begin{split}
	0 &\geq - \frac{1}{2}\Delta \left( \psi_{A_0} - G_{E,R}\right)\\
		&> \left(\frac{y}{2} \cdot \nabla \psi_{A_0}
			+ (1+E)\psi_{A_0}\right)
			- \left(\frac{y}{2} \cdot \nabla G_{E,R}
			+ (1+E - G_{E,R})G_{E,R}\right)\\
		&= G_{E,R}^2 > 0.
\end{split}\]
Here we used~\eqref{e.G_bounded} and~\eqref{e.c6193} to obtain the second inequality.  This is clearly a contradiction, so we conclude that $A_0 \leq (1+E) e^{4(1+E)}$.  It follows that, for all $|y|\in [4\sqrt{1+E}, R]$,
\[
	G_{E,R}(y)
		\leq (1+E) e^{4(1+E)} e^{-\frac{y^2}{4}},
\]
which concludes the proof.
\end{proof}

Next we show that there exists $E$ such that $G_{E,R}$ has $L^1$-norm one.

\begin{lemma}\label{l.mass_one}
	There exists $\mathcal{E}_R\footnote{We make a slight change in convention here using the italicized ``E'' in order to avoid clashing notation with $E_f$ for the squared $L^2$-norm of a function $f$.} \approx 1$, depending only on $R$, such that 
	\[
		\int_{B_R} G_{\mathcal{E}_R,R}(y) dy = 1.
	\]
\end{lemma}
\begin{proof}
	To establish this we use the continuity of solutions of elliptic equations with respect to their coefficients.  With this in mind, we need only find $E_1$ and $E_2$ such that
	\[
		\int_{B_R} G_{E_1, R}(y) dy \leq 1
			\qquad\text{ and }\qquad
		\int_{B_R} G_{E_2,R}(y) \geq 1.
	\]
	The second inequality follows from \Cref{p.solution_bounded}, after taking $E_2\gtrsim 1$.
	
	To find the first inequality, integrate~\eqref{e.G_bounded} over $B_R$ to find
	\be\label{e.c6201}
		-\frac{1}{2}\int_{\partial B_R} \frac{y}{|y|} \cdot \nabla G_{E_1,R} dy
			= E_1 \int_{B_R} G_{E_1,R} dy
				- \int_{B_R} G_{E_1,R}^2 dy.
	\ee
	As $G$ is positive in $B_R$ and zero on $\partial B_R$, we find that the left hand side is nonnegative.  Hence,
	\[
		\int_{B_R} G_{E_1,R}^2 dy
			\leq E \int_{B_R} G_{E_1,R} dy
	\]
	
	We wish to estimate the integral on the right hand side above.  Let $L>4 \sqrt{1+E_1}$ be a constant chosen in the sequel.  Then, using H\"older's inequality and \Cref{l.G_ER_upper_bound}, we obtain
	\[\begin{split}
		\int_{B_R} G_{E_1,R} dy
			&\leq \int_{B_L} G_{E_1,R} dy + \int_{B_R\setminus B_L} (1+E_1)e^{4(1+E_1) - \frac{y^2}{4}} dy\\
			&\lesssim L \left(\int_{B_L} G_{E_1,R}^2 dy\right)^{1/2}
				+ L^2(1+E_1) e^{4(1+E_1) -\frac{L^2}{4}}\\
			&\leq L \left(\int_{B_R} G_{E_1,R}^2 dy\right)^{1/2}
				+ L^2(1+E_1) e^{4(1+E_1) -\frac{L^2}{4}}\\
			&\leq  L \left( E_1 \int_{B_R} G_{E_1,R} dy\right)^{1/2}
				+ L^2(1+E_1) e^{4(1+E_1) -\frac{L^2}{4}}.
	\end{split}\]
	Choosing $L$ sufficiently large and then $E_1$ sufficiently small independent of $R$, we find
	\[
		\int_{B_R} G_{E_1,R} dy
			< 1.
	\]
	This concludes the proof.
\end{proof}

Using \Cref{l.mass_one}, we let $G_R = G_{\mathcal{E}_R,R}$ denote the solution of~\eqref{e.G_bounded} with mass one.  We now show that $\mathcal{E}_R \approx \|G_R\|_2^2$ as $R\to\infty$.

\begin{lemma}\label{l.E_R}
As $R$ tends to $\infty$, $|\mathcal{E}_R - \|G_R\|_2^2| \to 0.$
\end{lemma}
\begin{proof}
	Using~\eqref{e.c6201}, \Cref{l.mass_one}, and the fact that $G_R$ is radial, it is sufficient to show that
	\be\label{e.c6202}
		- \frac{1}{2} \int_{\partial B_R} \frac{y}{|y|} \cdot \nabla G_R(y)dy
			= - \pi R (G_R)_r(R) \to 0
	\ee
	Let $W = e^{y^2/4} G_R$ and, as in~\eqref{e:W},
	\[
		0
			= \frac{1}{2}\Delta W + W\left( \frac{1}{2} - \frac{y^2}{4} + \mathcal{E}_R - G_R\right).
	\]
	Abusing notation and changing to polar coordinates, we find
	\[
		0
			= \frac{1}{2}W_{rr} + \frac{1}{2r} W_r
				+ W \left( \frac{1}{2} - \frac{r^2}{4} + \mathcal{E}_R - G_R\right)
	\]
	We note that the reason for the change to working with $W$ is to work with an equation whose first order term is bounded uniformly regardless of $R$.  Hence, we can apply $L^2$ estimates for solutions of elliptic equations up to the boundary (see, e.g., \cite[Theorem 8.12]{GilbargTrudinger}), we find
	\[
		\|W\|_{H^2([R-1,R])}
			\lesssim \left\|\left(\frac{1}{2} - \frac{r^2}{4} + \mathcal{E}_R - G_R\right)W\right\|_{L^2([R-2,R])} + \|W\|_{L^2([R-2,R])}.
	\]
	Using \Cref{l.G_ER_upper_bound} to bound $G_R$ and, thus, $W$ from above and \Cref{l.mass_one} to bound $\mathcal{E}_R$ from above, we see that $W$ is uniformly bounded from above on $[R-2,R]$ as long as $R$ is sufficiently large.  We deduce that
	\[
		\|W\|_{H^2([R-1,R])}
			\lesssim R^2.
	\]
	Using the Sobolev embedding theorem and the relationship between $G_R$ and $W$, we find
	\[\begin{split}
		\|G_R\|_{C^1([R-1,R])}
			&\lesssim R e^{-\frac{(R-1)^2}{4}} \|W\|_{C^1([R-1,R])}\\
			&\lesssim R e^{-\frac{(R-1)^2}{4}} \|W\|_{H^2([R-1,R])}
			\lesssim R^3 e^{-\frac{(R-1)^2}{4}}.
	\end{split}\]
	This establishes~\eqref{e.c6202}, which concludes the proof.
\end{proof}

We now finish the construction of the steady state $G$. 

\begin{proof}[Proof of \Cref{t.hd_steady}.(ii)]
From \Cref{p.solution_bounded}, \Cref{l.G_ER_upper_bound}, and \Cref{l.mass_one} we have that
\[
	G_R \lesssim e^{-\frac{y^2}{4}}.
\]
Using a similar argument as we did in the conclusion of \Cref{l.E_R} along with the Schauder estimates for elliptic equations, we find that, for any $\alpha \in (0,1)$,
\[
	\| e^{y^2/5} G_R\|_{C^{2,\alpha}}
		\lesssim 1.
\]

We thus find a subsequence $R_n\to\infty$ as $n\to\infty$, and $G \in C^{2,\alpha}$ such that $G_{R_n} \to G$ uniformly in $C^2$, and, due to the decay in $y$, in $L^1$ and $L^2$ as well.  We conclude that $\int G \dy = 1$ and $\|G_{R_n}\|_2^2 \to \|G\|_2^2$.  From \Cref{l.E_R}, we further have that $E_{R_n} \to \|G\|_2^2$.

Using all conclusions from the above, we find that $G$ is a radial function satisfying
\[
	0 = \Delta G + \frac{y}{2} \cdot \nabla G + G(1 + \|G\|_2^2 - G),
\]
which concludes the proof.
\end{proof}

\subsubsection*{Acknowledgement}

 YG was partially supported by the NSF through DMS-2203014. CH was partially supported by NSF grants DMS-2003110 and DMS-2204615.

\end{document}